\documentclass[11pt]{amsart}
\usepackage{pb-diagram,amssymb,epic,eepic,verbatim,graphicx}
\usepackage{graphics,epsfig,psfrag}

\newtheorem{theorem}{Theorem}[subsection]
\newtheorem{corollary}[theorem]{Corollary}

\newtheorem{proposition}[theorem]{Proposition}
\newtheorem{lemma}[theorem]{Lemma}
\newtheorem{lem}[theorem]{}
\theoremstyle{definition}
\newtheorem{definition}[theorem]{Definition}
\theoremstyle{remark}
\newtheorem{remark}[theorem]{Remark}
\newtheorem{example}[theorem]{Example}
\newcommand{\blem}{\begin{lem} \rm}
\newcommand{\elem}{\end{lem}}
\newcommand\M{\mathcal{M}}

\renewcommand\M{\mathcal{M}}

\renewcommand{\L}{\mathcal{L}}

\newcommand{\T}{\mathcal{T}}
\newcommand{\J}{\mathcal{J}}

\newcommand{\U}{\mathcal{U}}

\newcommand{\F}{\mathcal{F}}
\newcommand{\N}{\mathbb{N}}
\newcommand{\R}{\mathbb{R}}

\newcommand{\C}{\mathbb{C}}
\newcommand{\cC}{\mathcal{C}}

\newcommand{\Z}{\mathbb{Z}}

\newcommand{\on}{\operatorname}

\newcommand{\Ham}{\on{Ham}}

\newcommand{\Lag}{\on{Lag}}

\newcommand{\loc}{{\on{loc}}}

\renewcommand{\ker}{ \on{ker}}

\newcommand{\im}{ \on{im}}

\newcommand\dirac{/\kern-1.2ex\partial} % Dirac operator
\newcommand\qu{/\kern-.7ex/} % Categorical quotients
\newcommand\lqu{\backslash \kern-.7ex \backslash} % Categorical
\newcommand\dr{r_+ \kern-.7ex - \kern-.7ex r_-}
 
\def\pd{\partial}

\def\bu{{\bar u}}
\def\bdelta{{\bar\delta}}
\def\hx{{\hat\xi}}
\def\he{{\hat\eta}}
\def\tint{{\textstyle\int}}
\def\CP{{\mathbb{CP}}}

\newcommand{\ol}{\overline}

\newcommand\eps{\epsilon}
\newcommand\Om{\Omega}
\newcommand\om{\omega}

\newcommand{\f}{\frac}
\newcommand{\lan}{\langle}
\newcommand{\ran}{\rangle}
\newcommand{\hh}{{\f{1}{2}}}

\newcommand{\ti}{\tilde}
\newcommand{\tM}{\widetilde{\M}}

\newcommand\cL{\mathcal{L}}

\newcommand\ul{\underline}

\renewcommand\Im{\on{Im}}

\newcommand\bdefn{\begin{definition}}
\newcommand\edefn{\end{definition}}
\newcommand\bea{\begin{eqnarray*}}
\newcommand\eea{\end{eqnarray*}}
\newcommand\bcv{\left[ \begin{array}{r} }
\newcommand\ecv{\end{array} \right] }

\newcommand\bma{\left[ \begin{array} }
\newcommand\ema{\end{array} \right]}
\newcommand\ben{\begin{enumerate}}
\newcommand\een{\end{enumerate}}
\newcommand\beq{\begin{equation}}
\newcommand\eeq{\end{equation}}
\newcommand\bex{\begin{example}}
\newcommand\bsj{\left\{ \begin{array}{rrr} }
\newcommand\esj{\end{array} \right\}}

\newcommand\Id{\on{Id}}
\newcommand\cI{\mathcal{I}}

\newcommand\eex{\end{example}}

\newcommand\sx{*\kern-.5ex_X}

\def\mathunderaccent#1{\let\theaccent#1\mathpalette\putaccentunder}
\def\putaccentunder#1#2{\oalign{$#1#2$\crcr\hidewidth \vbox
to.2ex{\hbox{$#1\theaccent{}$}\vss}\hidewidth}}

\begin{document}

\title[Floer cohomology and composition of Lagrangian correspondences]{Floer cohomology and geometric composition of Lagrangian correspondences}

\author{Katrin Wehrheim and Chris T. Woodward}

\address{Department of Mathematics,
Massachusetts Institute of Technology,
Cambridge, MA 02139.
{\em E-mail address: katrin@math.mit.edu}}

\address{Department of Mathematics, 
Rutgers University,
Piscataway, NJ 08854.
{\em E-mail address: ctw@math.rutgers.edu}}

\begin{abstract}  
We prove an isomorphism of Floer cohomologies under geometric composition
of Lagrangian correspondences in exact and monotone settings.

\vspace{-10mm}
\end{abstract} 

\maketitle

%\tableofcontents

\section{Introduction}  
Lagrangian correspondences were described by Weinstein
\cite{we:le,we:sc} as generalizations of symplectomorphisms, in an
attempt to build a symplectic category with composable morphisms
between non-symplectomorphic manifolds.  By definition a {Lagrangian
  correspondence} from $M_0$ to $M_1$ is a Lagrangian submanifold in
the product, $L_{01} \subset M_0^-\times M_1$, with respect to the
symplectic structure $(-\omega_{M_0})\times\omega_{M_1}$.  The basic
examples are graphs of symplectomorphisms.  Composition of
symplectomorphisms generalizes to {\em geometric composition} of
Lagrangian correspondences $L_{01}\subset M_0^-\times M_1$,
$L_{12}\subset M_1^-\times M_2$, defined by
\begin{equation}\label{alg comp}
L_{01} \circ L_{12} := \bigl\{ (x_0,x_2)\in M_0\times M_2 \,\big|\, \exists x_1 :
(x_0,x_1)\in L_{01}, (x_1,x_2)\in L_{12} \bigr\} .
\end{equation}
In general this will be a singular subset of $M_0^-\times M_2$ which is
isotropic at smooth points. However, if we assume transversality of the
intersection $ L_{01} \times_{M_1} L_{12} := \bigl( L_{01} \times
L_{12}\bigr) \cap \bigl(M_0^- \times \Delta_{M_1}\times M_2 \bigr)$,
then the restriction of the projection $ \pi_{02}: M_0^- \times M_1
\times M_1^- \times M_2 \to M_0^- \times M_2$ to $L_{01} \times_{M_1}
L_{12}$ is an immersion \cite{gu:rev,we:co}, and hence $L_{01}\circ L_{12} \subset
M_0^-\times M_2$ is an immersed Lagrangian correspondence.  We will study the class of {\em embedded} geometric
compositions, for which in addition $\pi_{02}$ is injective, and hence
$L_{01} \circ L_{12}$ is a smooth Lagrangian correspondence.

Lagrangian correspondences arise naturally in various contexts.  Perutz \cite{per:lag,per:lag2} proposed a construction of three and four-manifold invariants using
Floer theory for Lagrangian correspondences in symmetric products, 
which generalize the tori in Heegard Floer homology \cite{OSZ}.
Seidel proposed a generalized version of his exact triangle in Floer cohomology \cite{se:lo} for fibered versions of symplectic Dehn twists,
whose vanishing cycle is a spherically fibered Lagrangian correspondence.
Seidel and Smith \cite{ss:li} proposed a symplectic definition of
Khovanov homology, using Lagrangians constructed as geometric compositions of the fibered vanishing cycles.  Finally, moduli spaces of flat bundles on three-dimensional cobordisms define Lagrangian correspondences between the moduli spaces of
bundles on the boundary surfaces, such that composition of cobordisms corresponds to geometric composition of correspondences.
The associated Floer cohomology groups, which we construct in \cite{fielda}, 
may be viewed as symplectic versions of instanton Floer homology for three manifolds.

Naturally the question arises of how composition of correspondences 
affects Floer cohomology. In this paper we prove that Floer cohomology is isomorphic under embedded geometric composition. For a precise general statement, it is best to use the language of quilted Floer cohomology developed in \cite{we:co} which defines $HF(L_{01},L_{12},\ldots, L_{(k-1)k})$ for a cyclic sequence of Lagrangian
correspondences $L_{(\ell-1)\ell}\subset M_{\ell-1}^-\times M_\ell$ between symplectic manifolds $M_0, M_1,\ldots, M_k=M_0$. If the composition 
$L_{(\ell-1)\ell}\circ L_{\ell(\ell+1)}$ is embedded, then we obtain under suitable monotonicity assumptions a canonical isomorphism
\begin{equation}\label{eq:iso}
HF(\ldots, L_{(\ell-1)\ell}, L_{\ell(\ell+1)}, \ldots)
\cong
HF(\ldots, L_{(\ell-1)\ell}\circ L_{\ell(\ell+1)}, \ldots) .
\end{equation}
Here the quilted Floer cohomology on the left hand side counts $k$-tuples of 
pseudoholomorphic strips $(u_j:\R\times[0,1]\to M_j)_{j=0,\ldots,k-1}$, whose boundaries match up via the Lagrangian correspondences, $(u_{j-1}(s,1),u_j(s,0))\in L_{(j-1)j}$. 
On the right hand side of \eqref{eq:iso}, no strip in $M_\ell$ is taken into account, and the strips 
$M_{\ell-1}$ and $M_{\ell+1}$ match up directly via 
$(u_{\ell-1}(s,1),u_{\ell+1}(s,0))\in L_{(\ell-1)\ell}\circ L_{\ell(\ell+1)}$.
Rather than going through the general definition in detail, we will prove in detail the following representative example in the familiar notation of Floer cohomology for pairs of Lagrangians in the same symplectic manifold.

\begin{theorem} \label{main}  
Let $M_0,M_1,M_2$ be symplectic manifolds 
that are either compact or satisfy the `bounded geometry' assumptions as in \cite[Chapter~7]{se:bo}.\footnote{
More precisely, we consider symplectic manifolds that are the interior of Seidel's compact symplectic manifolds with boundary and corners. 
%K
We can in fact deal with more general noncompact manifolds, such as cotangent bundles or symplectic manifolds with convex ends, for which bubbling can be excluded in moduli spaces up to dimension $1$, as detailed in Section~\ref{sec:mon}.
Moreover, we require that transverse Floer trajectory spaces be constructed as in Section~\ref{sec:qHF} using almost complex structures $J$ such that, with respect to $J$-compatible metrics, up to second derivatives of $J$ as well as the curvature are uniformly bounded.
}.
Let
$$ L_0 \subset M_0, \ \ L_{01} \subset M_0^- \times M_1, \ \ L_{12}
\subset M_1^- \times M_2, \ \ L_2 \subset M_2^-  $$
be compact Lagrangian submanifolds such that the geometric composition
$L_{01}\circ L_{12}$ is embedded.
Then the canonical bijection
$(L_0 \times L_{12})\cap (L_{01} \times L_2) \cong (L_0 \times L_2) \cap(L_{01} \circ L_{12})$ induces an isomorphism
\begin{equation} \label{maybeiso}
HF(L_0 \times L_{12},L_{01} \times L_2)
\overset{\sim}{\to}
HF(L_0 \times L_2, L_{01} \circ L_{12}) ,
\end{equation}
provided the following assumptions hold:
\begin{enumerate}
\item \label{b}
The pair $(L_0\times L_{12}$, $L_{01}\times L_2)$ of Lagrangian submanifolds in $M_0\times M_1^-\times M_2$
is monotone (or exact) for Floer theory, that is with some $\tau> 0$ (or $\tau =0$)  we have
$$ 2 \tint v^*\omega_N = \tau \cdot I_{\rm Maslov}(v^*T(L_0\times L_{12}), v^*T(L_{01}\times L_2)) $$
for all maps from the annulus $v:S^1\times[0,1]\to M_0\times M_1^-\times M_2$ with 
Lagrangian boundary conditions $v(S^1\times\{0\})\subset L_0\times L_{12}$ and
$v(S^1\times\{1\})\subset L_{01}\times L_2$.
The Maslov index is defined by choosing a trivialization $v^*T(M_0\times M_1^-\times M_2)\cong S^1\times[0,1]\times\C^n$, then $I_{\rm Maslov}(v^*T(L_0\times L_{12}), v^*T(L_{01}\times L_2))$
is the difference of Maslov indices of the two loops in the Lagrangian Grassmannian of $\C^n$.
\item \label{c}
The minimal positive Maslov index in \eqref{b} is $2$, that is there exists no annulus $v$ with
$I_{\rm Maslov}(v^*T(L_0\times L_{12}), v^*T(L_{01}\times L_2)) =1$.
\item \label{d}
Each of the $L_0,L_{01},L_{12},L_2$ has minimal Maslov index $\geq 3$.
(Here the minimal Maslov index of $L\subset M$ is the positive generator of $I_{\rm Maslov}(\pi_2(M,L))\subset\Z$.)
\end{enumerate}
\end{theorem} 

% Note that this Theorem only has substantial content if at least two of $M_0,M_1,M_2$ have nonzero dimension.
% For $M_1=\pt$ a point, the isomorphism is the obvious one for two pairs of Lagrangians
% $L_0,L'_{0}\subset M_0$ and $L'_{2},L_2\subset M_2$, 
% $$
% HF(L_0 \times L'_{2} , L'_{0} \times L_2) 
% \cong HF(L_0 , L'_{0})\otimes HF(L_2, L'_{2}) \cong
% HF(L_0 \times L_2, L'_{0} \times L'_{2}) .
% $$
% For $M_0=M_2=\pt$ our embeddedness assumption says that the two Lagrangians $L_{01},L_{12}\subset M_1$ intersect transversely in zero or one point, corresponding to $L_{01}\circ L_{12}=\emptyset$ or $\pt\times\pt$. In that case our strip-shrinking analysis actually does not apply, but $HF(L_{01},L_{12})$ is trivially $\Z$ or $\{0\}$, corresponding to $HF(\pt\times\pt, L_{01}\circ L_{12})$.

Note that (a) implies monotonicity on homotopy groups for the symplectic manifolds, i.e.\ $[\omega_{M_i}] = \tau  c_1(TM_i)$ on $\pi_2(M_i)$ for $i=0,1,2$, as well as for each Lagrangian, i.e.\ 
$2[\omega_{M}] = \tau I_{\rm Maslov}$ on $\pi_2(M,L)$ for $(M,L)$ given by $(M_0,L_0)$, $(M_0^-\times M_1,L_{01})$, $(M_1^-\times M_2,L_{12})$, or $(M_2,L_2)$.
Assumptions \eqref{b} and \eqref{c} are necessary in their full strength for a subtle bubble exclusion argument, as explained below. They are met, for example, if all Lagrangians are 
orientable and exact, or if they are orientable, 
monotone, and the image of either $\pi_1(L_0\times L_{12})$ or $\pi_1(L_{01}\times L_2)$ in $\pi_1(M_0\times M_1\times M_2)$ is torsion.
In \cite{quiltfloer} we discuss some alternative conditions ensuring monotonicity.
Note that \eqref{c} also is the natural assumption that excludes self-connecting trajectories in the construction of Floer homology. Similarly, \eqref{d} is needed only to ensure that Floer homology is well defined. In \cite{fieldb} we generalize Theorem~\ref{main} to an isomorphism in the derived category of matrix factorization, allowing to drop assumption~\eqref{d}.

In this paper, the isomorphism \eqref{maybeiso} of Floer cohomology
groups is completely proven only with $\Z_2$-coefficients. 
The discussion of coherent orientations
-- in the presence of orientations and relative spin structures on the Lagrangians --  
can be found in \cite{orient}.  There should also be versions of this
result for Floer cohomology with gradings, coefficients in flat vector
bundles, and Novikov rings.  We give a detailed statement and proof
for the gradings in \cite{quiltfloer}. 

\begin{figure}[ht]
\begin{picture}(0,0)%
\includegraphics{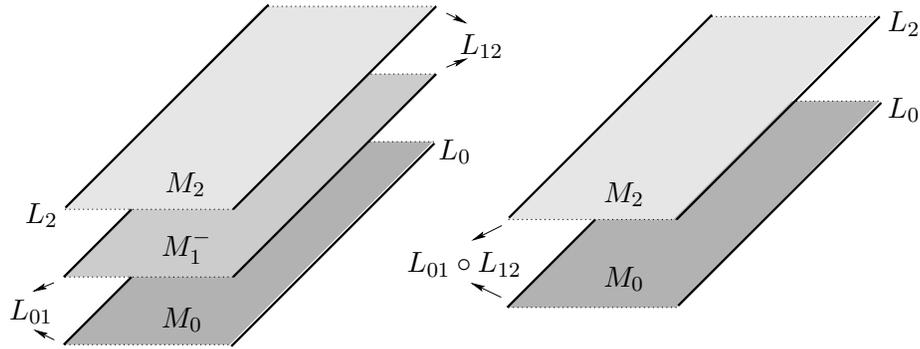}%
\end{picture}%
\setlength{\unitlength}{1865sp}%
\begingroup\makeatletter\ifx\SetFigFont\undefined%
\gdef\SetFigFont#1#2#3#4#5{%
  \reset@font\fontsize{#1}{#2pt}%
  \fontfamily{#3}\fontseries{#4}\fontshape{#5}%
  \selectfont}%
\fi\endgroup%
\begin{picture}(12252,4524)(136,-4213)
\put(226,-3880){\makebox(0,0)[lb]{{{$L_{01}$}}}}
\put(436,-2600){\makebox(0,0)[lb]{{{$L_2$}}}}
\put(2296,-2186){\makebox(0,0)[lb]{{{$M_2$}}}}
\put(2251,-3086){\makebox(0,0)[lb]{{{$M_1^-$}}}}
\put(2251,-3986){\makebox(0,0)[lb]{{{$M_0$}}}}
\put(6211,-350){\makebox(0,0)[lb]{{{$L_{12}$}}}}
\put(5931,-1750){\makebox(0,0)[lb]{{{$L_0$}}}}
\put(8126,-2321){\makebox(0,0)[lb]{{{$M_2$}}}}
\put(8126,-3446){\makebox(0,0)[lb]{{{$M_0$}}}}
\put(5511,-3250){\makebox(0,0)[lb]{{{$L_{01}\circ L_{12}$}}}}
\put(11900,-50){\makebox(0,0)[lb]{{{$L_{2}$}}}}
\put(11900,-1200){\makebox(0,0)[lb]{{{$L_0$}}}}
\end{picture}%
\caption{Tuples of pseudoholomorphic strips that are counted for 
$HF(L_0 \times L_{12},L_{01} \times L_2)$
and for $HF(L_0 \times L_2,L_{01} \circ L_{12})$}
\label{layer}
\end{figure}

Throughout we will use the construction of Floer cohomology based on \cite{fl:rel,oh:fl1,fhs:tr}.
The Floer differential for $(L_0 \times L_{12},L_{01} \times L_2)$ counts triples of pseudoholomorphic strips in $M_0,M_1^-,M_2$ (see Figure \ref{layer} below).  In the standard definition, one would take the width of all three strips to be equal,
but we show in \cite{quiltfloer} that 
one can in fact allow the widths of the strips to differ.  
%(These domains are not conformally equivalent due to the identification between boundary components.) 
The main difficulty then is to prove that under the stated assumptions and with the width of the middle strip sufficiently close to zero, the triples of pseudoholomorphic strips in
$M_0,M_1^-,M_2$ are in one-to-one correspondence with the pairs of
pseudoholomorphic strips in $M_0,M_2$ that are counted in the Floer
differential for $(L_0 \times L_2,L_{01} \circ L_{12})$.
As in similar situations in Floer theory, the proof is an application
of the implicit function theorem, on one hand, and compactness results
for shrinking the middle strip, on the other.  In the limit
various kinds of bubbling can occur: Sphere bubbles in $M_0$, $M_1$, $M_2$;
disk bubbles in  
$(M_0,L_0)$, $(M_2,L_2)$, 
$(M_0\times M_1,L_{01})$, $(M_1\times M_2,L_{12})$, $(M_0\times M_2,L_{01}\circ L_{12})$;
and a novel type of bubble which we call a {\em figure eight bubble}. 
The latter is a triple of pseudoholomorphic maps
$v_0: \R \times (-\infty,-1] \to M_0,  v_1: \R \times [-1,1] \to M_1, v_2: \R \times [1,\infty) \to M_2$ such that 
$ (v_0(\tau, -1), v_1(\tau,-1)) \in L_{01}$,  $(v_1(\tau,1),v_2(\tau,1)) \in L_{12}$.

To explain the name, note that under stereographic projection to the sphere, or after transformation $z\mapsto \frac 1z$ of $\C\cong\R^2$, the lines $\Im(z) = \pm 1$ appear as a figure eight as in Figure \ref{figureeight}.
These pictures are labeled in the pictorial language of \cite{quilts}: The maps $v_0,v_1,v_2$ form a ``quilt'' on the punctured $S^2$, whose ``patches'' are the domains of the three maps (labeled by the target spaces), and with ``seams" on the intersections of these domains (labeled by the ``seam condition'' $L_{01}$ or $L_{12}$ that is satisfied there).
We conjecture that the maps $(v_0,v_1,v_2)$ can be extended continuously to the closure of their domains in $S^2$ by a point $(v_0(\infty),v_1(\infty),v_2(\infty))\in
L_{01}\times M_2\cap M_0\times L_{12}$.
\begin{figure}[ht]
\includegraphics[width=4in,height=1.5in]{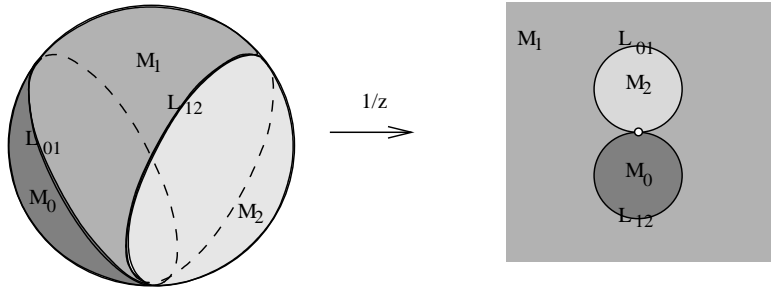}
\caption{Figure Eight bubble} \label{figureeight}
\end{figure}
However, we cannot in general prove this removal of singularities,
nor is there a readily available Fredholm theory for seams touching tangentially as in Figure~\ref{figureeight}.
Thus we are lacking the construction of a moduli space of figure eight bubbles.
Instead, as in \cite{we:en} we exclude bubbling by energy quantization without establishing a geometric description of the bubble. This method hinges on strict monotonicity with a nonnegative constant $\tau\geq 0$ as well as the $2$-grading assumption \eqref{c}. 

Theorem~\ref{main} has a wide range of applications:
First, it provides a tool for symplectic topology, which has not yet been exhaustively used.
In \cite{quiltfloer} we give examples of elementary Floer homology calculations arising from the representation of symplectic quotients as Lagrangian correspondence.
For example, a simple iteration in $n$ confirms the calculation $HF(T^n_{\rm Cl},T^n_{\rm Cl})\cong H_*(T^n)$  of Cho \cite{cho:hol} for the Clifford torus in $\CP^n$.
We also show that non-displaceability of Lagrangians in product symplectic manifolds follows directly if the Lagrangian, viewed as correspondence, has an image in one of the factors whose Floer homology is nonzero.
This explains e.g.\ the nondisplaceability of standard Lagrangian embeddings
$(S^1)^{n-k} \times S^{2k-1}\hookrightarrow (\CP^{k-1})^-\times \CP^{n}$ 
(for example the sphere $S^3 \hookrightarrow(\CP^1)^-\times\CP^2$)
by the fact that their projection to $\CP^n$ contains the nondisplaceable Clifford torus.
An application to non-triviality of symplectic mapping class groups is given in \cite{fieldb}.  Second, our isomorphism is key to proving the topological invariance of various Floer cohomology groups arising from decompositions in low-dimensional topology; for example, the symplectic version of instanton (knot) homology constructed in \cite{fielda,fieldb}, and Seidel-Smith homology and Heegard-Floer homology, for which it provides alternative constructions \cite{reza:ss}, \cite{lekili:heegard, perutz-lekili in progress}\footnote{
Excluding figure eight bubbling in negatively monotone symmetric products requires a somewhat more subtle analysis. Using a weak removable singularity theorem, it suffices to establish that potential homotopy classes of figure eight bubbles have zero energy \cite{w:banff,multiple}. This seems to be the case for all correspondences introduced by Perutz.
At the time of last revision of this paper note that the alternative approach presented in \cite{ll} assumes real valued symplectic actions, which directly implies our assumption \eqref{b} with $\tau=0$.
}.

Third, from a more conceptual point of view, Theorem~\ref{main} is
used in \cite{we:co} to give a solution to the problem in Weinstein's
construction that composition of Lagrangian correspondences is not
always defined.  Using the result here, one may construct a symplectic
$2$-category, in which all Lagrangian correspondences are composable
morphisms and Floer cohomology groups (as $2$-morphism spaces) are
well defined.  Thus one removes the quotes in Weinstein's ``category"
by promoting the construction to a $2$-category, using Floer theory.\\

{\em We thank Paul Seidel and Ivan Smith for encouragement and helpful discussions.}

\section{Floer cohomology for monotone Lagrangian correspondences}

In this section we first explain why both Floer cohomologies in Theorem~\ref{main} are well defined. Then we give a specific ``quilted" setup and choice of perturbations for both that reduce the isomorphism of Floer cohomologies to a bijection of moduli spaces that is proven in Section~\ref{shrink}.

\subsection{Monotonicity assumptions and index identities}
\label{sec:mon}

The significance of the monotonicity and Maslov index assumptions in Theorem~\ref{main} is the following energy-index relation and relative grading.

\begin{proposition} \label{prop monotone}
Suppose that the pair $(L_0,L_1)$ of Lagrangians in $M$ is monotone, transverse,
and has minimal annulus Maslov index $N\geq 2$. (That is, $N$ is the positive generator
of $\{I_{\rm Maslov}(v^*TL_0, v^*TL_1)\,|\, v:S^1\times[0,1] \to M, v(S^1\times\{j\})\subset L_j\}\subset\Z$.)

Then for any $x_\pm \in L_0\cap L_1$ there exist constants
$c(x_-,x_+)\in\R$ and $\mu(x_-,x_+)\in\Z$ such that for all
strips $u:\R\times[0,1]\to M$ with boundary values in $(L_0,L_1)$ and limits $u(\pm\infty, \cdot)=x_\pm$ we have
\begin{equation}\label{energy index}
2 E(u) =  \tau\cdot {\rm Ind}(D_u) + c(x_-,x_+) , \qquad
 {\rm Ind}(D_u) \equiv \mu(x_-,x_+) \quad\text{mod}\; N .
\end{equation}
Here $E(u)= \int u^*\omega $ is the energy and $D_u$ the linearized Cauchy-Riemann operator at $u$.
\end{proposition}
\begin{proof}
Given two strips $u_1, u_2:\R\times[0,1]\to M$ we can glue them together (reversing the orientation of $u_2$) to an annulus $v:S^1\times[0,1] \to M$, then 
$\int v^*\omega= E(u_1)- E(u_2)$ and $I_{\rm Maslov}(v^*TL_0, v^*TL_1)=
 {\rm Ind}(D_{u_1})- {\rm Ind}(D_{u_2})$. 
 So the energy-index relation follows from monotonicity
 $\int v^*\omega = \tau I_{\rm Maslov}(v^*TL_0, v^*TL_1)$, 
 and the index identity follows from $I_{\rm Maslov}(v^*TL_0, v^*TL_1)\subset N\Z$.
\end{proof}

%K
The energy-index relation ensures energy bounds for the moduli spaces of fixed index and thus compactness up to bubbling (`pointwise blow-up of the gradient') and breaking of trajectories (`nontrivial amounts of energy moving off into both ends of the strip').\footnote{
%K
For a noncompact symplectic manifold, one needs to establish $\cC^0$-bounds on the holomorphic maps, before 'standard Gromov compactness' can be quoted. Note that the domains of maps under our consideration are such that each interior point has bounded distance from a boundary point, where the maps take values in a compact Lagrangian submanifold or in the projection of a compact Lagrangian correspondence to one factor.
Hence it suffices to establish uniform bounds on the gradient (i.e.\ exclude bubbling).
}
Together with the index identity it excludes bubbling in moduli spaces of index less than $N$ as follows: Any bubbling leads to a new (possibly broken) trajectory connecting the same points but with less energy.\footnote{
Such energy loss can be established by proving convergence of rescaled maps to disks or spheres.
Alternatively, this can be shown by a mean value inequality as in \cite[4.3]{ms:jh}, \cite{we:en}, or  Lemma~\ref{bubb}, which only requires uniform bounds on the curvature and up to second derivatives of the almost complex structure $J$ w.r.t.\ a $J$-compatible metric on $M$.
%K
Such bounds will also be required for the proof of energy loss during strip shrinking in Lemma~\ref{bubb}, hence they are a standing assumption for noncompact manifolds.

In addition, both approaches require the removable singularity theorem (\cite[Thm 4.1.2]{ms:jh}) to hold on $M$.
}
By monotonicity, less energy means strictly less index. By the index identity mod $N$ that means negative index. 
By transversality (previously established for moduli spaces of negative index) that means an empty set: The new trajectory doesn't exist, so the bubbling didn't happen. We spelled out this argument because we will use it again to exclude figure eight bubbling -- by only proving energy loss, not actually giving a geometric description of the bubble.

Working with $N=2$ there is just one point in the construction of Floer cohomology where this argument fails: The $1$-dimensional moduli spaces of self-connecting Floer trajectories have index $2$, so bubbling could lead to an index $0$ solution (which are always constant due to the $\R$-action). Assumption \eqref{d} serves to exclude this scenario by index additivity arguments: 
Any pseudoholomorphic disk bubble with boundary on $L$ will reduce the index by at least $N_L$, the minimal Maslov index on $\pi_2(M,L)$. So $N_L\geq 3$ ensures that the remaining solution would have negative index (and the same holds for sphere bubbles whose Chern number would be at least $\frac 12 N_L$).
Note that this argument, unlike the previous bubbling exclusion by energy loss, requires an identification of the bubbles as spheres and disks.\footnote{
For noncompact symplectic manifolds, this requires a compactification as in \cite{se:bo} or the use of the maximum principle on convex ends. Alternatively, one could restrict to $N\geq 3$ 
%K
(e.g.\ exact Lagrangians in a cotangent bundle have $N=\infty$) or use any other valid argument to prove $\partial^2=0$.
}
In our case it also requires that we work with a split almost complex structure (preserving the factors of $M_0\times M_1^-\times M_2$), otherwise pseudoholomorphic disks in the product manifold don't necessarily have the minimal index of a disk in one of the factors.
We will show in Section~\ref{sec:qHF} that we can achieve transversality with a split almost complex structure, and hence our assumptions indeed ensure that the Floer cohomology $HF(L_0 \times L_{12},L_{01} \times L_2)$ is well defined.
The next Lemma shows that the Floer cohomology $HF(L_0 \times L_2, L_{01} \circ L_{12})$ for the composed Lagrangian correspondence is also well defined. 

\begin{lemma} \label{neednoviterbo}
In the setting of Theorem~\ref{main}, the assumptions \eqref{b} and \eqref{c} imply the analogous assumptions for the pair $(L_0 \times L_2, L_{01} \circ L_{12})$ of Lagrangians in $M_0\times M_2^-$.
Assumption \eqref{d} implies that $\partial^2=0$ on $CF(L_0 \times L_2, L_{01} \circ L_{12})$, and hence the Floer cohomology is well defined.
\end{lemma}
\begin{proof}
Consider any annulus $(u_0\times u_2) : S^1\times[0,1]\to M_0\times M_2^-$ with Lagrangian boundary conditions $(u_0\times u_2)(S^1\times\{0\})\subset L_0\times L_{2}$ and $(u_0\times u_2)(S^1\times\{1\})\subset L_{01}\circ L_{12}$. 
By the embedded composition there exists a unique lift $u_1:S^1\to M_1^-$ such that
$(u_0|_{t=1}\times u_1)(S^1)\subset L_{01}$ and
$(u_1\times u_2|_{t=1})(S^1)\subset L_{12}$.
Now we can reverse the parametrization in $\overline{u}_2(s,t):=u_2(s,1-t)$ and extend $u_1$ constant along $[0,1]$ 
to define an annulus $(u_0\times u_1\times \overline{u}_2):S^1\times[0,1]\to M_0\times M_1^-\times M_2$ as in \eqref{b}.
Here $u_1^*\omega_1=0$, hence 
$\int (u_0\times u_1\times \overline{u}_2)^* (\omega_0 \times (-\omega_1) \times \omega_2) = \int (u_0\times u_2)^* (\omega_0 \times (- \omega_2 ))$.
To identify the Maslov indices, pick the same trivializations $u_j^*TM_j\cong S^1\times[0,1]\times V_j$ for $j=0,2$ in both cases, then equality follows from the identity
\begin{equation}\label{maslovid}
 I(\gamma_{01}) + I(\gamma_{12}) =  I(\gamma_{01}\times\gamma_{12}) = I(\gamma_{02})
\end{equation}
for loops of Lagrangians $\gamma_{01}:S^1\to\Lag(V_0^-\times V_1)$,
$\gamma_{12}:S^1\to\Lag(V_1^-\times V_2)$,
and 
$\gamma_{02}:S^1\to\Lag(V_0^-\times V_2)$ given by
$\gamma_{02}(s)=(\gamma_{01}(s)\times\gamma_{12}(s))\cap (V_0\times\Delta_{V_1}\times V_2)$.
The first equality is simply additivity of the Maslov index. To see the second equality we fix Lagrangians $\Lambda_j\subset V_j$ for $j=0,2$, then the Maslov indices can be expressed as the intersection number with $\Lambda_0\times\Delta_{V_1}\times\Lambda_2$ resp.\ $\Lambda_0\times\Lambda_2$.
With this choice the intersections are identified,
$$
K(s):= \bigl(\gamma_{01}(s)\times\gamma_{12}(s)\bigr) \cap \bigl(\Lambda_0\times\Delta_{V_1}\times\Lambda_2\bigr)
\cong
\gamma_{02}(s) \cap 
\bigl(\Lambda_0\times\Lambda_2 \bigr) .
$$
Now we need to compare the crossing forms $\Gamma_{0112}(s),\Gamma_{02}(s) : K(s) \to \R$ at regular crossings $s\in S^1$.
Fix a Lagrangian complement $\gamma_{02}(s)^c\subset V_0\times V_2^-$, 
then $\gamma_{02}(s)^c\times\Delta_{V_1}$, after appropriate transposition of factors, is a Lagrangian complement for $\gamma_{01}(s)\times\gamma_{12}(s)$, due to the assumption of transversality $(L_{01}\times L_{12}) \pitchfork (M_0\times\Delta_{M_1}\times M_2)$.
So for $v_{0112}=(v_0,v_1,v_1,v_2)\in K(s)$ one finds
$(w_0,w_2)(t)\in \gamma_{02}(s)^c$ and $w_1\in V_1$  such that 
$v+(w_0,w_1,w_1,w_2)(t)\in (\gamma_{01}\times\gamma_{12})(s+t)$.
For the corresponding vector $v_{02}=(v_0,v_2)\in K(s)$ this automatically gives
$v_{02}+(w_0,w_2)(t)\in \gamma_{02}(s+t)$.
With this we identify the crossing forms 
\begin{align*}
\Gamma_{0112}(s)v_{0112} 
&= \tfrac d{dt}\bigr|_{t=0} (\omega_0\oplus -\omega_1\oplus \omega_1 \oplus -\omega_2) \bigl(v_{0112},(w_0,w_1,w_1,w_2)(t)\bigr) \\
&= \tfrac d{dt}\bigr|_{t=0} \bigl( - \omega_0(v_0, w_0) + \omega_1(v_1, w_1)
- \omega_1(v_1, w_1 ) + \omega_2(v_2, w_2) \bigr) \\
&= \tfrac d{dt}\bigr|_{t=0} (\omega_0\oplus -\omega_2)
\bigl(v_{02}, (w_0,w_2)(t) \bigr)
= \Gamma_{02}(s)v_{02} .
\end{align*}
This proves equality of the Maslov indices in \eqref{maslovid} and this finishes the proof of \eqref{b} and \eqref{c}.

In the absence of assumption \eqref{d} we have $\partial^2 = w {\rm Id}$ a multiple of the identity in both Floer theories, see \cite{oh:fl1} and \cite{fieldb}. 
A derived version of Theorem~\ref{main} implies that the value of $w$ is the same for both theories, see Remark~\ref{rmk:derived main}.
Assuming \eqref{b} for the pair $(L_0 \times L_{12}, L_{01} \times L_{2})$ we obtain $w=0$ and thus also $\partial^2=0$ on $CF(L_0 \times L_2, L_{01} \circ L_{12})$.
\end{proof}

The index calculation in \eqref{maslovid} analogously holds for strips. This identifies the index on the two complexes in Theorem \ref{main}. 
Recall here from \cite{fl:rel} that the index of the linearized Cauchy-Riemann operator $D_u$ 
at a map $u:\R\times[0,1]\to M$ with Lagrangian boundary conditions $u(\R\times\{i\})\subset L_i$ for $i=0,1$ and limits
$u(s,\cdot) \underset{s\to\pm\infty}{\longrightarrow} L_0\pitchfork L_1$ at transverse intersection points is given by the Maslov-Viterbo index,
$$
{\rm Ind}(D_u) = I_{MV}(u) := I(\gamma_0,\gamma_1) , \qquad \gamma_i(s)= T_{u(s,i)} L_i .
$$
Here the Maslov index of the pair of paths is defined by choosing a trivialization $u^*TM \cong \R\times[0,1] \times V$ (independent of $t\in[0,1]$ for $s\to\pm\infty$ ) so that $\gamma_i$ becomes a path of Lagrangian subspaces in the symplectic vector space $V$.

\begin{lemma}   \label{index}
Let $L_0\subset M_0$, $L_{01}\subset M_0^-\times M_1$, 
$L_{12}\subset M_1^-\times M_2$, and $L_2\subset M_2^-$ be Lagrangians 
such that the composition $L_{01} \circ L_{12}=:L_{02}$ is embedded.
Suppose that the intersection $L_0\times L_{12}\cap L_{01}\times L_2$
(and hence also $L_0\times L_2 \cap L_{01}\circ L_{12}$) is transverse
and consider a map $(u_0,u_2):\R\times[0,1]\to M_0\times M_2$ taking
boundary values in $(L_0\times L_2,L_{01} \circ L_{12})$, and limiting
to intersection points as $s\to\pm\infty$.  Let
$(u_0,u_1,\overline{u}_2):\R\times[0,1]\to M_0 \times M_1 \times M_2$
be the corresponding map which takes boundary values in $(L_0\times
L_{12},L_{01} \times L_{2})$ and satisfies $\partial_t u_1 = 0$.
(Here $\overline{u}_2$ reverses the $[0,1]$-parametrization of $u_2$.)
Then the indices of the linearized operators and the energies are equal, $$
{\rm Ind}(D_{(u_0,u_2)})=
{\rm Ind}(D_{(u_0,u_1,\overline{u}_2)}),
\qquad
E((u_0,u_2)) = E((u_0,u_1,\overline{u}_2)).
$$
\end{lemma}

\begin{proof}
The identity of Maslov indices follows as in Lemma~\ref{neednoviterbo}.
Alternatively, it could be deduced from a more general result of Viterbo \cite[Proposition 3]{vi:in}.
For the energies just note that
$\int u_2^*\omega_2 = \int \overline{u}_2^*(-\omega_2)$
and $ \int u_1^*\omega_1 = 0$.
\end{proof}

\subsection{Quilted setup for Floer cohomology}
\label{sec:qHF}

As in Theorem~\ref{main} let $M_0,M_1,M_2$ be symplectic manifolds and let
$$ L_0 \subset M_0, \ \ L_{01} \subset M_0^- \times M_1, \ \ L_{12}
\subset M_1^- \times M_2, \ \ L_2 \subset M_2^-  $$
be Lagrangian submanifolds such that the geometric composition
$L_{02}:=L_{01}\circ L_{12}$ is embedded.
The aim of this section is to introduce the ``quilted" setup and give compatible choices of perturbation data for the two Floer cohomologies
$HF(L_0 \times L_{12},L_{01} \times L_2)$ and $HF(L_0 \times L_2, L_{02})$.

First, we need to fix Hamiltonians\footnote{
%K
If some of the symplectic manifolds are noncompact, then we work throughout with Hamiltonian functions that are supported in fixed compact neighbourhoods of the Lagrangians.
}
such that the perturbed intersection points are finite and nondegenerate.  In fact, the following Proposition shows that we can pick a Hamiltonian of split type which achieves simultaneous transversality for the intersection points in both Floer theories.
Given a pair of time-dependent Hamiltonian functions
$(H_0,H_2) \in C^\infty([0,1] \times M_0)\times C^\infty([0,1] \times M_2)$
consider the Hamiltonians 
$H_{02}(t, x_0,x_2)=H_0(t,x_0) - H_2(1-t,x_2)$ on $M_0\times M_2$ and 
$H_{012}(t, x_0,x_1,x_2)=H_0(t,x_0) + H_2(t,x_2)$ on $M_0\times M_1\times M_2$ 
and denote their time $1$ flows by $\phi^{H_{02}}$ and $\phi^{H_{012}}$.
Then the perturbed intersection points $\phi^{H_{02}}(L_0\times L_2)\cap L_{02}$ can be identified with 
$$L_{0} \times_{\phi^{H_0}} L_{02} \times_{\phi^{H_2}} L_2
= \bigl\{ (m_0, m_2) \in M_0 \times M_2
\,\big|\, m_{0} \in L_0, (\phi^{H_0}(m_{0}),m_{2}) \in L_{02} , 
\phi^{H_2}(m_{2}) \in L_2 \bigr\} 
$$
and analogously
$$
\phi^{H_{012}}(L_0\times L_{12})\cap (L_{01}\times L_2) \cong 
L_{0} \times_{\phi^{H_0}} L_{01} \times_{\phi^{H_1}}  L_{12} \times_{\phi^{H_2}} L_2,
$$ 
where $\phi^{H_{j}}$ is the time $1$ flow of the Hamiltonian $H_{j}$ and we use the trivial function $H_1\equiv 0$ on $M_1$.
Note that the Hamiltonians are constructed such that the perturbed intersection points for the two Floer theories are still canonically identified. Indeed, by assumption every point in $L_{02}=L_{01}\circ L_{12}$ has a unique lift to $L_{01}\times_{\Id_{M_1}}L_{12}$.

\begin{proposition}  \label{pert split}
There is a dense open subset 
$\Ham(L_0,L_{02},L_2)\subset C^\infty([0,1] \times M_0)\times C^\infty([0,1] \times M_2)$ such that for every $(H_0,H_2)\in\Ham(L_0,L_{02},L_2)$ and $H_1\equiv 0$ 
the defining equations for
both sets $L_{0} \times_{\phi^{H_0}} L_{02} \times_{\phi^{H_2}} L_2$
and $L_{0} \times_{\phi^{H_0}} L_{01} \times_{\phi^{H_1}}  L_{12} \times_{\phi_{1}^{H_2}} L_2$ are transversal.
\end{proposition}  

\begin{proof}  By assumption $L_0,L_{02},L_2$  are embedded submanifolds
and so locally they are the zero sets of submersions 
$\psi_{0}: M_0 \to \R^{n_0}$,
$\psi_{02}: M_0\times M_2 \to \R^{n_0+n_2}$,
$\psi_{2}: M_2 \to \R^{n_2}$.  
Then the defining equations for 
$L_{0} \times_{\phi^{H_0}} L_{02} \times_{\phi^{H_2}} L_2$ are
\begin{equation} \label{define} 
\psi_{0} ( m_{0} ) = 0 ,\quad
\psi_{02} \bigl( \phi^{H_0}(m_{0}),m_{2} \bigr) = 0 , \quad
\psi_{2} \bigl( \phi^{H_2}(m_{2}) \bigr) = 0 .
\end{equation}
Consider the universal moduli $\U$ space of data
$(H_0,H_2,m_0,m_2)$ satisfying \eqref{define}, where now
each $H_j$ has class $C^\ell$ for some $\ell\geq 2$.  
The linearized equations for $\U$ are
\begin{equation} \label{surj} 
D \psi_{0} ( v_{0} ) = 0 , \quad
D \psi_{02} ( D \phi^{H_0} (h_{0},v_{0}),v_{2}) = 0 ,\quad
D \psi_{2} \bigl( D \phi^{H_2}(h_2,v_{2}) \bigr) = 0 .
\end{equation}
for $v_j \in T_{m_j} M_j$ and $h_j \in C^\ell([0,1] \times M_j)$.
The product of the operators on the left-hand sides of \eqref{surj} are surjective
since each of the maps $ C^\ell([0,1] \times M_{j}) \to T_{\phi^{H_j}(m_{j})}
M_{j}$, $h_{j} \mapsto D \phi^{H_j} (h_{j},0) $
is surjective.  So by the implicit
function theorem $\U$ is a smooth Banach manifold, and we consider its
projection to $ C^\ell([0,1] \times M_0)\times  C^\ell([0,1] \times M_2)$.  By the
Sard-Smale theorem, the set of regular values is dense.
On the other hand, the set of regular values is clearly open.  Hence
the set of smooth functions that are regular values is open and dense.
This is exactly the set of functions $H=(H_0,H_2)$ such that the perturbed intersection 
$L_{0} \times_{\phi^{H_0}} L_{02} \times_{\phi^{H_2}} L_2$ is transversal. 

Moreover, the perturbed intersection
$L_{0} \times_{\phi^{H_0}} L_{01} \times_{\phi^{H_1}}  L_{12} \times_{\phi_{1}^{H_2}} L_2$ is also transversal, since by assumption $L_{01}\times L_{12}$ is transverse
to the diagonal $M_0\times \Delta_{M_1}\times M_2$.
\end{proof} 

In the following, instead of working with perturbed intersection points, we will apply the Hamiltonian diffeomorphisms to the Lagrangians to achieve transversality.
Replacing $L_0$ with $L_0'=\phi^{H_0}(L_0)$ and $L_2$ with $L_2'=(\phi^{H_2})^{-1}(L_2)$ the generators of the two Floer chain groups are the transverse intersections 
\begin{align*}
(L_0'\times L_2') \cap L_{01} 
&\cong  L'_0\times_{\Id_{M_0}}L_{02}\times_{\Id_{M_2}}L'_2 , \\
(L_0'\times L_{12}) \cap (L_{01}\times L_2')
& \cong 
L'_0\times_{\Id_{M_0}}L_{01}\times_{\Id_{M_1}}L_{12}\times_{\Id_{M_2}}L'_2 .
\end{align*}
The forgetful map $(m_0,m_1,m_2)\mapsto (m_0,m_2)$ is a bijection from $\cI$ to 
$(L_0'\times L_2') \cap L_{01}$ since by assumption $L_{01}\times_{\Id_{M_1}}L_{12}\to L_{02}$ is bijective. 
So, after a
%K deleted 'generic'
Hamiltonian perturbation, we have a natural isomorphism of the Floer chain groups
\begin{equation} \label{iso chain}
CF(L_0\times L_{12},L_{01}\times L_2)  \overset{\sim}{\to} CF(L_0\times L_2,L_{02}) 
\end{equation}
and it remains to identify the Floer differentials.
For that purpose we now drop the Hamiltonian from the notation: By abuse of notation we can assume to start out with unperturbed transverse intersections and a natural bijection
$$
\cI := (L_0\times L_2) \pitchfork L_{01} \cong  (L_0\times L_{12}) \pitchfork (L_{01}\times L_2) .
$$
To investigate the Floer trajectories note that we consider $(L_0\times L_2,L_{02})$ as a pair of Lagrangians in $M_0\times M_2^-$
and $(L_0\times L_{12},L_{01}\times L_2)$ as a pair of Lagrangians in $M_0\times M_1^-\times M_2$.
For any symplectic manifold $M$ let $\J(M)$ be the space of almost complex structures on $M$ that are compatible with the symplectic structure $\omega_M$.\footnote{
If $M$ is noncompact, then we assume as in \cite{se:bo} that the almost complex structure extends to the compact symplectic manifold with boundary and corners, whose interior is $M$. More generally, it would suffice to work with any noncompact $M$ and $J$ for which the bubble exclusion arguments hold, as detailed in Section~\ref{sec:mon} 
%K
and Lemma~\ref{bubb}.
In particular, this requires uniform bounds on the curvature and up to second derivatives of the almost complex structures $J$ with respect to $J$-compatible metrics.
}
We pick time-dependent almost complex structures $J_0\in\cC^\infty([0,1],\J(M_0))$ and $J_2 \in \cC^\infty([0,1],\J(M_2))$,
then $J_0(t,m_0)\times (- J_2(1-t,m_2))$ defines a compatible almost complex structure on $M_0\times M_2^-$.
Now any pseudoholomorphic strip $w_{02}:\R\times[0,1]\to M_0\times M_2^-$ with boundary values on $(L_0\times L_2,L_{02})$ corresponds by ``unfolding" $w_{02}(s,t)=(u_0(s,t),u_2(s,1-t))$  to a pair of strips $\bigl( u_i:\R\times[0,1]\to M_i\bigr)_{i=0,2}$ satisfying
\begin{align} \label{u0 u2}
&\partial_s u_0 + J_0(t,u_0) \partial_t u_0=0,\qquad
\partial_s u_2 + J_2(t, u_2) \partial_t u_2=0, \\
&
u_0(s,0)\in L_0, 
\qquad (u_0(s,1),u_2(s,0))\in L_{02},
\qquad u_2(s,1) \in L_2 .  \nonumber
\end{align}
Similarly, pick an almost complex structure $J_1\in\J(M_1)$, then 
$J_0\times (-J_1)\times J_2$ defines a compatible almost complex structure on $M_0\times M_1^-\times M_2$ and any pseudoholomorphic strip with boundary values on $(L_0\times L_{12},L_{01}\times L_2)$ corresponds by ``unfolding" to a triple of strips 
$\bigl(v_i:\R\times[0,1]\to M_i\bigr)_{i=0,1,2}$ satisfying
\begin{align} \label{v0 v1 v2}
&\partial_s v_0 + J_0(t,v_0)\partial_t v_0=0 ,\quad
\partial_s v_1 + J_1(v_1)\partial_t v_1=0 ,\quad
\partial_s v_2 + J_2(t,v_2)\partial_t v_2=0 , \\
&
v_0(s,0)\in L_0, 
\quad (v_0(s,1),v_1(s,0))\in L_{01},
\quad (v_1(s,1),v_2(s,0))\in L_{12},
\quad v_2(s,1)\in L_2 . \nonumber
\end{align}
In both cases, the trajectories have finite energy $\sum_i \int |\partial_s u_i|^2$ resp.\ 
$\sum_i \int |\partial_s v_i|^2$ iff they converge uniformly to intersection points 
\begin{equation}\label{conv}
\lim_{s\to\pm\infty} (u_0,u_2)(s,\cdot) = (x_0^\pm,x_2^\pm)\in\cI
\quad \text{resp.}\quad 
\lim_{s\to\pm\infty} (v_0,v_1,v_2)(s,\cdot) = (x_0^\pm,x_1^\pm, x_2^\pm)\in\cI .
\end{equation}
For any $x^-,x^+\in\cI$ let us denote by 
$$
\tM^1_0(x^-,x^+) = \bigl\{ (u_0,u_2) \,\big|\, \eqref{u0 u2}, \eqref{conv},
\;{\rm Ind}(D_{(u_0,u_2)})=1 \bigr\}
$$ 
the one dimensional (i.e.\ index $1$) component of the moduli space of Floer trajectories for $(L_0\times L_2,L_{02})$.
One can achieve transversality of these moduli spaces (of any index $\leq 1$) by choosing $t$-dependent almost complex structures $J_0$ and $J_2$ that are constant near $t=0$ and $t=1$.\footnote{  
Indeed, note that the unique continuation theorem \cite[Thm.4.3]{fhs:tr} applies to the
interior of each nonconstant strip $u_i:\R\times(0,1)\to M_i$.  
It implies that the set of regular points, $(s_0,t_0)\in\R\times(0,1)$ with
$\pd_s u_i(s_0,t_0)\neq 0$ and $u_i^{-1}(u_j(\R\cup\{\pm\infty\}) , t_0
) =\{(s_0,t_0)\}$, is open and dense. These points can be used to
prove surjectivity of the linearized operator for a universal moduli
space of solutions with respect to split almost complex structures $(J_0,J_2)$. 
(The constant solutions are automatically transverse due to the previously
ensured transversality of the intersection points.) Note that it suffices to work with almost complex structures that are $t$-independent outside of $[\frac 13,\frac 23]$.
The existence of a 
%K
comeagre set of regular $(J_0,J_2)$ then follows
from the usual Sard-Smale argument as in \cite{ms:jh}.
}
Note that we cannot expect a bijection with the moduli spaces of Floer trajectories for
$(L_0\times L_{12},L_{01}\times L_2)$ as in \eqref{v0 v1 v2}.
However, by the independence theorem in \cite{quiltfloer}, the cohomology defined from the above Floer differential is isomorphic to the cohomology defined by the ``quilted Floer differential" arising from the moduli spaces
$$
\tM^1_\delta(x^-,x^+) = \bigl\{ (v_0,v_1,v_2) \,\big|\, \eqref{v0 v1 v2}_\delta, \eqref{conv}, \;{\rm Ind}(D_{(v_0, v_1,\overline{v}_2)})=1 \bigr\} 
$$ 
for any choice of $\delta>0$.
Here we consider strips $v_0,v_2$ of width $1$ as before but middle strips $v_1:\R\times[0,\delta]\to M_1$ of width $\delta>0$, and $\eqref{v0 v1 v2}_\delta$ denotes the same boundary value problem as above except for the seam condition $(v_1(s,\delta),v_2(s,0))\in L_{12}$. Moreover, we use almost complex structures $J_{0,\delta},J_{2,\delta}$
that converge to $J_0,J_2$ in the $\cC^\infty$-topology as $\delta\to 0$.
The specific choice follows from the constructions in the proof
\footnote{
Due to more technical folding, $J_{0,\delta},J_{2,\delta}$ are given by rescaling $J_0$ to $[0,1-\delta/2]$ and $J_2$ to $[\delta/2,1]$, and extending them constantly by $J_0(1)$ and $J_2(0)$ respectively.
The convergence holds since each $J_i$ is smooth and constant near $t=0,1$.
}
and will also ensure that the moduli spaces $\tM^1_\delta(x^-,x^+)$ are cut out transversely for $\delta>0$ sufficiently small.

\begin{figure}[ht]
\begin{picture}(0,0)%
\includegraphics{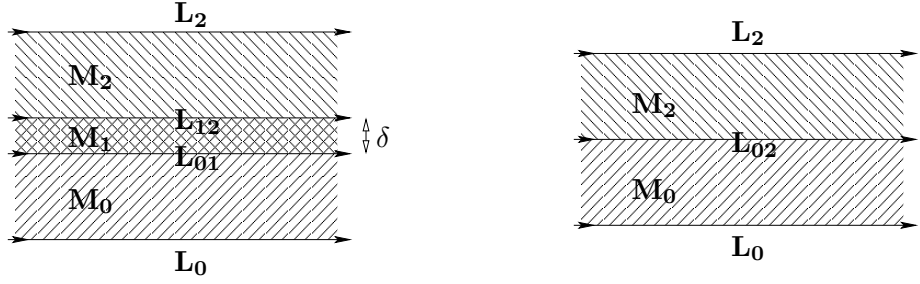}%
\end{picture}%
\setlength{\unitlength}{2200sp}%
\begingroup\makeatletter\ifx\SetFigFont\undefined%
\gdef\SetFigFont#1#2#3#4#5{%
  \reset@font\fontsize{#1}{#2pt}%
  \fontfamily{#3}\fontseries{#4}\fontshape{#5}%
  \selectfont}%
\fi\endgroup%
\begin{picture}(9549,3135)(589,-2461)
\put(2476,504){\makebox(0,0)[lb]{{$\mathbf{L_2}$}}}
\put(2476,-700){\makebox(0,0)[lb]{{$\mathbf{L_{12}}$}}}
\put(2476,-1111){\makebox(0,0)[lb]{{$\mathbf{L_{01}}$}}}
\put(2476,-2300){\makebox(0,0)[lb]{{$\mathbf{L_0}$}}}
\put(8751,-2130){\makebox(0,0)[lb]{{$\mathbf{L_0}$}}}
\put(8751,-961){\makebox(0,0)[lb]{{$\mathbf{L_{02}}$}}}
\put(8751,300){\makebox(0,0)[lb]{{$\mathbf{L_2}$}}}
\put(7626,-511){\makebox(0,0)[lb]{{$\mathbf{M_2}$}}}
\put(7626,-1486){\makebox(0,0)[lb]{{$\mathbf{M_0}$}}}
\put(1276,-1600){\makebox(0,0)[lb]{{$\mathbf{M_0}$}}}
\put(1276,-900){\makebox(0,0)[lb]{{$\mathbf{M_1}$}}}
\put(1276,-200){\makebox(0,0)[lb]{{$\mathbf{M_2}$}}}
\put(4750,-850){\makebox(0,0)[lb]{{$\delta$}}}
\end{picture}%
\caption{Shrinking the middle strip} \label{shrinkfig}
\end{figure}

In order to prove Theorem~\ref{main} it now suffices to show that the isomorphism \eqref{iso chain} of chain groups descends to cohomology for an appropriate choice of $\delta>0$. We will prove this by establishing a bijection between the Floer trajectories for $(L_0,L_{02},L_2)$ on strips of width $(1,1)$ and those for $(L_0,L_{01},L_{12},L_2)$ on strips of
width $(1,\delta,1)$ for sufficiently small width $\delta>0$ of the middle strip.  These Floer trajectories are pseudoholomorphic quilts associated to the pictures in Figure \ref{shrinkfig}.
More precisely, we will consider the (zero dimensional, compact) moduli spaces of Floer trajectories modulo $\R$-translation and prove the following.

\begin{theorem}  \label{traj bij}
Under the assumptions \eqref{b}, \eqref{c} of Theorem~\ref{main} and for all sufficiently small $\delta>0$, the moduli spaces $\tM^1_{\delta}(x^-,x^+)$ are regular and
there is a bijection
$$
\T_{\delta}:  \, 
\M^1_0(x^-,x^+) := \tM^1_0(x^-,x^+)/\R
\longrightarrow  
\tM^1_{\delta}(x^-,x^+)/\R 
=: \M^1_{\delta}(x^-,x^+) .
$$
\end{theorem}

\begin{remark}  \label{rmk:derived main}
In the situation of Theorem~\ref{main} except for assumption \eqref{d}, the constructions in this section provide naturally isomorphic chain groups 
$CF(L_0\times L_{1},L_{02})$ and $CF(L_0\times L_{12},L_{01}\times L_2)$  and well defined differentials $\partial_0$ resp.\ $\partial_\delta$ on them, defined from the moduli spaces $\M_0^1(x^-,x^+)$ and $\M_\delta^1(x^-,x^+)$.
As discussed in Section~\ref{sec:mon}, due to obstructions from disks of minimal Maslov index $2$, both differentials square to a multiple of the identity, see \cite{oh:fl1} and \cite{fieldb}. 
So we have $\partial_0^2 = w_0 {\rm Id}$ and $\partial_\delta^2 = w_\delta {\rm Id}$ for any $\delta>0$ (as long as the moduli spaces $\M_\delta^1(x^-,x^+)$ are regular).
Now Theorem~\ref{traj bij} implies that for sufficiently small $\delta>0$ and any $x\in\cI$ (viewed as generator in both chain groups) we have
$w_0 \lan x \ran = \partial_0^2 \lan x \ran = \partial_\delta^2 \lan x \ran = w_\delta \lan x \ran$, and hence $w_0=w_\delta$. (If $\cI$ is empty then both theories are trivial.)

If $w_\delta=0$ (e.g.\ by assumption \eqref{d}) or $w_0=0$ for some other reason, then this proves that both Floer cohomologies are well defined and (again by Theorem~\ref{traj bij}) are isomorphic.

For any value of $w_0=w_\delta$ this proves that there exists a canonical isomorphism
\begin{equation*} 
\bigl(CF(L_0 \times L_{12},L_{01} \times L_2), \partial_0 \bigr)
\overset{\sim}{\longrightarrow}
\bigl( CF(L_0 \times L_2, L_{01} \circ L_{12}) , \partial_\delta \bigr) 
\end{equation*}
\end{remark}
in the derived category of factorizations of $w_0{\rm Id}$.

\section{Bijection of moduli spaces under strip shrinking}
\label{shrink} 

In this section we prove Theorem \ref{traj bij}.
We start by describing the strategy of proof and introducing the relevant
notations.  First we use the assumption that $L_{01} \circ L_{12}$ is
embedded by $\pi_{02}$.  Consider a solution
$u=(u_0,u_2)\in\tM^1_0(x^-,x^+)$, that is a pair $u_0:\R\times[0,1]\to
M_0$, $u_2:\R\times[0,1]\to M_2$ of index $1$, with limits
$\lim_{s\to\pm\infty} (u_0,u_2)(s,\cdot) = x^\pm$, and satisfying
\begin{align*}
&\ol{\partial}_{J_0}u_0=0,\qquad
\ol{\partial}_{J_2}u_2=0, \\
&
u_0|_{t=0}\in L_0, 
\qquad (u_0|_{t=1},u_2|_{t=0})\in L_{02},
\qquad u_2|_{t=1}\in L_2 .
\end{align*}
We can identify $(u_0,u_2)$ with the map $u_{02}:\R\times[0,1]\to
M_0\times M_2$ given by $u_{02}(s,t)=(u_0(s,1-t),u_2(s,t))$, which
satisfies $\lim_{s\to\pm\infty} u_{02}(s,\cdot) = x^\pm$ and
\begin{align*}
\ol{\partial}_{J_{02}}u_{02}=0,\qquad
u_{02}|_{t=0}\in L_{02}, 
\qquad u_{02}|_{t=1}\in L_0\times L_2 .
\end{align*}
Here we denoted $J_{02}(s,t):=(-J_0(s,1-t),J_2(s,t))$.  We will also
denote $\bar J_{02} := J_{02}|_{t=0}$ and $\bar
u_{02}:=u_{02}|_{t=0}:\R\to L_{02}$.  
Finally, we will denote by
$(L_{01}\times L_{12})^T\subset M_0\times M_2\times M_1\times M_1$
the obvious transposition of factors.
Since $\pi_{02}:L_{01}\times_{M_1}L_{12}\to L_{02} \subset M_0 \times M_2$
is transversal and embedded, there is a unique smooth map
$\ell_1:L_{02}\to M_1$ such that
\begin{equation}\label{ell1}
(x_{02},\ell_1(x_{02}),\ell_1(x_{02})) \in ( L_{01}\times L_{12} )^T 
\qquad\forall x_{02}\in L_{02}.
\end{equation}
This provides the lift $\bar u_1:=\ell_1\circ\bar u_{02}: \R \to M_1$.
We also denote by
$\bar u:=(\bar u_{02},\bar u_1,\bar u_1)$ the extension $\R \times
[0,\delta] \to M_0\times M_2\times M_1\times M_1$ that is constant
along $[0,\delta]$.  Given $\delta$ these choices are unique, so we
can identify $u$ with the pair $(u_{02},\bar u)$.  In the same spirit
we find unique points $x^\pm_1\in M_1$ such that $(x^\pm,x^\pm_1)\in
(L_0\times L_{12})\cap ( L_{01}\times L_2 ) \subset M_0\times
M_1\times M_2$.  In this notation we have the limit
$\lim_{s\to\pm\infty} \bar u_1(s) = x_1^\pm$.  
Given $u\in\tM^1_0(x^-,x^+)$ as above and $\delta>0$ we wish to find a
corresponding $(v_0,v_1,v_2)\in\tM^1_{\delta}(x^-,x^+)$, that is a
triple $v_0:\R\times[0,1]\to M_0$, $v_1:\R\times[0,\delta]\to M_1$,
$v_2:\R\times[0,1]\to M_2$ with limits $\lim_{s\to\pm\infty}
(v_0,v_2)(s,\cdot) = x^\pm$, $\lim_{s\to\pm\infty} v_1(s,\cdot) =
x^\pm_1$, and satisfying
\begin{align*}
&\ol{\partial}_{J_{0,\delta}}v_0=0,\qquad
\ol{\partial}_{J_1}v_1=0,\qquad
\ol{\partial}_{J_{2,\delta}}v_2=0, \\
&
v_0(s,0)\in L_0, 
\quad (v_0(s,1),v_1(s,0))\in L_{01},
\quad (v_1(s,\delta),v_2(s,0))\in L_{12},
\quad v_2(s,1)\in L_2 . \nonumber
\end{align*}
Here $J_{0,\delta},J_{2,\delta}$ are given by linearly rescaling $J_0$ to $[0,1-\delta/2]$ and $J_2$ to $[\delta/2,1]$, and extending them constantly by $J_0(1)$ and $J_2(0)$ respectively.
This choice of almost complex structures is more natural in the following reformulation of the $\delta$-moduli spaces.

Let $\bdelta:=\delta/(2-\delta)$ (or
equivalently $\delta=2\bdelta/(1+\bdelta)$).  Instead of the triple
strip we consider a quadruple of maps $v=(v_{02},v_{02}',v_1,v_1')$
with $v_{02}\in\cC^\infty(\R\times[0,1],M_0\times M_2)$,
$v_{02}'\in\cC^\infty(\R\times[0,\bdelta],M_0\times M_2)$,
$v_1,v_1'\in\cC^\infty(\R\times[0,\bdelta],M_1)$ that have limits
$\lim_{s\to\pm\infty} v_{02}(s,\cdot) = \lim_{s\to\pm\infty}
v_{02}'(s,\cdot) = x^\pm$, $\lim_{s\to\pm\infty} v_1(s,\cdot) =
\lim_{s\to\pm\infty} v_1(s,\cdot)=x^\pm_1$, and satisfy
\begin{align} \label{vvvv}
&\ol{\partial}_{J_{02}}v_{02}=0,\qquad
\ol{\partial}_{-\bar J_{02}}v'_{02}=0,\qquad
\ol{\partial}_{-J_1}v_1'=0, \qquad
\ol{\partial}_{J_1}v_1=0, \nonumber\\
&
(v_{02}',v_{02})|_{t=0}\in {M_0\times M_2}, \qquad 
(v_1',v_1)|_{t=0}\in \Delta_{M_1}, \\
&
(v_{02}',v_1',v_1)|_{t=\bdelta}\in (L_{01}\times L_{12})^T, \qquad 
v_{02}|_{t=1}\in L_0\times L_2 . \nonumber
\end{align}
For notational convenience we will also group these quadruples of maps
as $v=(v_{02},\hat v)$ with $\hat v=(v_{02}',v_1,v_1')$.  Then we can
abbreviate $J=(J_{02},\hat J)$ with $\hat J:=(-\bar J_{02},-J_1,J_1)$,
and reformulate (\ref{vvvv}) as
\begin{align*}
& \ol{\partial}_{J}v := 
\bigl( \ol{\partial}_{J_{02}}v_{02} \,,\,
\ol{\partial}_{\hat J}\hat v \bigr) =0 , \\
&
(v_{02},\hat v)|_{t=0}\in \Delta_{M_0\times M_2}\times \Delta_{M_1}, 
\qquad
\hat v_{t=\bdelta}\in (L_{01}\times L_{12})^T, \qquad 
v_{02}|_{t=1}\in L_0\times L_2 . \nonumber
\end{align*}
We denote the moduli space of such solutions $v=(v_{02},\hat v)$ by
$\widehat\M^1_\bdelta(x^-,x^+)$.  It is in one-to-one correspondence
to $\tM^1_{\delta}(x^-,x^+)$ as follows: Given
$v=(v_{02},v_{02}',v_1',v_1)\in\widehat\M^1_\bdelta(x^-,x^+)$ we
obtain $\bar v=(v_0,v_1,v_2)\in\tM^1_{\delta}(x^-,x^+)$ from
\begin{align*} 
\bigl( v_0(s,1-t), v_2(s,t)\bigr)
&=\left\{\begin{array}{ll}
v_{02}'((1+\bdelta)s,\bdelta - (1+\bdelta)t)  
&\text{for}\;\; 0\le t \le \frac 12 \delta , \\
v_{02}((1+\bdelta)s,(1+\bdelta)t -\bdelta) &\text{for}\;\; \frac 12 \delta \le t \le 1 ,
\end{array}\right. \\
v_1(s,t)
&=\left\{\begin{array}{ll}
v_1'((1+\bdelta)s,\bdelta - (1+\bdelta)t) &\text{for}\;\; 0\le t \le \frac 12 \delta , \\
v_1((1+\bdelta)s,(1+\bdelta)t - \bdelta) 
&\text{for}\;\;\frac 12 \delta \le t \le \delta .
\end{array}\right. 
\end{align*}
The two different formulations for double and triple strips each are
indicated in Figure~\ref{strips}.
\begin{figure}[ht]
\begin{picture}(0,0)%
\includegraphics{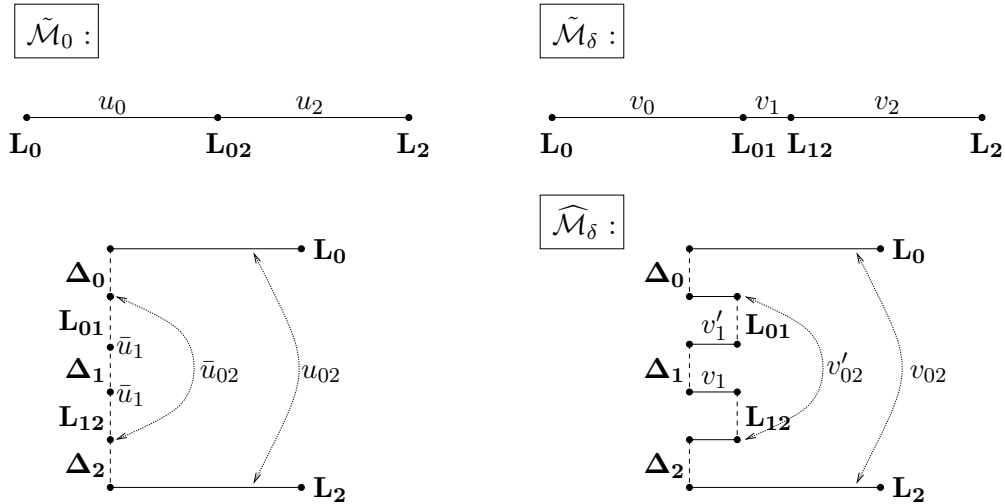}%
\end{picture}%
\setlength{\unitlength}{1973sp}%
\begingroup\makeatletter\ifx\SetFigFont\undefined%
\gdef\SetFigFont#1#2#3#4#5{%
  \reset@font\fontsize{#1}{#2pt}%
  \fontfamily{#3}\fontseries{#4}\fontshape{#5}%
  \selectfont}%
\fi\endgroup%
\begin{picture}(12309,6229)(451,-7103)
\put(10726,-5636){\makebox(0,0)[lb]{$v'_{02}$}}
\put(11776,-5636){\makebox(0,0)[lb]{$v_{02}$}}
\put(11551,-4136){\makebox(0,0)[lb]{{{$\mathbf{L_0}$}}}}
\put(11551,-7136){\makebox(0,0)[lb]{{{$\mathbf{L_2}$}}}}
\put(4126,-5636){\makebox(0,0)[lb]{{{$u_{02}$}}}}
\put(2851,-5636){\makebox(0,0)[lb]{{{$\bar u_{02}$}}}}
\put(1801,-5936){\makebox(0,0)[lb]{{{$\bar u_1$}}}}
\put(4276,-4136){\makebox(0,0)[lb]{{{$\mathbf{L_0}$}}}}
\put(4276,-7136){\makebox(0,0)[lb]{{{$\mathbf{L_2}$}}}}
\put(1801,-5336){\makebox(0,0)[lb]{{{$\bar u_1$}}}}
\put(7276,-3836){\makebox(0,0)[lb]{{{$\widehat{\mathcal M}_\delta$ :}}}}
\put(9151,-5111){\makebox(0,0)[lb]{{{$v'_1$}}}}
\put(9151,-5711){\makebox(0,0)[lb]{{{$v_1$}}}}
\put(9701,-6236){\makebox(0,0)[lb]{{{$\mathbf{L_{12}}$}}}}
\put(9701,-5086){\makebox(0,0)[lb]{{{$\mathbf{L_{01}}$}}}}
\put(8426,-4436){\makebox(0,0)[lb]{{{$\mathbf{\Delta_0}$}}}}
\put(8426,-5636){\makebox(0,0)[lb]{{{$\mathbf{\Delta_1}$}}}}
\put(8426,-6836){\makebox(0,0)[lb]{{{$\mathbf{\Delta_2}$}}}}
\put(1151,-4436){\makebox(0,0)[lb]{{{$\mathbf{\Delta_0}$}}}}
\put(1076,-5036){\makebox(0,0)[lb]{{{$\mathbf{L_{01}}$}}}}
\put(1151,-5636){\makebox(0,0)[lb]{{{$\mathbf{\Delta_1}$}}}}
\put(1076,-6236){\makebox(0,0)[lb]{{{$\mathbf{L_{12}}$}}}}
\put(1151,-6836){\makebox(0,0)[lb]{{{$\mathbf{\Delta_2}$}}}}
\put(4051,-2261){\makebox(0,0)[lb]{{{$u_2$}}}}
\put(8251,-2261){\makebox(0,0)[lb]{{{$v_0$}}}}
\put(11326,-2261){\makebox(0,0)[lb]{{{$v_2$}}}}
\put(451,-2811){\makebox(0,0)[lb]{{{$\mathbf{L_0}$}}}}
\put(7126,-2811){\makebox(0,0)[lb]{{{$\mathbf{L_0}$}}}}
\put(2926,-2811){\makebox(0,0)[lb]{{{$\mathbf{L_{02}}$}}}}
\put(5326,-2811){\makebox(0,0)[lb]{{{$\mathbf{L_2}$}}}}
\put(10226,-2811){\makebox(0,0)[lb]{{{$\mathbf{L_{12}}$}}}}
\put(12526,-2811){\makebox(0,0)[lb]{{{$\mathbf{L_2}$}}}}
\put(9526,-2811){\makebox(0,0)[lb]{{{$\mathbf{L_{01}}$}}}}
\put(1576,-2261){\makebox(0,0)[lb]{{{$u_0$}}}}
\put(7276,-1436){\makebox(0,0)[lb]{{{$\tilde{\mathcal M}_\delta$ :}}}}
\put(9826,-2261){\makebox(0,0)[lb]{{{$v_1$}}}}
\put(676,-1436){\makebox(0,0)[lb]{{{$\tilde{\mathcal M}_0$ :}}}}
\end{picture}%
\caption{Double and triple strips} \label{strips}
\end{figure}
The bijection $\T_{\delta}$ to the moduli space $\M^1_0(x^-,x^+)$
can then be established via a bijection
\begin{equation}\label{map}
\T_\bdelta: \M^1_0(x^-,x^+) \to  \M^1_{\bdelta}(x^-,x^+) 
:=\widehat\M^1_\bdelta(x^-,x^+)/\R .
\end{equation}
This map will be constructed by the implicit function theorem~\ref{solving}.
We prove injectivity in corollary~\ref{injective}, and the
surjectivity will follow from the compactness theorem~\ref{compact}.

\subsection{Implicit function theorem} \label{ift}

The purpose of this section is to construct the map
$
\T_\delta: \M^1_0(x^-,x^+) \to  \M^1_{\delta}(x^-,x^+) 
$
of Theorem~\ref{traj bij}.
We will do this by constructing the map \eqref{map},
with $\bdelta$ replaced by $\delta$,
from the following implicit function theorem.

\begin{theorem}\label{solving}
There exist constants $C_0$, $\eps>0$, and $\delta_0>0$ such that 
the following holds for every $\delta\in(0,\delta_0]$.
For every $u\in\tM^1_0(x^-,x^+)$ there exists a unique
$v_u\in\widehat\M^1_\delta(x^-,x^+)$ such that
$v_u=e_u(\xi)$ with $\xi\in \Gamma_{1,\delta}(\eps)\cap K_0$.
The solution moreover satisfies
\begin{equation} \label{sqrtd}
\|\xi\|_{H^2_{1,\delta}}\leq C_0 \delta^{\frac 14} .
\end{equation}
\end{theorem}

Here $e_u(\xi):=(v_{02},v_{02}',v_1',v_1)$ 
is given in terms of $u=(u_{02},\bar u)$ and
$\xi=(\xi_{02},\hat\xi)$ with 
$\xi_{02}\in\Gamma(u_{02}^*T (M_0\times M_2))$
and 
$\hat\xi=(\xi_{02}',\xi_1',\xi_1)\in\Gamma(\bar u^*T(M_0\times M_2\times M_1\times M_1))$.
The precise definitions of the exponential map $e_u$, 
the $\eps$-ball $\Gamma_{1,\delta}(\eps)$, the $H^2_{1,\delta}$-norm,
and the local slice $K_0$ of the $\R$-shift symmetry
will be given in the process of the proof.

To prove the theorem we fix a solution $u\in\tM^1_0(x^-,x^+)$, 
and in the following will allow all constants to depend on $u$
up to translation in $\R$.
(Since $\M^1_0(x^-,x^+)$ is finite we can then 
easily find uniform constants $C_0$ and $\delta_0>0$.)
We will then roughly solve $\ol{\partial}_{J}e_u(\xi)=0$ 
for sections $\xi=(\xi_{02},\hat\xi)$,
$\hat\xi=(\xi_{02}',\xi_1',\xi_1)$ satisfying the boundary conditions
\begin{align}\label{LL}
(\xi_{02}',\xi_{02})|_{t=0}\in 
T_{(\bar u_{02},\bar u_{02})}\Delta_{M_0\times M_2} , &\qquad 
(\xi_1',\xi_1)|_{t=0}\in T_{(\bar u_1,\bar u_1)}\Delta_1 , \\
\hat\xi|_{t=\delta}=(\xi_{02}',\xi_1',\xi_1)|_{t=\delta}
\in T_{\bar u}(L_{01}\times L_{12})^T , & \qquad 
\xi_{02}|_{t=1}\in T_{u_{02}}(L_0\times L_2). \nonumber
\end{align}
The exponential map $e_u(\xi)$ will then be constructed such that the 
nonlinear Lagrangian boundary conditions are satisfied automatically.
The index of the new solution $v_u$ will coincide with that of the
given solution $u$ due to Lemma~\ref{index}.
Here we identified $v_u$ with a solution $\bar v_u\in\tM^1_{\tilde\delta}(x^-,x^+)$,
$\tilde\delta=2\delta/(1+\delta)$.
Then the homotopy between $v_u=e_u(\xi)$ and $(u_{02},\bar u)$
induces a homotopy $\bar v_u\cong (u_0,\bar u_1,u_2)$.

To set up the implicit function theorem we introduce the space 
of $H^k$-sections over $(u_{02},\bar u)$
for $k\in\N_0$, 
$$ 
H^k_{1,\delta}:=\left\{(\eta_{02},\eta_{02}',\eta_1',\eta_1) \left|
\begin{aligned}
\eta_{02}\in H^k(\R\times[0,1],u_{02}^* T(M_0\times M_2)), \\
\eta_{02}'\in H^k(\R\times[0,\delta],\bar u_{02}^* T(M_0\times M_2)), \\
\eta_1',\eta_1\in H^k(\R\times[0,\delta],\bar u_1^* TM_1) \qquad\quad 
\end{aligned}\right\}\right..
$$
We also write these sections as $\eta=(\eta_{02},\hat\eta)\in H^k_{1,\delta}$,
where the subscripts indicate the width of the domains of $\eta_{02}$ and
$\hat\eta=(\eta_{02}',\eta_1',\eta_1)\in
H^k(\R\times[0,\delta],\bar u^* T(M_0\times M_2\times M_1\times M_1))$. 
The corresponding $H^k$-norm on this space is
\begin{align*}
&\bigl\|(\eta_{02},\eta_{02}',\eta_1',\eta_1)\bigr\|_{H^k_{1,\delta}}^2
:=
\|\eta_{02}\bigr\|_{H^k(\R\times[0,1])}^2
+ \|\hat\eta\bigr\|_{H^k(\R\times[0,\delta])}^2  \\
&\qquad =
\|\eta_{02}\bigr\|_{H^k(\R\times[0,1])}^2
+ \|\eta_{02}'\bigr\|_{H^k(\R\times[0,\delta])}^2 
+ \|\eta_1'\bigr\|_{H^k(\R\times[0,\delta])}^2 
+ \|\eta_1\bigr\|_{H^k(\R\times[0,\delta])}^2 .
\end{align*}
We denote the space of $H^2$-sections satisfying the boundary conditions by
$$
\Gamma_{1,\delta}
:=\bigl\{\xi\in H^2_{1,\delta} \,\big|\, (\ref{LL}) \bigr\} 
$$
and equip this space with the norm
$$
\bigl\|\xi\bigr\|_{\Gamma_{1,\delta}}
:=
\|\xi\bigr\|_{H^2_{1,\delta}}
+ \|\nabla\xi\bigr\|_{L^4_{1,\delta}} ,
$$
where we added the $L^4$-norm $\|\nabla(\xi_{02},\hat\xi)\|_{L^4_{1,\delta}} :=
\bigl( \|\nabla\xi_{02}\|_{L^4(\R\times[0,1])}^4
+ \|\nabla\hat\xi\|_{L^4(\R\times[0,\delta])}^4 \bigr)^{1/4}
$ 
on the multi-strip.
We denote the $\eps$-ball in $\Gamma_{1,\delta}$ by
$$
\Gamma_{1,\delta}(\eps)
:=\bigl\{\xi\in H^2_{1,\delta} \,\big|\, 
\|\xi\|_{\Gamma_{1,\delta}}<\eps , (\ref{LL}) \bigr\} .
$$
We equip the target space
$\Om_{1,\delta} := H^1_{1,\delta}$
with the norm
$$
\|\eta\|_{\Om_{1,\delta}}
:= \|\eta\|_{H^1_{1,\delta}} + \|\eta\|_{L^4_{1,\delta}} .
$$
The reason for adding the $L^4$-norms in domain and target is that we
do not have uniform Sobolev embeddings on the strips of varying width.
Instead, we build the necessary Sobolev multiplication properties into
the norms.
The definitions of all these norms also involves a choice of metric on each manifold $M_0,M_1,M_2$. Different choices yield equivalent norms.\footnote{
This remains true if some $M_i$ are noncompact, since the images of $u_{02}$, $\bar u$ are contained in compact sets.
}

Next, we make some preparations for defining an exponential map 
that is compatible with the boundary conditions (\ref{LL}). 

\begin{lemma} (Existence of compatible quadratic corrections) \label{QQ}
There exists $\eps_0>0$ and smooth families of maps (defined on the
$\eps_0$-balls)
$$
Q_s: T_{\bar u(s)}\bigl( M_0\times M_2\times M_1\times M_1 \bigr)
\supset B_{\eps_0}
\to T_{\bar u(s)}\bigl( M_0\times M_2\times M_1\times M_1 \bigr),
\qquad\forall s\in\R,
$$
$$
Q^{02}_{s,t} : T_{u_{02}(s,t)}(M_0\times M_2)\supset B_{\eps_0}^{02} 
\to T_{u_{02}(s,t)}(M_0\times M_2)
\qquad\forall (s,t)\in\R\times[0,1],
$$
that are a diffeomorphism onto their image and have the following 
properties:
\begin{description}
\item[(Quadratic)] $Q_s(0)=0$, $d Q_s(0)\equiv 0$, $Q^{02}_{s,t}(0)=0$, and
$d Q^{02}_{s,t}(0)\equiv 0$ for all $(s,t)\in\R\times[0,1]$.
In particular, there is a constant $C_Q$ such that for all
$\hat\xi\in B_{\eps_0}$ and $\xi_{02}\in B_{\eps_0}^{02}$
\begin{equation}\label{Quest}
|Q_s(\hat\xi)|\leq C_Q |\hat\xi|^2 ,\qquad
|Q_{s,t}^{02}(\xi_{02})|\leq C_Q |\xi_{02}|^2 .
\end{equation}
\item[(Linearizing $\mathbf{L_{01} \times L_{12}}$)] 
$\exp_{\bar u(s)}\circ \, ( 1+Q_s )$ maps 
$T_{\bar u(s)}(L_{01}\times L_{12})^T\cap B_{\eps_0}$
to $(L_{01}\times L_{12})^T$.
\item[(Linearizing $\mathbf{M_0 \times M_2 \times \Delta_1}$)]
$\exp_{\bar u(s)}\circ \, (1+Q_s )$ maps 
$T_{\bar u(s)}(M_0\times M_2\times \Delta_1)\cap B_{\eps_0}$
to $M_0\times M_2\times \Delta_1$.
\item[(Linearizing $\mathbf{L_{02}}$)]
$\exp_{u_{02}(s,1)}\circ \, ( 1+Q^{02}_{s,1} )$ maps 
$T_{u_{02}(s,1)}L_{02}\cap B_{\eps_0}^{02}$ to $L_{02}$.
\item[(Compatible)]
Restricting $Q_s$ to $T_{\bar u}(M_0\times M_2\times\Delta_1)$
and composing it with the projection
${\rm Pr}_{02}:T_{(\bar u_{02},\bar u_1,\bar u_1)}(M_0\times M_2\times M_1\times M_1)
\to T_{\bar u_{02}}(M_0\times M_2)$
yields a map that is independent of the $T_{(\bar u_1,\bar u_1)}\Delta_1$-component.
The resulting family
$$
Q^{02}_s: T_{\bar u_{02}(s)}(M_0\times M_2)\supset B_{\eps_0}^{02} 
\to T_{\bar u_{02}(s)}(M_0\times M_2)
$$
coincides with $Q^{02}_{s,0}$.
\end{description}
\end{lemma}

\begin{proof}
We fix $s\in\R$ and restrict the exponential map $\exp_{\bar u(s)}$ to
a geodesic ball around $0$. The subsequent constructions will depend
smoothly on $s\in\R$, which we drop from now on.  By assumption the
submanifold
$\cL_{0211}
:=\exp_{\bar u}^{-1}(L_{01}\times L_{12})^T$ in the
vector space $X:=T_{\bar u}(M_0\times M_2\times M_1\times M_1)$ is
transverse to the subspace $\Delta:=T_{\bar u}(M_0\times M_2\times
\Delta_1)$.  Their intersection $\hat\cL_{02}:=\cL_{0211}\cap\Delta$ is
diffeomorphic to the submanifold $\cL_{02}:=\exp_{\bar
u_{02}}^{-1}(L_{02})\subset T_{\bar u_{02}}(M_0\times M_2)$ by a map
$(m_0,m_2)\mapsto (m_0,m_2,m_1,m_1)$ with uniquely determined
$m_1=m_1(m_0,m_2)$.  So we have a direct sum decomposition
$$
\Delta = T_{\bar u_{02}}(M_0\times M_2) \times T_{(\bar u_1,\bar u_1)}\Delta_1
=
T_0\hat\cL_{02} \oplus \bigl( (T_0\cL_{02})^\perp \times\{0\} \bigr)
\oplus \bigl( \{0\}\times T_{(\bar u_1,\bar u_1)}\Delta_1 \bigr) .
$$
As a submanifold we can now write $\hat\cL_{02}\subset\Delta$ as the graph of a map
$\psi$ over a sufficiently small $\eps$-ball,
$$
\psi=\psi_{02}^\perp\times\psi_{11} : T_0\hat\cL_{02}\supset B_\eps \to
\bigl(T_0\cL_{02}\bigr)^\perp \times T_{(\bar u_1,\bar u_1)}\Delta_1 
$$ 
with $\psi(0)=0$ and $d\psi(0)\equiv 0$.  We moreover pick a
complement $C$ of $T_0\hat\cL_{02}\subset T_0\cL_{0211}$,
$$ T_0\cL_{0211} = C \oplus T_0\hat\cL_{02},
$$
then the transversality $X=T_0\cL_{0211} + \Delta$ implies the splitting 
\begin{equation}\label{split X}
X = 
C \oplus T_0\hat\cL_{02} 
\oplus \bigl(T_0\cL_{02}\bigr)^\perp \times\{0\}
\oplus \{0\}\times T_{(\bar u_1,\bar u_1)}\Delta_1 .
\end{equation}
We write $X \ni x=x_C+x_{02}+(x_{02}^\perp,0)+(0,x_{11})$ in this splitting
and define a map $\Psi: X\supset B_\eps \to X$ by
\begin{align*}
\Psi(x) &:= x + (\psi_{02}^\perp(x_{02}),0) + (0,\psi_{11}(x_{02})) \\
& =
x_C + x_{02} + (x_{02}^\perp + \psi_{02}^\perp(x_{02}),0) 
+ (0,x_{11}+\psi_{11}(x_{02})) .
\end{align*}
\begin{figure}[ht]
\setlength{\unitlength}{0.00043333in}
\begingroup\makeatletter\ifx\SetFigFont\undefined%
\gdef\SetFigFont#1#2#3#4#5{%
  \reset@font\fontsize{#1}{#2pt}%
  \fontfamily{#3}\fontseries{#4}\fontshape{#5}%
  \selectfont}%
\fi\endgroup%
{\renewcommand{\dashlinestretch}{30}
\begin{picture}(7388,5439)(0,-10)
\path(704.214,2881.071)(582.000,2862.000)(698.811,2821.315)
\path(582,2862)(6587,2319)
\path(6464.786,2299.929)(6587.000,2319.000)(6470.189,2359.685)
\path(3143.483,1259.299)(3162.000,1137.000)(3203.214,1253.625)
\path(3162,1137)(3447,4137)
\path(3465.517,4014.701)(3447.000,4137.000)(3405.786,4020.375)
\path(1872,4677)(387,4812)(12,642)
	(6612,12)(6972,4212)(5367,4362)
\dashline{60.000}(1857,4677)(5352,4362)
\path(4570,2795)(4548,2538)
\path(2202,5412)(2204,5412)(2207,5410)
	(2214,5408)(2224,5405)(2239,5400)
	(2260,5394)(2285,5386)(2316,5377)
	(2353,5366)(2395,5353)(2441,5340)
	(2493,5324)(2548,5308)(2606,5291)
	(2667,5274)(2730,5256)(2794,5238)
	(2858,5220)(2923,5202)(2986,5185)
	(3049,5168)(3111,5152)(3171,5137)
	(3230,5123)(3287,5109)(3342,5097)
	(3395,5085)(3446,5074)(3495,5065)
	(3543,5056)(3590,5048)(3635,5040)
	(3678,5034)(3721,5029)(3763,5024)
	(3804,5020)(3844,5017)(3884,5015)
	(3923,5013)(3963,5012)(4002,5012)
	(4041,5012)(4081,5013)(4120,5015)
	(4160,5017)(4200,5020)(4241,5024)
	(4283,5029)(4326,5034)(4369,5040)
	(4414,5048)(4461,5056)(4509,5065)
	(4558,5074)(4609,5085)(4662,5097)
	(4717,5109)(4774,5123)(4833,5137)
	(4893,5152)(4955,5168)(5018,5185)
	(5081,5202)(5146,5220)(5210,5238)
	(5274,5256)(5337,5274)(5398,5291)
	(5456,5308)(5511,5324)(5563,5340)
	(5609,5353)(5651,5366)(5688,5377)
	(5719,5386)(5744,5394)(5765,5400)
	(5780,5405)(5790,5408)(5797,5410)
	(5800,5412)(5802,5412)
\path(5802,5412)(5801,5410)(5799,5407)
	(5796,5400)(5791,5389)(5784,5374)
	(5774,5353)(5762,5328)(5747,5298)
	(5731,5263)(5713,5224)(5692,5181)
	(5671,5136)(5649,5088)(5626,5039)
	(5603,4989)(5580,4939)(5558,4889)
	(5536,4840)(5515,4792)(5494,4746)
	(5475,4701)(5457,4658)(5440,4617)
	(5424,4577)(5408,4539)(5394,4502)
	(5381,4467)(5369,4433)(5357,4399)
	(5346,4367)(5336,4335)(5327,4304)
	(5318,4273)(5310,4243)(5302,4212)
	(5294,4180)(5287,4147)(5280,4114)
	(5273,4081)(5267,4048)(5262,4013)
	(5256,3977)(5252,3941)(5247,3902)
	(5243,3863)(5239,3821)(5235,3778)
	(5231,3733)(5228,3686)(5225,3637)
	(5222,3587)(5219,3535)(5217,3482)
	(5214,3429)(5212,3376)(5210,3325)
	(5209,3274)(5207,3227)(5206,3183)
	(5205,3144)(5204,3109)(5203,3080)
	(5203,3057)(5203,3039)(5202,3026)
	(5202,3018)(5202,3014)(5202,3012)
\path(2202,5412)(2201,5410)(2199,5407)
	(2196,5400)(2191,5389)(2184,5374)
	(2174,5353)(2162,5328)(2147,5298)
	(2131,5263)(2113,5224)(2092,5181)
	(2071,5136)(2049,5088)(2026,5039)
	(2003,4989)(1980,4939)(1958,4889)
	(1936,4840)(1915,4792)(1894,4746)
	(1875,4701)(1857,4658)(1840,4617)
	(1824,4577)(1808,4539)(1794,4502)
	(1781,4467)(1769,4433)(1757,4399)
	(1746,4367)(1736,4335)(1727,4304)
	(1718,4273)(1710,4243)(1702,4212)
	(1694,4180)(1687,4147)(1680,4114)
	(1673,4081)(1667,4048)(1662,4013)
	(1656,3977)(1652,3941)(1647,3902)
	(1643,3863)(1639,3821)(1635,3778)
	(1631,3733)(1628,3686)(1625,3637)
	(1622,3587)(1619,3535)(1617,3482)
	(1614,3429)(1612,3376)(1610,3325)
	(1609,3274)(1607,3227)(1606,3183)
	(1605,3144)(1604,3109)(1603,3080)
	(1603,3057)(1603,3039)(1602,3026)
	(1602,3018)(1602,3014)(1602,3012)
\path(1602,3012)(1604,3012)(1607,3010)
	(1614,3008)(1624,3005)(1639,3000)
	(1660,2994)(1685,2986)(1716,2977)
	(1753,2966)(1795,2953)(1841,2940)
	(1893,2924)(1948,2908)(2006,2891)
	(2067,2874)(2130,2856)(2194,2838)
	(2258,2820)(2323,2802)(2386,2785)
	(2449,2768)(2511,2752)(2571,2737)
	(2630,2723)(2687,2709)(2742,2697)
	(2795,2685)(2846,2674)(2895,2665)
	(2943,2656)(2990,2648)(3035,2640)
	(3078,2634)(3121,2629)(3163,2624)
	(3204,2620)(3244,2617)(3284,2615)
	(3323,2613)(3363,2612)(3402,2612)
	(3441,2612)(3481,2613)(3520,2615)
	(3560,2617)(3600,2620)(3641,2624)
	(3683,2629)(3726,2634)(3769,2640)
	(3814,2648)(3861,2656)(3909,2665)
	(3958,2674)(4009,2685)(4062,2697)
	(4117,2709)(4174,2723)(4233,2737)
	(4293,2752)(4355,2768)(4418,2785)
	(4481,2802)(4546,2820)(4610,2838)
	(4674,2856)(4737,2874)(4798,2891)
	(4856,2908)(4911,2924)(4963,2940)
	(5009,2953)(5051,2966)(5088,2977)
	(5119,2986)(5144,2994)(5165,3000)
	(5180,3005)(5190,3008)(5197,3010)
	(5200,3012)(5202,3012)
\dashline{60.000}(5202,3012)(5202,3010)(5202,3006)
	(5202,2998)(5203,2985)(5203,2967)
	(5203,2944)(5204,2915)(5205,2880)
	(5206,2841)(5207,2797)(5209,2750)
	(5210,2699)(5212,2648)(5214,2595)
	(5217,2542)(5219,2489)(5222,2437)
	(5225,2387)(5228,2338)(5231,2291)
	(5235,2246)(5239,2203)(5243,2161)
	(5247,2122)(5252,2083)(5256,2047)
	(5262,2011)(5267,1976)(5273,1943)
	(5280,1910)(5287,1877)(5294,1844)
	(5302,1812)(5310,1781)(5318,1751)
	(5327,1720)(5336,1689)(5346,1657)
	(5357,1625)(5369,1591)(5381,1557)
	(5394,1522)(5408,1485)(5424,1447)
	(5440,1407)(5457,1366)(5475,1323)
	(5494,1278)(5515,1232)(5536,1184)
	(5558,1135)(5580,1085)(5603,1035)
	(5626,985)(5649,936)(5671,888)
	(5692,843)(5713,800)(5731,761)
	(5747,726)(5762,696)(5774,671)
	(5784,650)(5791,635)(5796,624)
	(5799,617)(5801,614)(5802,612)
\dashline{60.000}(1602,3012)(1602,3010)(1602,3006)
	(1602,2998)(1603,2985)(1603,2967)
	(1603,2944)(1604,2915)(1605,2880)
	(1606,2841)(1607,2797)(1609,2750)
	(1610,2699)(1612,2648)(1614,2595)
	(1617,2542)(1619,2489)(1622,2437)
	(1625,2387)(1628,2338)(1631,2291)
	(1635,2246)(1639,2203)(1643,2161)
	(1647,2122)(1652,2083)(1656,2047)
	(1662,2011)(1667,1976)(1673,1943)
	(1680,1910)(1687,1877)(1694,1844)
	(1702,1812)(1710,1781)(1718,1751)
	(1727,1720)(1736,1689)(1746,1657)
	(1757,1625)(1769,1591)(1781,1557)
	(1794,1522)(1808,1485)(1824,1447)
	(1840,1407)(1857,1366)(1875,1323)
	(1894,1278)(1915,1232)(1936,1184)
	(1958,1135)(1980,1085)(2003,1035)
	(2026,985)(2049,936)(2071,888)
	(2092,843)(2113,800)(2131,761)
	(2147,726)(2162,696)(2174,671)
	(2184,650)(2191,635)(2196,624)
	(2199,617)(2201,614)(2202,612)
\dashline{60.000}(2202,612)(2204,612)(2207,610)
	(2214,608)(2224,605)(2239,600)
	(2260,594)(2285,586)(2316,577)
	(2353,566)(2395,553)(2441,540)
	(2493,524)(2548,508)(2606,491)
	(2667,474)(2730,456)(2794,438)
	(2858,420)(2923,402)(2986,385)
	(3049,368)(3111,352)(3171,337)
	(3230,323)(3287,309)(3342,297)
	(3395,285)(3446,274)(3495,265)
	(3543,256)(3590,248)(3635,240)
	(3678,234)(3721,229)(3763,224)
	(3804,220)(3844,217)(3884,215)
	(3923,213)(3963,212)(4002,212)
	(4041,212)(4081,213)(4120,215)
	(4160,217)(4200,220)(4241,224)
	(4283,229)(4326,234)(4369,240)
	(4414,248)(4461,256)(4509,265)
	(4558,274)(4609,285)(4662,297)
	(4717,309)(4774,323)(4833,337)
	(4893,352)(4955,368)(5018,385)
	(5081,402)(5146,420)(5210,438)
	(5274,456)(5337,474)(5398,491)
	(5456,508)(5511,524)(5563,540)
	(5609,553)(5651,566)(5688,577)
	(5719,586)(5744,594)(5765,600)
	(5780,605)(5790,608)(5797,610)
	(5800,612)(5802,612)
\put(4300,2887){\makebox(0,0)[lb]{$\hat\cL_{02}$}}
\put(4300,4682){\makebox(0,0)[lb]{
$\cL_{0211}$
}}
\put(6500,3787){\makebox(0,0)[lb]{$\Delta$}}
\put(5650,2450){\makebox(0,0)[lb]{$T_0 \cL_{02}$}}
\put(4607,2520){\makebox(0,0)[lb]{$\psi$}}
\end{picture}
}
\caption{Construction of quadratic corrections: Writing $\hat\cL_{02}$ as graph} 
\end{figure}
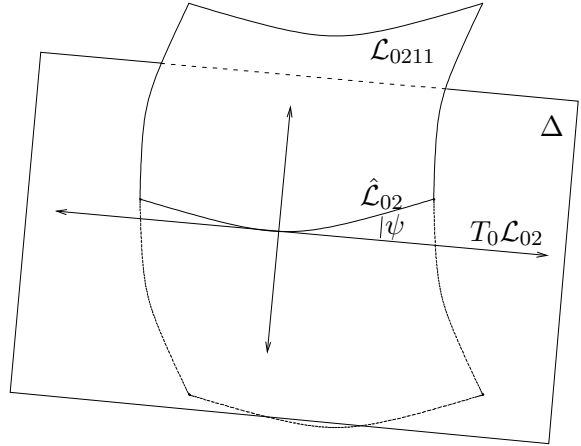
This map linearizes the intersection, 
$\Psi(T_0\hat\cL_{02})=\hat\cL_{02}$,
and we have $\Psi(0)=0$ and $d\Psi(0)={\rm Id}$.
In order to linearize the entire Lagrangian $\cL_{0211}$ we 
remark that $T_0\bigl(\Psi^{-1}(\cL_{0211})\bigr) = d\Psi(0)^{-1}T_0\cL_{0211} =T_0\cL_{0211}$.
So we can write $\Psi^{-1}(\cL_{0211})$ as graph of a map 
$$
\phi=\phi_{02}^\perp\times\phi_{11} :
 T_0\cL_{0211}\supset B_\eps \to
\bigl(T_{\bar u_{02}}\cL_{02}\bigr)^\perp \times T_{(\bar u_1,\bar u_1)}\Delta_1 
$$
with $\phi(0)=0$, $d\phi(0)\equiv 0$, 
and by the previous construction $\phi|_{T_0\hat\cL_{02}}\equiv 0$.

\begin{figure}
\setlength{\unitlength}{0.00043333in}
\begingroup\makeatletter\ifx\SetFigFont\undefined%
\gdef\SetFigFont#1#2#3#4#5{%
  \reset@font\fontsize{#1}{#2pt}%
  \fontfamily{#3}\fontseries{#4}\fontshape{#5}%
  \selectfont}%
\fi\endgroup%
{\renewcommand{\dashlinestretch}{30}
\begin{picture}(7388,5439)(0,-10)
\path(704.214,2881.071)(582.000,2862.000)(698.811,2821.315)
\path(582,2862)(6587,2319)
\path(6464.786,2299.929)(6587.000,2319.000)(6470.189,2359.685)
\path(1872,4677)(387,4812)(12,642)
	(6612,12)(6972,4212)(5367,4362)
\dashline{60.000}(1857,4677)(5352,4362)
\path(1602,4812)(5202,4812)(5202,912)
	(1602,912)(1602,4812)
\path(1614,2773)(1614,4833)(5224,4823)(5204,2443)
\dashline{60.000}(5194,2463)(5194,903)(1634,923)(1604,2783)
\path(2202,5412)(2204,5412)(2207,5412)
	(2214,5412)(2224,5412)(2239,5412)
	(2260,5412)(2285,5412)(2316,5412)
	(2353,5412)(2395,5412)(2441,5412)
	(2493,5412)(2548,5412)(2606,5412)
	(2667,5412)(2730,5412)(2794,5412)
	(2858,5412)(2923,5412)(2986,5412)
	(3049,5412)(3111,5412)(3171,5412)
	(3230,5412)(3287,5412)(3342,5412)
	(3395,5412)(3446,5412)(3495,5412)
	(3543,5412)(3590,5412)(3635,5412)
	(3678,5412)(3721,5412)(3763,5412)
	(3804,5412)(3844,5412)(3884,5412)
	(3923,5412)(3963,5412)(4002,5412)
	(4041,5412)(4081,5412)(4120,5412)
	(4160,5412)(4200,5412)(4241,5412)
	(4283,5412)(4326,5412)(4369,5412)
	(4414,5412)(4461,5412)(4509,5412)
	(4558,5412)(4609,5412)(4662,5412)
	(4717,5412)(4774,5412)(4833,5412)
	(4893,5412)(4955,5412)(5018,5412)
	(5081,5412)(5146,5412)(5210,5412)
	(5274,5412)(5337,5412)(5398,5412)
	(5456,5412)(5511,5412)(5563,5412)
	(5609,5412)(5651,5412)(5688,5412)
	(5719,5412)(5744,5412)(5765,5412)
	(5780,5412)(5790,5412)(5797,5412)
	(5800,5412)(5802,5412)
\path(5802,5412)(5801,5410)(5799,5407)
	(5796,5400)(5791,5389)(5784,5373)
	(5774,5352)(5762,5326)(5747,5295)
	(5731,5260)(5713,5220)(5692,5176)
	(5671,5129)(5649,5079)(5626,5028)
	(5603,4976)(5580,4924)(5558,4872)
	(5536,4820)(5515,4770)(5494,4721)
	(5475,4673)(5457,4627)(5440,4582)
	(5424,4538)(5408,4496)(5394,4455)
	(5381,4415)(5369,4376)(5357,4338)
	(5346,4300)(5336,4263)(5327,4225)
	(5318,4188)(5310,4150)(5302,4112)
	(5295,4077)(5289,4042)(5283,4007)
	(5277,3971)(5272,3934)(5267,3895)
	(5262,3856)(5257,3815)(5253,3773)
	(5249,3729)(5245,3684)(5242,3636)
	(5238,3586)(5235,3534)(5232,3480)
	(5229,3424)(5227,3366)(5224,3305)
	(5222,3243)(5219,3179)(5217,3114)
	(5215,3049)(5213,2983)(5212,2918)
	(5210,2854)(5209,2792)(5207,2733)
	(5206,2677)(5205,2627)(5204,2581)
	(5204,2540)(5203,2506)(5203,2477)
	(5203,2454)(5202,2437)(5202,2425)
	(5202,2418)(5202,2414)(5202,2412)
\path(2202,5412)(2201,5410)(2199,5407)
	(2196,5400)(2191,5389)(2184,5373)
	(2174,5353)(2162,5327)(2147,5297)
	(2131,5261)(2113,5222)(2092,5179)
	(2071,5132)(2049,5084)(2026,5034)
	(2003,4983)(1980,4931)(1958,4880)
	(1936,4830)(1915,4781)(1894,4733)
	(1875,4687)(1857,4642)(1840,4599)
	(1824,4558)(1808,4518)(1794,4479)
	(1781,4441)(1769,4404)(1757,4369)
	(1746,4334)(1736,4299)(1727,4265)
	(1718,4230)(1710,4196)(1702,4162)
	(1695,4129)(1688,4096)(1682,4063)
	(1676,4029)(1670,3994)(1665,3959)
	(1660,3922)(1656,3885)(1651,3846)
	(1647,3805)(1643,3763)(1640,3719)
	(1636,3673)(1633,3625)(1630,3575)
	(1627,3523)(1624,3469)(1622,3414)
	(1619,3357)(1617,3299)(1615,3241)
	(1613,3183)(1611,3125)(1610,3069)
	(1608,3015)(1607,2964)(1606,2916)
	(1605,2873)(1604,2835)(1603,2802)
	(1603,2775)(1603,2753)(1602,2737)
	(1602,2725)(1602,2718)(1602,2714)(1602,2712)
\dashline{60.000}(5202,2412)(5202,2410)(5202,2405)
	(5203,2396)(5203,2382)(5204,2364)
	(5205,2340)(5207,2312)(5208,2280)
	(5211,2245)(5213,2209)(5216,2171)
	(5219,2133)(5223,2096)(5227,2060)
	(5231,2024)(5236,1990)(5241,1957)
	(5246,1926)(5252,1895)(5259,1864)
	(5266,1834)(5274,1804)(5282,1774)
	(5292,1743)(5302,1712)(5310,1687)
	(5319,1662)(5329,1637)(5339,1610)
	(5350,1582)(5361,1553)(5374,1523)
	(5387,1491)(5402,1458)(5417,1422)
	(5434,1385)(5452,1345)(5471,1304)
	(5491,1260)(5512,1214)(5535,1167)
	(5558,1118)(5582,1068)(5606,1017)
	(5630,966)(5654,916)(5677,868)
	(5700,822)(5720,779)(5739,741)
	(5755,707)(5769,679)(5780,656)
	(5789,639)(5795,626)(5799,618)
	(5801,614)(5802,612)
\dashline{60.000}(1602,3012)(1602,3010)(1602,3006)
	(1602,2998)(1603,2985)(1603,2967)
	(1603,2944)(1604,2915)(1605,2880)
	(1606,2841)(1607,2797)(1609,2750)
	(1610,2699)(1612,2648)(1614,2595)
	(1617,2542)(1619,2489)(1622,2437)
	(1625,2387)(1628,2338)(1631,2291)
	(1635,2246)(1639,2203)(1643,2161)
	(1647,2122)(1652,2083)(1656,2047)
	(1662,2011)(1667,1976)(1673,1943)
	(1680,1910)(1687,1877)(1694,1844)
	(1702,1812)(1710,1781)(1718,1751)
	(1727,1720)(1736,1689)(1746,1657)
	(1757,1625)(1769,1591)(1781,1557)
	(1794,1522)(1808,1485)(1824,1447)
	(1840,1407)(1857,1366)(1875,1323)
	(1894,1278)(1915,1232)(1936,1184)
	(1958,1135)(1980,1085)(2003,1035)
	(2026,985)(2049,936)(2071,888)
	(2092,843)(2113,800)(2131,761)
	(2147,726)(2162,696)(2174,671)
	(2184,650)(2191,635)(2196,624)
	(2199,617)(2201,614)(2202,612)
\dashline{60.000}(2202,612)(2204,612)(2207,612)
	(2214,612)(2224,612)(2239,612)
	(2260,612)(2285,612)(2316,612)
	(2353,612)(2395,612)(2441,612)
	(2493,612)(2548,612)(2606,612)
	(2667,612)(2730,612)(2794,612)
	(2858,612)(2923,612)(2986,612)
	(3049,612)(3111,612)(3171,612)
	(3230,612)(3287,612)(3342,612)
	(3395,612)(3446,612)(3495,612)
	(3543,612)(3590,612)(3635,612)
	(3678,612)(3721,612)(3763,612)
	(3804,612)(3844,612)(3884,612)
	(3923,612)(3963,612)(4002,612)
	(4041,612)(4081,612)(4120,612)
	(4160,612)(4200,612)(4241,612)
	(4283,612)(4326,612)(4369,612)
	(4414,612)(4461,612)(4509,612)
	(4558,612)(4609,612)(4662,612)
	(4717,612)(4774,612)(4833,612)
	(4893,612)(4955,612)(5018,612)
	(5081,612)(5146,612)(5210,612)
	(5274,612)(5337,612)(5398,612)
	(5456,612)(5511,612)(5563,612)
	(5609,612)(5651,612)(5688,612)
	(5719,612)(5744,612)(5765,612)
	(5780,612)(5790,612)(5797,612)
	(5800,612)(5802,612)
\put(6500,3787){\makebox(0,0)[lb]{$\Delta$}}
\put(2000,2722){\makebox(0,0)[lb]{$\scriptstyle{T_0 \Psi^{-1}(\hat\cL_{02})=\Psi^{-1}(\hat\cL_{02})}$}}
\put(4100,5000){\makebox(0,0)[lb]{$\Psi^{-1}\cL_{0112}$}}
\put(2909,4300){\makebox(0,0)[lb]{$T_0 \Psi^{-1}\cL_{0112}$}}
\end{picture}
}
\caption{Linearizing $\L_{0112}$ compatibly with $\hat\L_{02}$ } 
\end{figure}
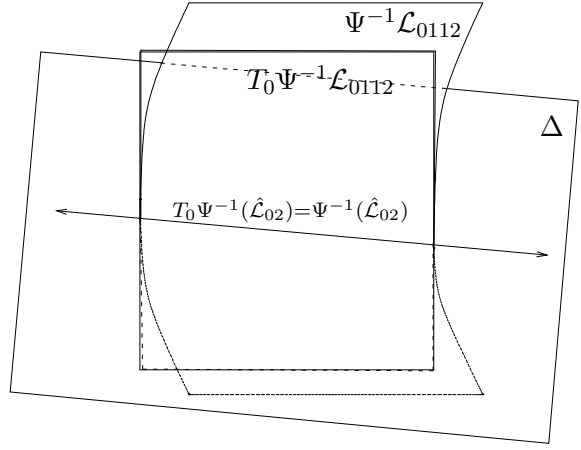
%
%curved.fig
%

Finally we define the entire linearization $\Phi: X\supset B_\eps \to X$ by
\begin{align*}
\Phi(x)
&:= \Psi\bigl( x + (\phi_{02}^\perp(x_C+x_{02}),0) + (0,\phi_{11}(x_C+x_{02})) \bigr) 
\end{align*}
for $x=x_C+x_{02}+(x_{02}^\perp,0)+(0,x_{11})$ in the splitting (\ref{split X}).
Now $Q_s:=\Phi - {\rm Id}$ is quadratic and linearized $(L_{01}\times L_{12})^T$ by construction. 
Explicitly, we have
\begin{equation}\label{Q}
Q_s(x) =
\bigl( \psi_{02}^\perp(x_{02}) + \phi_{02}^\perp(x_C+x_{02}) ,
 \psi_{11}(x_{02}) + \phi_{11}(x_C+x_{02}) \bigr) .
\end{equation}
The construction moreover ensures that $Q_s$ linearizes $M_0\times M_2\times\Delta_1$, 
that is $\Phi(\Delta)\subset\Delta$, 
since $x\in\Delta=\{x_C=0\}$ is mapped to
$\Phi(x)= x + \bigl(\psi_{02}^\perp(x_{02}),\psi_{11}(x_{02})\bigr) \in\Delta$.

To construct $Q_s^{02}$ compatible with $Q_s$ note that for 
$x=(m_0,m_2,m_1,m_1)\in T_{\bar u}(M_0\times M_2\times\Delta_1)\subset X$
we have a splitting
$$
x=(m_0,m_2,0,0) + (0,0,m_1,m_1)
=
x_C+x_{02}+(x_{02}^\perp,0)+(0,x_{11}+(m_1,m_1)),
$$
where $x_C, x_{02}, x_{02}^\perp, x_{11}$ only depend on $(m_0,m_2)$.
With this we can see in (\ref{Q}) that indeed 
$Q_s(m_0,m_2,m_1,m_1)$ is independent of $m_1$.
We then simply define
$Q^{02}_{s,0}(m_0,m_2):={\rm Pr}_{02}Q_s(m_0,m_2,0,0)$.
Moreover, a graph construction as above provides a map
$Q^{02}_{s,1} : T_{u_{02}(s,1)}(M_0\times M_2)\supset B_\eps^{02} 
\to T_{u_{02}(s,1)}(M_0\times M_2)$
that is quadratic and linearizes $L_{02}$.
Now the two families $Q^{02}_{s,0}$ and $Q^{02}_{s,1}$ can easily
be interpolated by the smooth family 
$Q^{02}_{s,t}:=(1-t)Q^{02}_{s,0} + t Q^{02}_{s,1}$
of quadratic maps.
\end{proof}

With these quadratic corrections we can now define the exponential map
$e_u$ by $e_u(\xi):=( e_{u_{02}}(\xi_{02}) , e_{\bar u}(\hat\xi) )$
for $\xi=(\xi_{02},\hat\xi)\in\Gamma_{1,\delta}(\eps)$, where
\begin{align} \label{expu}
e_{u_{02}}(\xi_{02}):=\exp_{u_{02}}\circ \, (1+Q^{02}) (\xi_{02}), \qquad
e_{\bar u}(\hat\xi):=\exp_{\bar u}\circ \, (1+Q)(\hat\xi).
\end{align}
Note that we have the usual properties of an exponential map,
$$
e_u(0)=(u_{02},\bar u),\qquad
d e_u(0) = {\rm Id}.
$$
To define $e_u$ on $\Gamma_{1,\delta}(\eps)$ 
the $\eps>0$ should be chosen such that
$\|\xi_{02}\|_{\cC^0}$ and $\|\hat\xi\|_{\cC^0}$ are sufficiently small
for the quadratic corrections in Lemma~\ref{QQ} to be defined.\footnote{
If some $M_i$ are noncompact, in particular the interior of a compact manifold with boundary and corners as in \cite{se:bo}, then the choice of $\eps>0$ also ensures that the exponential map at $u_{02}$ resp.\ $\bar u$ is well defined. This is always possible with a uniform $\eps>0$ since the images of $u_{02}$ and $\bar u$ are contained in compact subsets.
}
Lemma~\ref{Sobolev} below ensures that we can pick
a uniform $\eps>0$ for all $\delta>0$.
Now solutions $v_u\in\widehat\M^1_\delta(x^-,x^+)$ 
in a neighborhood of $u$ correspond to zeroes of the map 
$\F_u : \Gamma_{1,\delta}(\eps) \to \Om_{1,\delta}$
given by
$$
\F_u(\xi):= \bigl(\,
\Phi_{u_{02}}(\xi_{02})^{-1} ( \ol{\partial}_{J_{02}} e_{u_{02}}(\xi_{02}) )  \,,\,
\Phi_{\bar u}(\hat\xi)^{-1} ( \ol{\partial}_{\hat J} e_{\bar u}(\hat\xi) ) \,\bigr).
$$
Here $\Phi_{u}(\xi)$ denotes the parallel transport
$T_{u}M\to T_{e_{u}(\xi)}M$ along the path $\tau\mapsto e_u(\tau\xi)$.
For $\Phi_{u_{02}}$ this parallel transport on $T(M_0\times M_2)$ can
simply use the Levi-Civita connection.
\label{connection}
In the definition of $\Phi_{\bar u}$ we however use
a Hermitian connection $\tilde\nabla$ on the tangent bundle
$T(M_0\times M_2\times M_1\times M_1)$ that leaves $\hat J$ invariant.
This can be done by the same construction as in 
\cite[Proposition 3.1.1]{ms:jh}, which brings the linearized 
operator into simple form.

Next, we introduce projections related to the various Lagrangians:
$$ 
\pi_{0211}^\perp 
\in{\rm Aut}\big( \cC^\infty(\R,\bar u^* T(M_0\times M_2\times M_1\times M_1)) \bigr) ,
$$
$$ 
\pi_{02} \;,\; \pi_{02}^\perp
\;\in{\rm Aut}\big( \cC^\infty(\R,\bar u_{02}^* T(M_0\times M_2)) \bigr)
$$
are linear operators, given by pointwise orthogonal projection onto the subspaces
$(T(L_{01}\times L_{12})^T)^\perp \subset T(M_0\times M_2\times M_1\times M_1)$
resp.\ 
$T L_{02}, (T L_{02})^\perp \subset T(M_0\times M_2)$.
The following lemma contains the estimates resulting from the transversality assumption.

\begin{lemma} (Quantitative transversality)  \label{lltrans}
There exists a constant $C$ such that the following holds.
\begin{enumerate}
\item
For every $\hat x = (x_0,x_2,x_1',x_1)\in M_0\times M_2\times M_1\times M_1$ 
\begin{align*}
d ( (x_0,x_2), L_{02} ) \leq C \bigl( d( \hat x , (L_{01}\times L_{12})^T ) + d(x_1', x_1 ) \bigr).
\end{align*}
\item
For every $s\in\R$ and
$\hx=(\xi_{02}',\xi_1',\xi_1)\in T_{\bu(s)}(M_0\times M_2\times M_1\times M_1)$ 
\begin{align*}
| \hx | 
&\leq C \bigl( | \pi_{02}\xi_{02}' | + \bigl| \xi_1'-\xi_1 \bigr|
+ \bigl| \pi_{0211}^\perp \hx \bigr| \bigr) , \\
| \pi_{02}^\perp\xi_{02}' |
&\leq 
C \bigl( | \pi_{0211}^\perp \hat\xi | + | \xi_1'-\xi_1 | \bigr).
\end{align*}
\item
For every 
$\hx\in\cC^\infty(\R,\bar u^* T(M_0\times M_2\times M_1\times M_1))$
\begin{align*} 
\Vert \hat\xi \Vert_{H^1(\R)}
&\leq 
C \bigl(\Vert \pi_{02} \xi_{02}' \Vert_{H^1(\R)}
+ \Vert \xi_1'-\xi_1 \Vert_{H^1(\R)} 
+ \Vert \pi_{0211}^\perp \hat\xi \Vert_{H^1(\R)} \bigr) ,
\end{align*} 
and the same holds with $H^1$ replaced by $\cC^1$ or $L^p$ for any $p\geq 1$. 
Moreover, 
\begin{align*} 
\Vert \pi_{02}^\perp \xi_{02}'  \Vert_{L^2(\R)}
&\leq 
C\bigl( \Vert \pi_{0211}^\perp \hat\xi  \Vert_{L^2(\R)}
+ \Vert \xi_1'-\xi_1 \Vert_{L^2(\R)} \bigr) , \\
\Vert \pi_{02}^\perp \xi_{02}' \Vert_{H^1(\R)}
&\leq 
C\bigl( \Vert \pi_{0211}^\perp \hat\xi \Vert_{H^1(\R)}
+ \Vert \xi_1'-\xi_1 \Vert_{H^1(\R)} 
+ \bigl\| |\pd_s\bu|\cdot|\hat\xi| \bigl\|_{L^2(\R)} \bigr) .
\end{align*}
\end{enumerate}
\end{lemma}

\begin{proof}
We fix metrics on each $M_j$ and use the induced split metrics on both $M_0\times M_2$ and $M_0\times M_2\times M_1\times M_1$.
Towards (a) note that we evidently have $d ( (x_0,x_2), L_{02} ) \leq d( (x_0,x_2,x_1',x_1) , \hat L_{02} )$ for $\hat L_{02}:= (L_{01}\times L_{12})^T \cap M_0\times M_2 \times \Delta_1$.
Moreover, we can estimate with some constant $C$
$$
d ( \hat x, \hat L_{02} ) \leq C \bigl( d( \hat x , (L_{01}\times L_{12})^T ) + d(\hat x , M_0\times M_2 \times \Delta_1) \bigr) ,
$$
since the Lagrangian $(L_{01}\times L_{12})^T$ intersects $M_0\times M_2\times\Delta_1$ transversally. (In more detail this follows from the linear theory below; if $\hat x$ does not lie in an exponential neighbourhood of $\hat L_{02}$, then both sides of the inequality are bounded away from zero, hence the quotient attains a positive minimum on the complement of the exponential neighbourhood.\footnote{
This remains true when some of the $M_i$ are noncompact. Indeed, we only need to consider the case of $d ( (x_0^\nu,x_2^\nu), L_{02} )\to \infty$ and simply note that $L_{02}$ as well as the projection ${\rm pr}_{02}(\hat L_{02})\subset M_0\times M_2$ are compact subsets. Hence we can bound $d ( (x_0^\nu,x_2^\nu), L_{02} ) \leq d ( (x_0^\nu,x_2^\nu), {\rm pr}_{02}(\hat L_{02} ) ) + D \leq d( \hat x^\nu, (L_{01}\times L_{12})^T ) + D$ with a finite constant $D= d(L_{02}, {\rm pr}_{02}(\hat L_{02} ))$ and obtain
$\frac{d( \hat x^\nu, (L_{01}\times L_{12})^T )}{d ( (x_0^\nu,x_2^\nu), L_{02} )} 
\geq 1 - \frac{D}{d ( (x_0^\nu,x_2^\nu), L_{02} )}  \to 1$.
}) 
This proves (a) since $d(\hat x , M_0\times M_2 \times \Delta_1)$ is bounded by $d(x_1',x_1)$.
To approach (b) note moreover that $\hat L_{02}$ injects to $L_{02}\subset M_0\times M_2$.
So at every point of $\hat L_{02}$ we have a decomposition
$T(M_0\times M_2\times M_1\times M_1) =  T\hat L_{02} \oplus ( T\hat L_{02} )^\perp$,
where we can change the first factor to $TL_{02}\times\{0\}$.
On the other hand, the transverse intersection implies 
\begin{equation} \label{transint}
( T\hat L_{02} )^\perp = \bigl(\{0\} \times (T\Delta_1)^\perp \bigr)
\oplus T(L_{01} \times L_{12})^\perp ,\end{equation}
so we obtain a splitting
\begin{equation} \label{split0211}
T(M_0\times M_2\times M_1\times M_1) 
= \bigl( TL_{02}\times\{0\}\bigr) \oplus \bigl(\{0\} \times (T\Delta_1)^\perp \bigr)
\oplus T(L_{01} \times L_{12})^\perp .
\end{equation}
This means that the product of the three orthogonal projections onto
the factors defines an isomorphism.  The norm of this isomorphism is
bounded at each $\bar u(s)\in\hat L_{02}$, so for every
$\hx=(\xi_{02}',\xi_1',\xi_1)\in T_{\bu(s)}(M_0\times M_2\times
M_1\times M_1)$ we have
$$ 
| \hx | \leq C \bigl( | \pi_{02}\xi_{02}' | + \bigl| \xi_1'-\xi_1 \bigr|
+ \bigl| \pi_{0211}^\perp \hx \bigr| \bigr)
$$
with a uniform constant $C$ as claimed in (b).
(Here the projection onto $(T\Delta_1)^\perp$ is given by
$(\xi_{02}',\xi_1',\xi_1)\mapsto \frac 12 (\xi_1'-\xi_1, \xi_1-\xi_1')$.)
Moreover, the splitting \eqref{split0211} commutes with
$$
T(M_0\times M_2) = TL_{02} \oplus ( TL_{02} )^\perp
$$
via the canonical projection on the left hand side, and on the right hand side the identity on 
$TL_{02}$ combined with a bounded map 
$\bigl(\{0\} \times (T\Delta_1)^\perp \bigr)
\oplus T(L_{01} \times L_{12})^\perp \to TL_{02} \oplus ( TL_{02} )^\perp$. This implies that
$$ 
| \pi_{02}^\perp\xi_{02}' |
\leq C\bigl( | \pi_{0211}^\perp \hat\xi | + | \xi_1'-\xi_1 | \bigr)
$$
with another uniform constant $C$. This proves (b).
For $\hx\in\cC^\infty(\R,\bar u^* T(M_0\times M_2\times M_1\times M_1))$
we can then apply the pointwise estimates to $\hx(s)$ 
and integrate over $s\in\R$ to obtain for any $p\geq 1$ including $p=\infty$
\begin{align} \label{previous} 
\Vert \hat\xi \Vert_{L^p(\R)}
&\leq 
C \bigl(\Vert \pi_{02} \xi_{02}' \Vert_{L^p(\R)}
+ \Vert \xi_1'-\xi_1 \Vert_{L^p(\R)} 
+ \Vert \pi_{0211}^\perp \hat\xi \Vert_{L^p(\R)} \bigr),  \\
\Vert \pi_{02}^\perp \xi_{02}' \Vert_{L^p(\R)}
&\leq 
C\bigl( \Vert \pi_{0211}^\perp \hat\xi \Vert_{L^p(\R)}
+ \Vert \xi_1'-\xi_1 \Vert_{L^p(\R)} \bigr) . \nonumber
\end{align}
In order to prove the $H^1$- and $\cC^1$-estimates we also 
apply the pointwise estimates to $\nabla_s\hx(s)$,
\begin{align*} 
| \nabla_s\hx | 
&\leq C \bigl( | \pi_{02}(\nabla_s\xi_{02}') | 
+ \bigl| \nabla_s\xi_1'-\nabla_s\xi_1 \bigr|
+ \bigl| \pi_{0211}^\perp(\nabla_s\hx) \bigr| \bigr), \\
| \pi_{02}^\perp(\nabla_s\xi_{02}') |
&\leq 
C\bigl( | \pi_{0211}^\perp (\nabla_s\hat\xi) | + | \nabla_s\xi_1'-\nabla_s\xi_1 | \bigr).
\end{align*} 
Here we will need the inequalities
\begin{align*} 
| \pi_{02}( \nabla_s\xi_{02}' ) | 
&\leq C\bigl( | \nabla_s (\pi_{02}(\xi_{02}') ) | + | \hat\xi | \bigr), \\
| \pi_{0211}^\perp( \nabla_s\hat\xi ) | 
&\leq C\bigl( | \nabla_s (\pi_{0211}^\perp(\hat\xi) ) | 
+ |\pd_s\bu|\cdot| \hat\xi | \bigr), \\
| \nabla_s (\pi_{02}^\perp(\xi_{02}') ) | 
&\leq C \bigl( | \pi_{02}^\perp (\nabla_s\xi_{02}' ) |  
+ |\pd_s\bu|\cdot|\hat\xi | \bigr) .
\end{align*} 
The first inequality (and similarly the others) 
can be seen by writing $\xi_{02}'$ in a local orthonormal frame
given by $(\gamma_i(s))_{i=1,\dots,k}\in \bu_{02}(s)^*T L_{02}$
and $(\eta_i(s))_{i=1,\dots,K}\in \bu_{02}(s)^*(T L_{02})^\perp$.
Writing $\hx = \sum \lambda^i\gamma_i + \sum \mu^i \eta_i$ we have
\begin{align*}
\bigl| \pi_{02}( \nabla_s\xi_{02}' ) - \nabla_s ( \pi_{02}(\xi_{02}') ) \bigr| 
& = \Bigl| \sum \lambda^i \bigl( \pi_{02}(\nabla_s\gamma_i) - \nabla_s \gamma_i \bigr)
 + \sum \mu^i \pi_{02}\nabla_s( \eta_i) \Bigr|  \\
&\leq C |\pd_s\bu_{02}|\cdot|\xi_{02}' | .
\end{align*}
Note here that $\nabla_s \gamma_i=\nabla_{\pd_s \bu_{02}} \gamma_i$ and 
$\nabla_s\eta_i=\nabla_{\pd_s \bu_{02}}\eta_i$ are uniformly bounded.
Putting things together we obtain the first estimate in (c) with an extra
$\|\hx\|_{L^2(\R)}$ or $\|\hx\|_{\cC^0(\R)}$ on the right hand side, for 
which we can use (\ref{previous}).
For the last estimate in (c) we obtain
\begin{align*} 
\| \nabla_s ( \pi_{02}^\perp \xi_{02}' ) \|_{L^2(\R)} 
&\leq C \bigl( \| \nabla_s ( \pi_{0211}^\perp \hat\xi ) \|_{L^2(\R)}
+ \| \nabla_s\xi_1'-\nabla_s\xi_1 \|_{L^2(\R)} 
+  \bigl\| |\pd_s\bu|\cdot|\hat\xi| \bigl\|_{L^2(\R)} \bigr) .
\end{align*} 
This finishes the proof of (c).
\end{proof}

The following lemma contains a Sobolev estimate with a constant
independent of the width $\delta$ of the middle strip; here the
transversality assumption is used in a crucial way.

\begin{lemma}\label{Sobolev} (Uniform Sobolev Estimate)
There is a constant $C_S$ such that for all $\delta\in(0,1]$ and
$\xi=(\xi_{02},\hat\xi)\in H^2_{1,\delta}$
\begin{align*}
& \| \xi_{02} \|_{\cC^0([0,1],H^1(\R))}
+ \| \hat\xi \|_{\cC^0([0,\delta],H^1(\R))} \\
& \leq C_S\bigl( \|\xi\|_{H^2_{1,\delta}} 
+\| (\xi_{02}- \xi'_{02})|_{t=0} \|_{H^1(\R)}
+ \Vert (\xi_1-\xi_1')|_{t=0} \Vert_{H^1(\R)} 
+ \Vert \pi_{0211}^\perp\hx|_{t=\delta} \Vert_{H^1(\R)}
\bigr).
\end{align*}
In particular, for all $p>2$ including $p=\infty$
and for $\xi\in \Gamma_{1,\delta}$
satisfying the boundary conditions (\ref{LL}),
$$
\| \xi_{02} \|_{L^p(\R\times[0,1])}
+ \| \hat\xi \|_{L^p(\R\times[0,\delta])}
\leq C_S \|\xi\|_{H^2_{1,\delta}} .
$$
\end{lemma}

\begin{proof}
The $\cC^0$- and $L^p$-estimates will follow from the 
continuous embeddings $H^1(\R)\hookrightarrow\cC^0(\R)$
and $H^1(\R)\hookrightarrow L^p(\R)$ for $p\geq 2$.
So it suffices to suppose by contradiction that there are 
sequences $\delta^\nu>0$ and $\xi^\nu\in H^2_{1,\delta^\nu}$ with 
$\| \xi_{02}^\nu \|_{\cC^0([0,1],H^1(\R))}
+ \| \hat\xi^\nu \|_{\cC^0([0,\delta^\nu],H^1(\R))}=1$
but 
$\|\xi^\nu\|_{H^2_{1,\delta^\nu}} 
+\| (\xi^\nu_{02}- \xi'^\nu_{02})|_{t=0} \|_{H^1(\R)}
+ \Vert (\xi_1^\nu-\xi_1'^\nu)|_{t=0} \Vert_{H^1(\R)} 
+ \Vert \pi_{0211}^\perp\hx^\nu|_{t=\delta^\nu} \Vert_{H^1(\R)}
\to 0$.
By the standard Sobolev embedding 
$$
H^2([0,1]\times\R) \subset H^1([0,1],X)\hookrightarrow\cC^0([0,1],X) 
\qquad\text{with}\; X=H^1(\R)
$$
this implies $\| \xi^\nu_{02} \|_{\cC^0([0,1],H^1(\R))} \to 0$, and so 
\begin{equation} \label{before}
 \| \xi'^\nu_{02}|_{t=0} \|_{H^1(\R)}
\leq \| \xi^\nu_{02}|_{t=0} \|_{H^1(\R)} 
+ \| (\xi^\nu_{02}- \xi'^\nu_{02})|_{t=0} \|_{H^1(\R)} \to 0 .
\end{equation}
We can moreover integrate for all $t_0\in[0,\delta^\nu]$ to obtain
\begin{equation}\label{txi1}
\| \hat\xi^\nu|_{t=t_0} - \hat\xi^\nu|_{t=\delta^\nu} \|_{H^1(\R)}^2
\leq \delta^\nu \int_0^{\delta^\nu} \|\nabla_t\hat\xi^\nu\|_{H^1(\R)}^2
\leq \delta^\nu \|\hat\xi^\nu\|_{H^2(\R\times[0,\delta^\nu])}^2
\to 0 .
\end{equation}
Using Lemma~\ref{lltrans} we then obtain
\begin{align*}
& \|\hat\xi^\nu|_{t=\delta^\nu}\|_{H^1(\R)} \\
&\leq C \bigl(\Vert \pi_{02}\xi_{02}'^\nu|_{t=\delta^\nu} \Vert_{H^1(\R)}
+ \Vert (\xi_1^\nu-\xi_1'^\nu)|_{t=\delta^\nu} \Vert_{H^1(\R)} 
+ \Vert \pi_{0211}^\perp\hx^\nu|_{t=\delta^\nu} \Vert_{H^1(\R)} \bigr) \\
&\leq C \bigl(
\Vert \pi_{02}(\xi_{02}'^\nu|_{t=\delta^\nu} - \xi_{02}'^\nu|_{t=0} ) \Vert_{H^1(\R)}
+ \Vert \pi_{02} (\xi'^\nu_{02})|_{t=0} \Vert_{H^1(\R)} \\
&\qquad
+ \Vert (\xi_1^\nu-\xi_1'^\nu)|_{t=0} \Vert_{H^1(\R)} 
+ 2 \Vert \hat\xi^\nu|_{t=\delta^\nu} - \hat\xi^\nu|_{t=0} \Vert_{H^1(\R)} 
+ \Vert \pi_{0211}^\perp\hx^\nu|_{t=\delta^\nu} \Vert_{H^1(\R)}  
\bigr) \\
&\to 0
\end{align*}
with uniform constants $C,C'$ by \eqref{LL}, \eqref{before}, \eqref{txi1}, and
a bound on the operator norm of $\pi_{02}$.
Now combining $\|\hat\xi^\nu|_{t=\delta^\nu}\|_{H^1(\R)}\to 0$ with (\ref{txi1})
proves 
$\|\hat\xi^\nu\|_{\cC^0([0,\delta^\nu],H^1(\R))}\to 0$
in contradiction to the assumption and the previously established fact that
 $\| \xi^\nu_{02} \|_{\cC^0([0,1],H^1(\R))} \to 0$.
\end{proof}

The solution $u$ of the $0$-equation corresponds to $\xi=0$,
which is an almost zero of $\F_u$. This and a quadratic estimate
for $d\F_u$ near $0$ is the content of the next lemma.
For later purposes we also compare $d\F_u(\xi)$
with the linearized operator $D_{e_u(\xi)}$ 
of $\ol{\pd}_J=(\ol{\pd}_{J_{02}},\ol{\pd}_{\hat J})$ 
at $e_u(\xi)$.
To state the comparison we will need the pointwise linear operator
$$
E_u(\xi)\eta := \tfrac{d}{d\tau} e_u(\xi+\tau\eta) |_{\tau=0} .
$$
It satisfies $E_u(0)={\rm Id}$, and since $e_u$ maps $\Gamma_{1,\delta}$
to the space of maps satisfying the boundary conditions in (\ref{vvvv}),
the linearization $E_u(\xi)$ maps $\Gamma_{1,\delta}$ to 
the space of sections 
$\zeta\in\Gamma(v_{02}^*TM_{02})\times \Gamma(\hat v^*TM_{0211})$
over $v=(v_{02},\hat v):=e_u(\xi)$, 
that satisfy the linearized boundary conditions 
$$
(\zeta_{02},\zeta)|_{t=0}\in 
T_v(\Delta_{M_0\times M_2}\times \Delta_{M_1}), 
\quad
\hat\zeta|_{t=\delta}\in T_{\hat v}(L_{01}\times L_{12}), \quad 
\zeta_{02}|_{t=1}\in T_{v_{02}}(L_0\times L_2) .
$$
The linearized operator $D_v$ acts on this space of sections and is given by
$$
D_v \zeta = \tilde\nabla_\tau \ol{\partial}_J e_v(\tau\zeta) |_{\tau=0} ,
$$
with the connection $\tilde\nabla$ introduced on 
page \pageref{connection}.
In this notation we have $D_{e_u(0)}=d\F_u(0)$.

\begin{lemma}\label{quadratic} (Uniform quadratic and error estimates)
There are uniform constants $\eps>0$ and $C_1, C_2,C_3$ such that for
all $\delta\in(0,1]$ and $\xi\in\Gamma_{1,\delta}(\eps)$,
$\eta\in\Gamma_{1,\delta}$
\begin{align*}
\|\F_u(0)\|_{\Om_{1,\delta}} &\leq C_1 \delta^{\frac 14}, \\
\| d\F_u(\xi)\eta - d\F_u(0)\eta \|_{\Om_{1,\delta}}
&\leq C_2 \| \xi \|_{\Gamma_{1,\delta}} \| \eta \|_{\Gamma_{1,\delta}} ,\\
\| d\F_u(\xi)\eta - \Phi_u(\xi)^{-1}D_{e_u(\xi)} E_u(\xi)\eta 
\|_{\Om_{1,\delta}}
&\leq C_3 \| \xi \|_{\Gamma_{1,\delta}} \| \eta \|_{\Gamma_{1,\delta}} .
\end{align*}
\end{lemma}
\begin{proof}
To estimate $\F_u(0)$ we recall that $u_{02}$ is pseudoholomorphic
and $\bar u$ is constant in $t$, so
\begin{align*}
\|\F_u(0)\|_{\Om_{1,\delta}}
&=\|(0,\partial_s \bar u)\|_{H^1_{1,\delta}} + \|(0,\partial_s \bar u)\|_{L^4_{1,\delta}} \\
&\leq \delta^{\frac 12} \bigl( \|\partial_s u_{02}|_{t=0}\|_{H^1(\R)} 
+ 2 \|\partial_s \bar u_1\|_{H^1(\R)} \bigr)
+ \delta^{\frac 14} \bigl( \|\partial_s u_{02}|_{t=0}\|_{L^4(\R)} 
+ 2 \|\partial_s \bar u_1\|_{L^4(\R)} \bigr) \\
&\leq C_1 \delta^{\frac 14} .
\end{align*}
Here $\partial_s u_{02}\to 0$ converges exponentially as
$s\to\pm\infty$, and so does 
$\partial_s \bar u_1=d\ell_1(\partial_s\bar u_{02})$, where 
$\ell_1$ from \eqref{ell1} has bounded differential.
This shows that the above constant $C_1$ is indeed finite.
For the third estimate we differentiate as in \cite[p.68]{ms:jh}
the identity
$\Phi_u(\xi+\tau\eta)\F_u(\xi+\tau\eta) = \ol{\pd}_J(e_u(\xi+\tau\eta))$
to obtain
\begin{equation}\label{dFD}
\Phi_u(\xi)d\F_u(\xi)\eta - D_{e_u(\xi)} E_u(\xi)\eta
= - \Psi_u(\xi,\eta,\F_u(\xi)) ,
\end{equation}
where the estimate for the right hand side
$$
\Psi_u(\xi,\eta,\zeta)
:=\tilde\nabla_\tau(\Phi_u(\xi+\tau\eta)\zeta)|_{\tau=0} 
$$
is part of the estimates below.
The first component of $\F_u$ is independent of $\delta$, so the
quadratic estimates for it simply follow from the continuous 
differentiability of $\F_u$. 
For the second component we follow the argument in \cite[Prop.3.5.3.]{ms:jh} to
obtain a uniform constant for all $\delta\in(0,1]$.
We need to consider
$$ \F_{\bar u}(\hat\xi):= \Phi_{\bar
u}(\hat\xi)^{-1}(\ol{\partial}_{\hat J} e_{\bar u}(\hat\xi) ),
$$
where $e_{\bar u}(\hat\xi)=\exp_{\bar u}(\hat\xi+Q(\hat\xi))$ is the
exponential map with quadratic correction defined in \eqref{expu}.
Note that our parallel transport $\Phi_{\bu}(\hx)$ is defined with
respect to the path $\tau\mapsto e_\bu(\tau\hx)$ and the Hermitian
connection $\tilde\nabla$ on $T(M_0\times M_2\times M_1\times M_1)$
that leaves $\hat J$ invariant.  Since $e_\bu(0)=\bu$ and $d
e_\bu(0)={\rm Id}$, the same path can be used in the definition of
$\nabla_\hx$ instead of the geodesic.  Now let $\xi,\eta
\in\Gamma_{1,\delta}$ with $\|\xi\|_{H^2_{1,\delta}}\leq\eps$.  Then
by Lemma~\ref{Sobolev}
$$
\|\hx\|_{\cC^0}\leq C_S\|\xi\|_{H^2_{1,\delta}}\leq C_S\eps =: c_0,
\qquad
\|\he\|_{\cC^0}\leq C_S\|\eta\|_{H^2_{1,\delta}}
$$
with a uniform constant $C_S$ thus a uniform constant $c_0$ that only
depends on $\eps$. In the following, all constants will be uniform
in the sense that they only depend on $c_0$ and hence $\eps$.
Next, we consider
$$
E_\bu(\hx)\he := \tfrac{d}{d\tau} e_\bu(\hx+\tau\he) |_{\tau=0} , 
\qquad
\Psi_\bu(\hx,\he,\zeta)
:=\tilde\nabla_\tau(\Phi_\bu(\hx+\tau\he)\zeta)|_{\tau=0} .
$$
Note that $E_\bu(0)={\rm Id}$ and that
$\Psi(0,\he,\zeta)=0$ since the covariant derivative exactly uses 
the parallel transport $\Phi_\bu(\tau\he)$.
Moreover, these maps are linear in $\he$ and $\zeta$, and they
depend smoothly on $\hx$.
So given $\eps$ and thus $|\hx|\leq c_0$ we have linear bounds
$$
|E_\bu(\hx)|\leq c_1 ,\qquad
|\nabla(E_\bu(\hx))|
\leq c_1 \bigl( |\nabla\hx| + |d\bu||\hx| \bigr),\qquad
|\Psi_\bu(\hx,\he,\zeta)|\leq c_1|\hx||\he||\zeta| 
$$
with a uniform constant $c_1$.
With these preparations we calculate from (\ref{dFD}),
using the notation of \cite[Prop.3.5.3.]{ms:jh},
\begin{align*}
&\Phi_\bu(\hx) \bigl( d\F_\bu(\hx)\he - d\F_\bu (0) \he \bigr)\\
&=
- \Psi_\bu ( \hx,\he,\F_\bu(\hx) )
+ \bigl( \nabla (E_\bu(\hx)) \he \bigr)^{0,1}
+ \bigl( \bigl( E_\bu(\hx) - \Phi_\bu(\hx) \bigr) \nabla\he \bigr)^{0,1} \\
&\quad
-\tfrac 12 \hat J(e_\bu(\hx)) 
\bigl( \bigl( (\nabla_{(E_\bu(\hx)\he
- \Phi_\bu(\hx)\he )}\hat J )(e_\bu(\hx)) \bigr)
\Phi_\bu(\hx) d\bar u \bigr)^{0,1} \\
&\quad
-\tfrac 12 \hat J(e_\bu(\hx)) 
\bigl( 
\bigl( (\nabla_{\Phi_\bu(\hx)\he} \hat J)(e_\bu(\hx)) 
- \Phi_\bu(\hx) (\nabla_\he \hat J)(\bu) \Phi_\bu(\hx)^{-1}  \bigr)
\Phi_\bu(\hx) d\bar u \bigr)^{0,1} \\
&\quad
-\tfrac 12 \hat J(e_\bu(\hx)) 
\bigl( ( \nabla_{E_\bu(\hx)\he}\hat J )(e_\bu(\hx)) 
(d (e_\bu(\hx)) - \Phi_\bu(\hx) d\bar u ) \bigr)^{0,1} .
\end{align*}
We then use the uniform bounds\footnote{
For noncompact $M_i$ we here need $\hat J, \nabla\hat J$ bounded only in a compact neighbourhood of the image of $\bar u$.
}
on $\|\hat J\|_\infty$, $\|\nabla\hat J\|_\infty$,  $\|\Phi_{\bar u}(\hat\xi)^{-1}\|$, $|d\bar u|$,
$|\hx|$, and the estimates
\begin{align*}
&|\F_\bu(\hx)|\leq C|d(e_\bu(\hx))| 
\leq c_2 \bigl( |\nabla\hx| + |d\bu| \bigr) , \qquad
|d(e_\bu(\hx)) - \Phi_\bu(\hx) d\bar u | 
\leq c_2 \bigl( |\nabla\hx| + |d\bu||\hx| \bigr), \\
& 
\bigl| E_\bu(\hx) - \Phi_\bu(\hx) \bigr|\leq c_2 |\hx|, \qquad
\bigl| (\nabla_{\Phi_\bu(\hx)\he} \hat J)(e_\bu(\hx)) 
- \Phi_\bu(\hx) (\nabla_\he \hat J)(\bu) \Phi_\bu(\hx)^{-1}  \bigr|
\leq c_2 |\hx||\he| 
\end{align*}
with a uniform constant $c_2$ to obtain
with a further uniform constant $c_3$
\begin{align*}
\bigl| d\F_\bu(\hx)\he - d\F_\bu (0) \he \bigr| 
%
%&\leq
%c_1 |\hx| |\he| |\F_\bu(\hx)|
%+ |\nabla(E_\bu(\hx))| |\he| 
%+ \bigl| E_\bu(\hx) - \Phi_\bu(\hx) \bigr| 
%\bigl( |\nabla\he| + \|\nabla\hat J\|_\infty |d\bar u| |\he| \bigr) \\
%&\quad
%+ \bigl| (\nabla_{\Phi_\bu(\hx)\he} \hat J)(e_\bu(\hx)) 
%- \Phi_\bu(\hx) (\nabla_\he \hat J)(\bu) \Phi_\bu(\hx)^{-1}  \bigr|
%|d\bar u|
%+ \| \nabla\hat J\|_\infty c_1 |\he|
%\bigl| d (e_\bu(\hx)) - \Phi_\bu(\hx) d\bar u \bigr| \\
%
&\leq
c_3 \bigl(|\hx| |\he| + |\he| |\nabla\hx| +| \hx | |\nabla\he| \bigr).
\end{align*}
So far these pointwise estimates were standard calculations.
Now we have to check that they actually lead to uniform bounds  
in the $\delta$-dependent norms.
The zeroth order part of the $\Om_{1,\delta}$-norm
over $\R\times[0,\delta]$ can be estimated with the
help of Lemma~\ref{Sobolev} by
\begin{align*}
\bigl\| d\F_\bu(\hx)\he - d\F_\bu (0) \he \bigr\|_{L^2} 
&\leq
c_3 \bigl(\|\hx\|_{L^4} \|\he\|_{L^4} + \|\he\|_{\cC^0} \|\nabla\hx\|_{L^2} 
+ \| \hx \|_{\cC^0} \|\nabla\he\|_{L^2} \bigr) \\
&\leq
c_3 ( C_S^2 + 2 C_S) \|\xi\|_{H^2_{1,\delta}} \|\eta\|_{H^2_{1,\delta}}, \\
\bigl\| d\F_\bu(\hx)\he - d\F_\bu (0) \he \bigr\|_{L^4} 
&\leq
c_3 \bigl(\|\hx\|_{L^8} \|\he\|_{L^8} + \|\he\|_{\cC^0} \|\nabla\hx\|_{L^4} 
+ \| \hx \|_{\cC^0} \|\nabla\he\|_{L^4} \bigr) \\
&\leq
c_3 ( C_S^2 + 2 C_S) 
\bigl( \|\xi\|_{H^2_{1,\delta}} + \|\xi\|_{L^4_{1,\delta}} \bigr)
\bigl( \|\eta\|_{H^2_{1,\delta}} + \|\nabla\eta\|_{L^4_{1,\delta}} \bigr).
\end{align*}
For the first order part of the $\Om_{1,\delta}$-norm 
one differentiates the above identity and uses further bounds on
 $\|\nabla^2\hat J\|_\infty$ and $|\nabla d\bar u|$ 
to find a pointwise bound
\begin{align*}
\bigl| \nabla \bigl( d\F_\bu(\hx)\he - d\F_\bu (0) \he \bigr) \bigr| 
&\leq
c_4 \bigl(|\hx| + |\nabla\hx| \bigr)\bigl(|\he| + |\nabla\he| \bigr) \\
&\quad
+c_4 \bigl( |\nabla^2\hx| |\he| +  |\nabla\hx|^2 |\he|
+ |\nabla\hx| |\nabla\he| + |\hx| |\nabla^2\he| \bigr) .
\end{align*}
Then we again use Lemma~\ref{Sobolev}
and $\|\nabla\hx\|_{L^2}\leq\eps$ to obtain with a final 
uniform constant $c_5$
\begin{align*}
&\bigl\|\nabla\bigl( d\F_\bu(\hx)\he - d\F_\bu (0) \he \bigr) \bigr\|_{L^2} \\
&\leq
c_4 \bigl(\|\hx\|_{L^4} + \|\nabla\hx\|_{L^4} \bigr)
\bigl(\|\he\|_{L^4} + \|\nabla\he\|_{L^4} \bigr) \\
&\quad
+ c_4 \bigl(
 \|\nabla^2\hx\|_{L^2} \|\he\|_{\cC^0} 
+ \|\nabla\hx\|_{L^2} \|\nabla\hx\|_{L^4} \|\he\|_{L^4}
+ \|\nabla\hx\|_{L^4} \|\nabla\he\|_{L^4} 
+ \|\hx\|_{\cC^0} \|\nabla^2\he\|_{L^2} \bigr) \\
&\leq
c_5 
\bigl( \|\xi\|_{H^2_{1,\delta}} + \|\nabla\xi\|_{L^4_{1,\delta}} \bigr)
\bigl( \|\eta\|_{H^2_{1,\delta}} + \|\nabla\eta\|_{L^4_{1,\delta}} \bigr).
\end{align*}
\vspace{-8mm}

\end{proof}

Theorem~\ref{solving} now follows from the implicit function theorem
\cite[A.3.4]{ms:jh} if we can establish
surjectivity and a uniform bound on the right inverse 
for the linearized operator 
\begin{align} \label{lin op}
D^\delta : \Gamma_{1,\delta} \to \Om_{1,\delta},
\qquad
D^\delta \xi &:= d\F_u(0)\xi = \bigl( D_{u_{02}}\xi_{02} \,,\, D_{\bar u}\hat\xi \bigr) ; \\
D_{u_{02}}\xi_{02}
&=\nabla_s\xi_{02}+J(u_{02})\nabla_t\xi_{02} 
+ \nabla_{\xi_{02}}J_{02}(u_{02})\pd_t u_{02} , \nonumber \\
D_{\bar u}\hat\xi 
&= \nabla_s \hat\xi + \hat J(\bar u)\nabla_t \hat\xi 
+ \tfrac 12 \nabla_{\hx}\hat J(\bu)\hat J(\bu) \pd_s \bu . \nonumber
\end{align}
Here $D_{u_{02}}$ and $D_{\bar u}$ are the linearized operators of
$\ol{\partial}_{J_{02}}$ at $u_{02}$ (which is pseudoholomorphic) and of
$\ol{\partial}_{\hat J}$ at $\bu$ (which satisfies $\pd_t\bu=0$)
respectively.  
(See \cite[Prop.3.1.1.]{ms:jh} for an explicit calculation of the linearized 
operators, and note that we identify $\Om^{0,1}(\R\times[0,1],u^*TM)$ with sections
of $u^*TM$ by $\gamma ds+ J\gamma dt\mapsto\gamma$.)
We can identify the cokernel of $D^\delta$ with $(\im
D^\delta)^\perp \subset (H^1_{1,\delta})^*$.  By elliptic regularity
any element in this cokernel can be represented by the $L^2$-inner
product $\lan \eta,\im D^\delta\ran=0$ with a smooth section $\eta$.
Partial integration then shows that $\eta\in\Gamma_{1,\delta}$
satisfies the boundary conditions (\ref{LL}) and lies in the kernel of
the formal adjoint operator, $(D^\delta)^*\eta=0$.  Note that
$(D^\delta)^*$ is given by $\bigl(-\nabla_s + J_{02}(u_{02})\nabla_t ,
-\nabla_s + \hat J(\bar u)\nabla_t \bigr)$ plus lower order terms.  So
$(D^\delta)^*$ has the same analytic properties as $D^\delta$, and we
will prove the surjectivity of $D^\delta$ by establishing injectivity
for $(D^\delta)^*$.

By our assumptions on the index and regularity of 
$(u_0,u_2)\in\tM^1_0(x^-,x^+)$ we know that the 
operator $D_{u_{02}}\oplus\pi_{02}^\perp$ on the space of 
sections in $H^2(u_{02}^*T(M_0\times M_2))$
with boundary conditions at $t=1$ in $T(L_0\times L_2)$
(where $\pi_{02}^\perp$ is the projection at $t=0$)
is surjective and has a one dimensional kernel 
$\ker(D_{u_{02}}\oplus\pi_{02}^\perp)$.
This is not a subspace of $\Gamma_{1,\delta}$, 
but we will fix a complement for every $\delta>0$
in the following sense,
$$
K_0 := \bigl\{ \xi=(\xi_{02},\hx)\in 
\Gamma_{1,\delta} \,\big|\,
\lan \xi_{02}, \ker(D_{u_{02}}\oplus\pi_{02}^\perp) \ran_{L^2}\equiv 0 
\bigr\} .
$$
Here we used the $L^2$-inner product on 
$H^2(\R\times[0,1], u_{02}^*T(M_0\times M_2))$.

Combining the uniform linear estimates Lemma~\ref{Ddelta} and 
Lemma~\ref{D02} we can choose
$\delta_0:=\frac 1{16} c_1^2 c_2^2>0$ 
such that for all $\delta\in(0,\delta_0)$
and $\xi\in\Gamma_{1,\delta}$ 
\begin{align*}
(1 +c_2^{-1}) \| (D^\delta)^* \xi \|_{\Om_{1,\delta}} 
&\geq
\tfrac 12 \| (D^\delta)^* \xi \|_{H^1_{1,\delta}} 
+ \tfrac 12 \| (D^\delta)^* \xi \|_{L^4_{1,\delta}} 
+ c_2^{-1} \Vert D_{u_{02}}^* \xi_{02} \Vert_{H^1(\R\times[0,1])} \\
&\geq 
\tfrac 12 c_1 \| \xi \|_{\Gamma_{1,\delta}}
- c_2^{-1}\sqrt{\delta} \Vert \nabla_t\hx \Vert_{H^1(\R\times[0,\delta])} 
\;\geq\;
\tfrac 14 c_1 \| \xi \|_{\Gamma_{1,\delta}} ,
\end{align*}
and similarly for all 
$\xi\in\Gamma_{1,\delta}\cap K_0$
\begin{align}\label{right inverse}
\| D^\delta \xi \|_{\Om^1_{1,\delta}} 
\geq \frac{c_1 c_2}{4(c_2+1)}  \| \xi \|_{\Gamma_{1,\delta}} .
\end{align}
The first estimate shows that $(D^\delta)^*$ is injective and hence
$D^\delta$ is surjective. The second estimate shows that its right 
inverse is uniformly bounded. 
It remains to check that $D^\delta$ stays surjective when
restricted to $K_0$.
This follows from the fact that both $D_{u_{02}}$
with boundary conditions in $(L_{02},L_0\times L_2)$
and $D^\delta=(D_{u_{02}},D_{\bar u})$ 
with boundary conditions (\ref{LL}) are surjective
and have the same index $1$ by Lemma~\ref{index} and the identification
$\tM^1_{\bar\delta}(x^-,x^+)\cong\widehat\M^1_{\bar\delta}(x^-,x^+)$.
So $D^\delta$ has a $1$-dimensional kernel, which is
transversal to $K_0$ by the last estimate, and hence
$D^\delta|_{K_0}$ must be surjective.
This finishes the proof of theorem~\ref{solving}.
Here $\eps>0$ is fixed such that the exponential map $e_u$ 
is defined on $\Gamma_{1,\delta}(\eps)$ and such that Lemma~\ref{quadratic}
holds.

\begin{corollary} \label{injective}
There exists $\delta_0>0$ such that the map 
$\T_\delta : \M^1_0(x^-,x^+) \to \M^1_\delta(x^-,x^+)$
given by $\T_\delta([u]):=[v_u]$ is well defined and injective
for all $\delta\in(0,\delta_0]$.
\end{corollary}
\begin{proof}
We choose $\delta_0\leq\eps^4 C_0^{-4}$ such that Theorem~\ref{solving} applies.
Then let $v_u=e_u(\xi)$ be the solution constructed from $u\in \tM^1_0(x^-,x^+)$
and consider a shifted $0$-solution $\tilde u=u(\cdot+\sigma)\in[u]$. 
Then $\tilde\xi:=\xi(\cdot+\sigma)$ satisfies 
$\|\tilde\xi\|=\|\xi\|\leq C_0\delta^{\frac 14}\leq\eps$, 
$\F_u(\tilde\xi)=0$, and the orthogonality condition
to $\ker(D_{\tilde u_{02}}\oplus\pi_{02}^\perp)$.
Hence
$v_{\tilde u}=e_{u(\cdot+\sigma)}(\xi(\cdot+\sigma))
=v_u(\cdot+\sigma)\in[v_u]$, so $\T_\delta([u])=[v_u]$ is well defined.

The injectivity of $\T_\delta$ follows from the fact that
$\M^1_0(x^-,x^+)$ consists of isolated points, so the
$\cC^0$-distance $d_{\cC^0}([u],[u'])>\Delta_0$
is bounded below by some $\Delta_0>0$ for all $[u]\neq[u']$.
On the other hand,
$d_{\cC^0}([\bu],\T_\delta([u])\leq C_0 C_S (1+C_Q) \delta^{\frac 14}$
by (\ref{sqrtd}), (\ref{Quest}), and Lemma~\ref{Sobolev}. 
So if we had $\T_\delta([u])=\T_\delta([u'])$ then
$d_{\cC^0}([u],[u'])\leq d_{\cC^0}([\bu],[\bu'])
\leq 2C_0 C_S (1+C_Q) \delta^{\frac 14}$ .
This implies $[u]=[u']$ whenever $\delta\leq\delta_0$, where
we choose $\delta_0\leq (2C_0 C_S (1+C_Q))^{-4}\Delta_0^4$.
\end{proof}

\subsection{Uniform estimates}

In this section we establish the uniform linear and nonlinear estimates
that are used in Sections~\ref{ift} and \ref{cpt}.
We will work in the setup of section~\ref{ift} and 
fix a solution $u\in\tM^1_0(x^-,x^+)$.
For convenience we denote the target spaces by $M_{02}:=M_0\times M_2$
and $M_{0211}:=M_0\times M_2\times M_1\times M_1$ and the symplectic structures 
by $\om_{02}=(-\om_0)\oplus\om_2$ and 
$\om_{0211}=\om_0\oplus(-\om_2)\oplus(-\om_1)\oplus\om_1$ respectively.
The nonlinear equation for $v=(v_{02},\hat v)$, 
$v_{02}:\R\times[0,1]\to M_{02}$,
$\hat v:\R\times[0,\delta]\to M_{0211}$ is
$$
\ol{\partial}_J v := \pd_s v + J(v)\pd_t v :=
\bigl( \pd_s v_{02} + J_{02}(v_{02})\pd_t v_{02} \,,\, 
 \pd_s \hat v + \hat J(\hat v)\pd_t \hat v \bigr) .
$$
We will need uniform estimates for the nonlinear operator
$\xi\mapsto\ol{\partial}_J e_u(\xi)$ on $\xi\in \Gamma_{1,\delta}(\eps)$
and the linearized operator $D^\delta$.
For that purpose we use the Levi-Civita connection on $M=M_{02}$ and
$M=M_{0211}$ respectively to identify
$T_u M\times T_u M \cong T_\xi T_u M$ for every $\xi\in T_u M$.
With this we decompose $Te(u,\xi):T_\xi T_u M \to T_{e_u\xi}M$ as
$$
Te(u,\xi)(X,\eta)= \pd_1 e(u,\xi) X  +  de_u(\xi) \eta \qquad
\forall \xi,X,\eta\in T_u M .
$$
We denote the pullback almost complex structure on $H^2_{1,\delta}$ 
under $de_u(\xi)$ by
\begin{eqnarray*}
J(\xi) &:=& ( J_{02}(\xi_{02}),\hat J(\hat\xi) ) \\
&:=& \bigl( (de_{u_{02}}(\xi_{02}))^{-1} J_{02}(e_{u_{02}}(\xi_{02})) 
de_{u_{02}}(\xi_{02}) ,
(de_{\bar u}(\hat\xi))^{-1} \hat J(e_{\bar u}(\hat\xi)) de_{\bar u}(\hat\xi) \bigr)
\end{eqnarray*}
for $\xi=(\xi_{02},\hat\xi)\in \Gamma_{1,\delta}(\eps)$.
With this we can express
\begin{equation}\label{dbarxi}
\ol{\partial}_J ( e_u(\xi) ) = 
de_u(\xi) \bigl( \nabla_s \xi + J(\xi) \nabla_t \xi \bigr)
+ \pd_1 e(u,\xi)\pd_s u + J(u) \pd_1 e(u,\xi)\pd_t u
\end{equation}
in terms of the nonlinear operator on $H^2_{1,\delta}$, 
$$
\nabla_s \xi + J(\xi) \nabla_t \xi :=
\bigl(\nabla_s \xi_{02} + J_{02}(\xi_{02})\nabla_t \xi_{02} \,,\, 
 \nabla_s \hat \xi + \hat J(\hat\xi)\nabla_t \hat\xi \bigr) .
$$
Note that $J(0)=(J_{02},\hat J)$ is the usual almost complex structure
at $(u_{02},\bar u)$,
so we can express the linearized operator \eqref{lin op} as
\begin{align*}
D^\delta \xi &= 
\nabla_s \xi + J(0) \nabla_t \xi  
+ \bigl( \nabla_{\xi_{02}}J_{02}(u_{02})\partial_t u_{02} \,,\, 
\tfrac 12 \nabla_{\hx}\hat J(\bu)\hat J(\bu)\pd_s\bu \bigr) .
\end{align*}

The following lemma provides uniform elliptic estimates.

\begin{lemma}\label{Ddelta}  \hspace{2mm}\\
\vspace{-5mm}
\begin{enumerate}
\item
There is a constant $C_1$ such that for all $\delta\in(0,1]$ and 
${\xi\in \Gamma_{1,\delta}}$
\begin{align*}
 \biggl| \int_{\{1\}\times\R} \om_{02}(\xi_{02},\nabla_s\xi_{02} ) \biggr|
+ \biggl| \int_{\{\delta\}\times\R} 
\om_{0211}(\hat\xi, \nabla_s\hat\xi ) \biggr| 
&\leq C_1 \bigl( \| \xi_{02}|_{t=1} \|_{H^0(\R)} 
+ \| \hx|_{t=\delta} \|_{H^0(\R)}  \bigr)^2 , \\
\biggl| \int_{\{1\}\times\R} \om_{02}(\nabla_s\xi_{02},\nabla_s^2\xi_{02} ) 
\biggr| 
+ \biggl| \int_{\{\delta\}\times\R} 
\om_{0211}(\nabla_s\hat\xi, \nabla_s^2\hat\xi ) \biggr|
&\leq C_1 \bigl( \| \xi_{02}|_{t=1} \|_{H^1(\R)}  
+ \| \hx|_{t=\delta} \|_{H^1(\R)} \bigr)^2 .
\end{align*}
\item
There is a constant $\eps>0$ and for every $c_0>0$ there is
a constant $C_1$ such that for all $\delta\in(0,1]$ and 
$\xi,\zeta\in H^2_{1,\delta}$ with $\|\zeta\|_\infty\leq \eps$,
$\|\nabla\zeta\|_\infty\leq c_0$ 
\begin{align*}
\| \xi \|_{H^1_{1,\delta}} 
&\leq C_1 \biggl( 
\| \nabla_s \xi + J(\zeta)\nabla_t\xi \|_{H^0_{1,\delta}}
+ \| \xi \|_{H^0_{1,\delta}}  \\
&\qquad\quad
+ \biggl| \int_{\{\delta\}\times\R} 
\om_{0211}(\hat\xi, \nabla_s\hat\xi ) \biggr|^{1/2}
+ \biggl| \int_{\{1\}\times\R} \om_{02}(\xi_{02},\nabla_s\xi_{02} ) 
\biggr|^{1/2} \biggr), \\
\| \xi \|_{H^2_{1,\delta}} 
&\leq C_1 \Bigl( 
\| \nabla_s \xi + J(\zeta)\nabla_t\xi \|_{H^1_{1,\delta}}
+ \| \xi \|_{H^0_{1,\delta}}  \\
&\qquad\quad
+ \biggl| \int_{\{\delta\}\times\R} 
\om_{0211}(\hat\xi, \nabla_s\hat\xi ) \biggr|^{1/2}
+ \biggl| \int_{\{\delta\}\times\R} 
\om_{0211}(\nabla_s\hat\xi, \nabla_s^2\hat\xi ) \biggr|^{1/2} \\
&\qquad\quad
+ \biggl| \int_{\{1\}\times\R} \om_{02}(\xi_{02},\nabla_s\xi_{02} ) 
\biggr|^{1/2}
+ \biggl| \int_{\{1\}\times\R} \om_{02}(\nabla_s\xi_{02},\nabla_s^2\xi_{02} ) 
\biggr|^{1/2}  \biggl) , \\
\| \nabla \xi \|_{L^4_{1,\delta}}
&\leq C_1 \Bigl( \| \xi \|_{H^2_{1,\delta}} 
+ \| \nabla_s \xi + J(\zeta)\nabla_t\xi \|_{L^4_{1,\delta}}
+ \| \hx|_{t=\delta} \|_{H^1(\R)} \Bigr) .
\end{align*}
\item
There is a constant $c_1 > 0$ such that 
for all $\delta\in(0,1]$ and $\xi\in\Gamma_{1,\delta}$
\begin{align*}
c_1\| \xi \|_{H^2_{1,\delta}}
&\leq
\| D^\delta \xi \|_{H^1_{1,\delta}}
+ \| \xi \|_{H^0_{1,\delta}}
+ \| \hx|_{t=\delta} \|_{H^1(\R)} 
+ \| \xi_{02}|_{t=1} \|_{H^1(\R)} , \\
c_1\| \nabla\xi \|_{L^4_{1,\delta}}
&\leq
\| D^\delta \xi \|_{H^1_{1,\delta}}
+ \| D^\delta \xi \|_{L^4_{1,\delta}}
+ \| \xi \|_{H^0_{1,\delta}} 
+ \| \hx|_{t=\delta} \|_{H^1(\R)}  
+ \| \xi_{02}|_{t=1} \|_{H^1(\R)}, 
\end{align*}
and the same holds with $D^\delta$ replaced by $(D^\delta)^*$.
\end{enumerate}
\end{lemma}  
\begin{proof}
We prove (a) in general for
$\int_\R\omega(\xi,\nabla_s\xi)$ and
$\int_\R\omega(\nabla_s\xi,\nabla_s^2\xi)$
with a Lagrangian section $\xi:\R\to u^*TL$ over a path $u:\R\to L$. 
These expressions vanish if $L$ is totally geodesic.
To estimate them in general we pick a smooth family of orthonormal frames
$(\gamma_i(s))_{i=1,\dots,k}\in u(s)^*TL$, then
\begin{align*}
&\xi = \sum \lambda^i\gamma_i  ,\quad
\nabla_s\xi = \sum \Bigl( \pd_s\lambda^i\gamma_i 
+ \lambda^i\nabla_s\gamma_i \Bigr) , \quad
\nabla_s^2\xi = \sum \Bigl( \pd_s^2\lambda^i\gamma_i 
+ 2 \pd_s\lambda^i\nabla_s\gamma_i 
+ \lambda^i\nabla_s^2\gamma_i \Bigr)
\end{align*}
with $\lambda:\R\to\R^k$. By the orthonormality we have
$|\lambda(s)|=|\xi(s)|$,
and using $(\gamma,J\gamma)$ as a trivialization for the 
definition of Sobolev norms on $u^*TM$ we obtain
$\|\lambda\|_{H^s(\R)}=\|\xi\|_{H^s(\R)}$.
We now use the identities $\omega(\gamma_i,\gamma_j)=0$ to
obtain
\begin{align*}
\biggl| \int_\R \omega(\xi,\nabla_s\xi) \biggr|
&\leq  \biggl| \int_\R C |\xi(s)| |\lambda(s)| ds \biggr| 
=  C \|\xi\|_{L^2(\R)}^2 , \\
\biggl| \int_\R \omega(\nabla_s\xi,\nabla_s^2\xi) \biggr|
&\leq \biggl| \int_\R C \bigl( |\nabla_s\xi| |\lambda|
+ |\nabla_s\xi| |\pd_s\lambda| +  |\pd_s\lambda|^2 + |\lambda|^2 \bigl) 
 \biggr|
\leq 
4 C \|\xi\|_{H^1(\R)}^2 ,
\end{align*}
where the constant $C$ only depends on $\gamma$ (that is on $u:\R\to L$)
up to third derivatives.
Here we used partial integration
\begin{align*}
\int_\R \sum_{i,j} \lambda^i \pd_s^2\lambda^j 
\omega(\nabla_s\gamma_i,\gamma_j) 
&= 
- \int_\R \sum_{i,j} \Bigl( \pd_s \lambda^i \pd_s\lambda^j 
\omega(\nabla_s\gamma_i,\gamma_j) 
+ \lambda^i \pd_s\lambda^j \pd_s \omega(\nabla_s\gamma_i,\gamma_j) \Bigr) .
\end{align*}
To prove (c) we can replace $D^\delta$ by $\nabla_s\xi + J(0)\nabla_t\xi$
since the difference of the operators is bounded in the different components
and norms by
\begin{align} \label{lower order}
&\bigl\| \nabla_{\xi_{02}}J_{02}(u_{02})\pd_t u_{02} 
\bigr\|_{H^0(\R\times[0,1])} 
+ \bigl\| \nabla_{\hx}\hat J(\bu) J(\bu) \pd_s \bu 
\bigr\|_{H^0(\R\times[0,\delta])} 
\leq C \bigl\| \xi \bigr\|_{H^0_{1,\delta}} , \nonumber\\
&\bigl\| \nabla_{\xi_{02}}J_{02}(u_{02})\pd_t u_{02} 
\bigr\|_{L^4(\R\times[0,1])} 
\leq C \bigl\| \nabla_{\xi_{02}}J_{02}(u_{02})\pd_t u_{02} 
\bigr\|_{H^1(\R\times[0,1])} 
\leq C \bigl\| \xi \bigr\|_{H^1_{1,\delta}} , \nonumber\\
&\bigl\| \nabla_{\hx}\hat J(\bu) J(\bu) \pd_s \bu 
\bigr\|_{H^1(\R\times[0,\delta])} 
\leq C \bigl\| \xi \bigr\|_{H^1_{1,\delta}}  , \\
&\bigl\| \nabla_{\hx}\hat J(\bu) J(\bu) \pd_s \bu 
\bigr\|_{L^4(\R\times[0,\delta])} 
\leq  C \| \nabla\hat J\|_\infty \| \pd_s\bu \|_\infty 
\|\hx\|_{L^4(\R\times[0,\delta])} 
\leq C \bigl\| \xi \bigr\|_{H^2_{1,\delta}} , \nonumber
\end{align}
where $C$ denotes any uniform constant.  The extra terms on the right
hand side will fit into the proof and will be recalled for the
relevant estimates.  The proof for $(D^\delta)^*$ is completely
analogous. 
To prove (b) and (c) we may moreover fix convenient metrics on $M_{02}$ and $M_{0211}$. In order to obtain the boundary terms involving the symplectic forms, we pick the induced metrics $\lan \cdot, \cdot \ran = \omega_{02}(\cdot, J_{02} \cdot )$ resp.\ $\lan \cdot, \cdot \ran = \omega_{0211}(\cdot, \hat J \cdot )$, noting that this introduces a smooth $t$-dependence in the metric on $M_{02}$.
We will now use the notation $\nabla_s\xi + J(\sigma\zeta)\nabla_t\xi$ to make partial integration calculations for the nonlinear ($\sigma=1$) and linear ($\sigma=0$) operator at the same time.  In the nonlinear case the almost complex structure
$J(\zeta)$ is not skew-adjoint since the metric is defined by $J(0)$. 
In order to obtain this property we work with the $L^2_{1,\delta}(\sigma\zeta)$-metric, which uses the pullback metric $g_{\sigma\zeta}=\lan \cdot,\cdot\ran_{\sigma\zeta}$
under $de_{u_{02}}(\sigma\zeta_{02})$ on $M_{02}$ and $de_{\bar
u}(\sigma\zeta)$ on $M_{0211}$ respectively.  In the linear case
$\sigma=0$ nothing has happened; in the nonlinear case we can pick
$\eps>0$ and hence $\|\zeta\|_\infty$ sufficiently small such that
$de_u(\zeta)$ is $\cC^0$-close to the identity, and hence the induced
$L^2_{1,\delta}(\zeta)$-norm is uniformly equivalent to the standard
$L^2_{1,\delta}$-norm.  With this in mind we start by calculating for
any $\zeta,\eta\in H^2_{1,\delta}$ with $\|\zeta\|_\infty\leq\eps$
(unless otherwise specified integrals are over two infinite strips of
width $\delta$ and $1$)
\begin{align*} 
&\| \nabla_s \eta + J(\sigma\zeta) \nabla_t\eta \|_{L^2_{1,\delta}(\sigma\zeta)}^2 \\
&=  
 \int \Bigl( | \nabla_s \eta |_{\sigma\zeta}^2 + | \nabla_t \eta |_{\sigma\zeta}^2
+ \lan \nabla_s\eta ,J(\sigma\zeta)\nabla_t\eta \ran_{\sigma\zeta}
- \lan \nabla_t\eta , J(\sigma\zeta)\nabla_s\eta \ran_{\sigma\zeta} \Bigr)
\end{align*} 
\begin{align*} 
&=  \| \nabla \eta \|_{L^2_{1,\delta}(\sigma\zeta)}^2
- \int \Bigl( 
\nabla_s g_{\sigma\zeta} \bigl( \eta , J(\sigma\zeta)\nabla_t\eta \bigr)
- \nabla_t g_{\sigma\zeta} \bigl( \eta , J(\sigma\zeta)\nabla_s\eta \bigr) \Bigr)\\
&\quad
- \int \Bigl(  \lan \eta , \bigl(
  \nabla_s\bigl( J(\sigma\zeta)\nabla_t\eta \bigr)
- \nabla_t\bigl( J(\sigma\zeta)\nabla_s\eta \bigr) \bigr) \ran_{\sigma\zeta} \Bigr)\\
&\quad
- \lim_{S\to\infty} 
\int_{\{s=-S\}} \lan \eta ,J(\sigma\zeta)\nabla_t\eta \ran_{\sigma\zeta}
+ \lim_{S\to\infty} 
\int_{\{s=S\}} \lan \eta ,J(\sigma\zeta)\nabla_t\eta \ran_{\sigma\zeta} \\
&\quad
+ \int_{\{0\}\times\R} \lan \eta , J(\sigma\zeta)\nabla_s\eta \ran_{\sigma\zeta} 
- \int_{\{1\}\times\R} \lan \eta_{02} , 
J_{02}(\sigma\zeta_{02})\nabla_s\eta_{02} \ran_{\sigma\zeta_{02}} 
- \int_{\{\delta\}\times\R} \lan \hat\eta , 
\hat J(\sigma\hat\zeta)\nabla_s\hat\eta \ran_{\sigma\hat\zeta} \\
&\geq  \| \nabla \eta \|_{L^2_{1,\delta}(\sigma\zeta)}^2
- \int C \bigl((1+\sigma c_0) |\eta| |\nabla \eta| + |\eta|^2 \bigr)
- \Om_{02}(\eta_{02}|_{t=1})- \Om_{0211}(\hat\eta|_{t=\delta}), 
\end{align*}
where we abbreviated
\begin{align*} 
\Om_{02}(\eta_{02}|_{t=1}):=
 \biggl| \int_{\{1\}\times\R} \om_{02}(\eta_{02},\nabla_s\eta_{02} ) \biggr|,
\qquad
\Om_{0211}(\hat\eta|_{t=\delta}):=
\biggl| \int_{\{\delta\}\times\R} 
\om_{0211}(\hat\eta, \nabla_s\hat\eta ) \biggr| .
\end{align*}
These boundary terms occur on the right hand side of (c) and they will be 
estimated by (a) to prove (b).
The boundary term at $t=0$ vanishes by the diagonal boundary conditions,
and the boundary terms at $S\to\pm\infty$ vanish since 
$\eta|_{\{s\in[S,S+1]\}}\to 0$ in the $H^2_{1,\delta}$-norm.
The error term can be estimated by
$$
\int C \bigl( (1+\sigma c_0) |\eta| |\nabla \eta| + |\eta|^2 \bigr)
\leq C \|\eta\|_{L^2_{1,\delta}(\sigma\zeta)}^2 
+ \frac 1 2 \|\nabla\eta\|_{L^2_{1,\delta}(\sigma\zeta)}^2
+ \frac 1 2 C^2 (1+\sigma c_0)^2 \|\eta\|_{L^2_{1,\delta}(\sigma\zeta)}^2  ,
$$
where the highest order term $\|\nabla\eta\|$ can be absorbed on the right
hand side.
From now on $C$ will denote any uniform constant (which is allowed to depend
on $c_0$ in the nonlinear case $\sigma=1$).
In summary, the estimates for $\eta=\xi$ and
$\eta=\nabla_s\xi$ are
\begin{align*}
\tfrac 1C \| \nabla \xi \|_{L^2_{1,\delta}}^2
&\leq
\bigl\| \nabla_s\xi + J(\sigma\xi)\nabla_t\xi 
\bigr\|_{L^2_{1,\delta}}^2
+ \|\xi\|_{L^2_{1,\delta}}^2 
+ \Om_{02}(\xi_{02}|_{t=1}) + \Om_{0211}(\hat\xi|_{t=\delta}) , \\
\tfrac 1C \| \nabla\nabla_s \xi \|_{L^2_{1,\delta}}^2
&\leq
\bigl\| \nabla_s \bigl( \nabla_s\xi + J(\sigma\xi)\nabla_t\xi \bigr)
\bigr\|_{L^2_{1,\delta}}^2
+ \|\nabla\xi\|_{L^2_{1,\delta}}^2  \\
&\qquad
+ \Om_{02}(\nabla_s\xi_{02}|_{t=1}) 
+ \Om_{0211}(\nabla_s\hat\xi|_{t=\delta}) .
\end{align*}
This already proves the first estimate in (b).
We can moreover use the identity 
$\nabla_t\xi=J(\sigma\zeta)\nabla_s\xi 
- J(\sigma\zeta)(\nabla_s\xi+J(\sigma\zeta)\nabla_t\xi)$
to obtain
\begin{align*}
\| \nabla\nabla_t \xi \|_{L^2_{1,\delta}}
&\leq
\| \nabla\nabla_s \xi \|_{L^2_{1,\delta}}
+ \| \nabla( \nabla_s\xi+J(\sigma\zeta)\nabla_t\xi )  \|_{L^2_{1,\delta}}
+ C \| \nabla\xi \|_{L^2_{1,\delta}} 
+ \sigma C c_0 \|\nabla\xi \|_{L^2_{1,\delta}} .
\end{align*}
In the linear case (c) these estimates combined with (a) 
and (\ref{lower order}) to prove the first estimate:  
\begin{align*}
c_1\| \xi \|_{H^2_{1,\delta}}
&\leq
\bigl\| D^\delta \xi \bigr\|_{H^1_{1,\delta}}
+ \| \xi_{02}|_{t=1} \|_{H^1(\R)} + \| \hx|_{t=\delta} \|_{H^1(\R)} 
+\|\xi\|_{L^2_{1,\delta}} 
\end{align*}
with a uniform constant $c_1>0$.
In the nonlinear case (b) we obtain similarly
\begin{align*}
C_1^{-1}\| \xi \|_{H^2_{1,\delta}}
&\leq
\bigl\| \nabla_s\xi + J(\zeta)\nabla_t\xi \bigr\|_{H^1_{1,\delta}}
+\|\xi\|_{L^2_{1,\delta}} 
+ \Om_{02}(\xi_{02}|_{t=1}) + \Om_{0211}(\hat\xi|_{t=\delta}) \\
&\qquad
+ \Om_{02}(\nabla_s\xi_{02}|_{t=1}) 
+ \Om_{0211}(\nabla_s\hat\xi|_{t=\delta})
\end{align*}
with a constant $C_1$ that depends on $\|\nabla\xi\|_\infty\leq c_0$.

The $L^4$-estimate for the linear and nonlinear operators
will arise by rescaling from the following basic estimate.
Here $\hat u:\R\times[0,1]\to M_{0211}$ will be given by
$\hat u(s,t)=\bar u(\delta s)$ for any $\delta\in(0,1]$.
Then for every $\he\in H^1(\R\times[0,1],\hat u^*TM_{0211})$
$$
\| \he \|_{L^4(\R\times[0,1])}
\leq C_0 \bigl( \|\he|_{t=1} \|_{L^2(\R)}
+ \|\nabla\he\|_{L^2(\R\times[0,1])} \bigr) .   
$$
This simply follows from the Sobolev embedding
$H^1(\R\times[0,1]) \hookrightarrow L^4(\R\times[0,1])$ and
$$
\|\he\|_{L^2(\R\times[0,1])}^2
\leq
\int_0^1 \biggl\| \he(\cdot,1) 
- \int_t^1 \nabla_t\he(\cdot,\tau) d\tau \biggr\|_{L^2(\R)}^2 d t
\leq
2\| \he|_{t=1} \|_{L^2(\R)}^2
+ 2 \|\nabla_t\he\|_{L^2(\R\times[0,1])}^2 .
$$
When applying this to $\he(s,t):=\nabla_s\hx(\delta s,\delta t)$
we encounter the following terms:
\begin{align*}
\| \he \|_{L^4(\R\times[0,1])}^2
&= \biggl( \int_{\R\times[0,1]}
|\nabla_s\hx(\delta s,\delta t)|^4 ds dt \biggr)^{1/2}  
= \delta^{-1}\| \nabla_s\hx \|_{L^4(\R\times[0,\delta])}^2 , \\
\| \he|_{t=1} \|_{L^2(\R)}^2
&= \int_{\R} |\nabla_s\hx(\delta s,\delta)|^2 ds 
= \delta^{-1} \| \nabla_s\hx|_{t=\delta} \|_{L^2(\R)}^2 , \\
\|\nabla\he\|_{L^2(\R\times[0,1])}^2 
&= \int_{\R\times[0,1]} \delta^2
| \nabla\nabla_s\hx (\delta s,\delta t) |^2 ds dt 
= \|\nabla\nabla_s\hx\|_{L^2(\R\times[0,\delta])}^2.
\end{align*}
Putting this together we find that
\begin{align*}
\| \nabla_s\hx \|_{L^4(\R\times[0,\delta])}
&\leq C_0 \bigl( \|\nabla_s\hx|_{t=\delta} \|_{L^2(\R)}
+ \|\nabla\nabla_s\hx\|_{H^2(\R\times[0,\delta])} \bigr) 
\leq C_0 \bigl( \|\hx|_{t=\delta} \|_{H^1(\R)}
+ \|\xi\|_{H^2_{1,\delta}} \bigr),
\end{align*}
where the estimate for $\| \xi \|_{H^2_{1,\delta}}$ 
is already established.
The $L^4$-estimate for $\nabla \xi_{02}$ follows
from the Sobolev embedding 
$H^1(\R\times[0,1])\hookrightarrow L^4(\R\times[0,1])$,
and for the last component we have
\begin{align*}
 \| \nabla_t\hx \|_{L^4(\R\times[0,\delta])} 
\leq 
\| \nabla_s\hx + \hat J(\sigma\hat\zeta)\nabla_t\hx 
\|_{L^4(\R\times[0,\delta])}
+ \| \nabla_s\hx \|_{L^4(\R\times[0,\delta])}.
\end{align*}
This finishes the proof of the second estimate, where we allow 
$\| \nabla_s\xi + J(\sigma\zeta)\nabla_t\xi \|_{L^4_{1,\delta}}$
on the right hand side, and the constant in the nonlinear case
depends on $\|\nabla\zeta\|_\infty\leq c_0$.
In the linear case the difference to $\| D^\delta \xi \|_{L^4_{1,\delta}}$ 
in (\ref{lower order}) is bounded by the previous estimate.
\end{proof}

The lemma below gives control of the lower-order terms appearing in
Lemma \ref{Ddelta} and in particular will be used to prove
surjectivity of the linearized operator.

\begin{lemma} \label{D02}  
\begin{enumerate}
\item
There is a constant $\eps>0$ and for every $c_0>0$ there is a constant $C_2$ 
such that for all $\delta\in(0,1]$ and $\xi,\zeta\in H^2_{1,\delta}$
with $\|\zeta\|_\infty\leq\eps$, $\|\nabla\zeta\|_\infty\leq c_0$ we have
\begin{align*}  
&\|\hat\xi|_{t=\delta}\|_{H^1(\R)} + \|\xi_{02}|_{t=1}\|_{H^1(\R)} \\
&\leq
C_2 \bigl( 
\Vert \nabla_s\xi_{02} + J_{02}(\zeta_{02})\nabla_t\xi_{02} \Vert_{H^1(\R\times[0,1])} 
+ \sqrt{\delta} \Vert \nabla_t\hx \Vert_{H^1(\R\times[0,\delta])}  
+\|\pi_{0211}^\perp\hat\xi|_{t=\delta}\|_{H^1(\R)} \\
&\qquad\quad
+ \|\xi_{02}\|_{L^2(\R\times[0,1])} 
+ \|(\xi_1'-\xi_1)|_{t=0}\|_{H^1(\R)}
+ \|(\xi_{02}'-\xi_{02})|_{t=0}\|_{H^1(\R)} \bigr).
\end{align*}
\item
There is a constant 
$c_2 > 0$ such that for all $\delta\in(0,1]$ and
$\xi\in\Gamma_{1,\delta}$ 
\begin{align*}  
c_2 \bigl( \|\hat\xi|_{t=\delta}\|_{H^1(\R)} + \|\xi_{02}|_{t=1}\|_{H^1(\R)} 
+ \Vert \xi \Vert_{H^0_{1,\delta}} \bigr)
&\leq
\Vert D_{u_{02}}^* \xi_{02} \Vert_{H^1(\R\times[0,1])} 
+ \sqrt{\delta} \Vert \nabla_t\hx \Vert_{H^1(\R\times[0,\delta])} ,
\end{align*}
and for all $\xi\in\Gamma_{1,\delta}\cap K_0$ 
\begin{align*} 
c_2 \bigl( \|\hat\xi|_{t=\delta}\|_{H^1(\R)} + \|\xi_{02}|_{t=1}\|_{H^1(\R)} 
+ \Vert \xi \Vert_{H^0_{1,\delta}} \bigr) 
&\leq
\Vert D_{u_{02}} \xi_{02} \Vert_{H^1(\R\times[0,1])} 
+ \sqrt{\delta} \Vert \nabla_t\hx \Vert_{H^1(\R\times[0,\delta])} .
\end{align*}
\end{enumerate}
\end{lemma}  
\begin{proof}
The constant $\eps>0$ in case (a) is chosen such that $e_{u_{02}}(\zeta_{02})$
and thus $J_{02}(\zeta_{02})$ is defined.
To prove (a) (and similar for (b)) we assume by contradiction 
that we have sequences $\delta^\nu>0$ and 
$\xi^\nu,\zeta^\nu\in H^2_{1,\delta^\nu}$ such that 
$\|\hat\xi^\nu|_{t=\delta^\nu}\|_{H^1(\R)} + \|\xi_{02}^\nu|_{t=1}\|_{H^1(\R)}=1$
(in case (b) add $\|\xi^\nu\|_{H^0_{1,\delta}}$ here),
but the right hand sides converges to zero.
For technical reasons we assume in addition 
$\|\xi^\nu_{02}\|_{H^1(\R\times[0,1])}\leq 1$,
which we will also disprove (i.e.\ we actually prove a stronger estimate with
this term on the left hand side).
First we integrate for all $t\in[0,\delta^\nu]$
\begin{equation}\label{txii}
\| \hx^\nu|_{t=t_0} - \hx^\nu|_{t=\delta^\nu} \|_{H^1(\R)}
\leq \int_0^{\delta^\nu} \|\nabla_t\hx^\nu\|_{H^1(\R)}
\leq \sqrt{\delta^\nu} \|\nabla_t\hx^\nu\|_{H^1(\R\times[0,\delta^\nu])}
\to 0 .
\end{equation}
Next, Lemma~\ref{lltrans} implies
\begin{align}  \label{L2H1}
\|\pi_{02}^\perp \xi^\nu_{02}|_{t=0} \|_{L^2(\R)} 
&\leq 
\|\pi_{02}^\perp \xi_{02}'^\nu |_{t=\delta^\nu} \|_{L^2(\R)}
+ \|\xi_{02}'^\nu|_{t=0} - \xi_{02}'^\nu|_{t=\delta^\nu}\|_{L^2(\R)}
+ \|(\xi_{02}'^\nu-\xi_{02}^\nu)|_{t=0}\|_{L^2(\R)} \nonumber\\
&\leq 
C \bigl( \|\pi_{0211}^\perp \hat\xi^\nu|_{t=\delta^\nu}\|_{L^2(\R)} 
+ \|\hat\xi^\nu|_{t=0} - \hat\xi^\nu|_{t=\delta^\nu}\|_{L^2(\R)} \nonumber\\
&\qquad
+ \|(\xi_1'^\nu-\xi_1^\nu)|_{t=0}\|_{L^2(\R)} 
+ \|(\xi_{02}'^\nu-\xi_{02}^\nu)|_{t=0}\|_{L^2(\R)} \bigr) 
 \to 0 , \nonumber\\
\|\pi_{02}^\perp \xi^\nu_{02}|_{t=0} \|_{H^1(\R)} 
&\leq 
\|\pi_{02}^\perp \xi_{02}'^\nu |_{t=\delta^\nu} \|_{H^1(\R)}
+ \|\xi_{02}'^\nu|_{t=0} - \xi_{02}'^\nu|_{t=\delta^\nu}\|_{H^1(\R)}
+ \|(\xi_{02}'^\nu-\xi_{02}^\nu)|_{t=0}\|_{H^1(\R)}  \\
&\leq 
C \bigl( \|\pi_{0211}^\perp \hat\xi^\nu |_{t=\delta^\nu} \|_{H^1(\R)}
+ \|\hat\xi^\nu|_{t=0} - \hat\xi^\nu|_{t=\delta^\nu}\|_{H^1(\R)} 
+ \|(\xi_1'^\nu-\xi_1^\nu)|_{t=0}\|_{H^1}
% really _{H^1(\R)} but it doesn't fit ;-(
\nonumber\\
&\qquad\qquad\qquad
+ \|(\xi_{02}'^\nu-\xi_{02}^\nu)|_{t=0}\|_{H^1(\R)}
+ \bigl\| |\pd_s\bu|\cdot|\hat\xi^\nu|_{t=\delta^\nu}| \bigr\|_{L^2(\R)} 
  \bigr) . \nonumber 
\end{align}
In the two cases of (b) we use the boundary conditions 
for $\xi^\nu\in\Gamma_{1,\delta}$ here. 
In all three cases the hardest step is now to prove that
$\bigl\| |\pd_s\bu|\cdot|\hat\xi^\nu|_{t=\delta^\nu}| \bigr\|_{L^2(\R)}\to 0$.
Here we exploit the assumption that $\|\xi_{02}^\nu\|_{H^1(\R\times[0,1])}$ is
bounded. This implies a bound on $\|\xi_{02}^\nu|_{t=0}\|_{L^2(\R)}$.
Now we find a convergent subsequence
$\xi_{02}^\nu\to \xi_{02}^\infty\in H^1(\R\times[0,1],u_{02}^*TM_{02})$
in the weak $H^1$-topology, and at the same time
$\xi_{02}^\nu|_{t=0}\to \xi_{02}^\infty|_{t=0}$ in the $L^2$-norm on every
compact set. (The Sobolev embedding 
$H^1(\Om)\hookrightarrow L^2(\pd\Om))$ 
is compact for compact domains $\Om\subset\R\times[0,1]$ with smooth
boundary $\pd\Om$, see e.g.\ \cite[Theorem~6.3]{ad:so}.)
In case (a) the limit has to be $\xi^\infty_{02}=0$ since 
$\| \xi^\infty_{02} \|_{L^2(\R\times[0,1])} \leq 
\liminf_{\nu\to\infty} \| \xi^\nu_{02} \|_{L^2(\R\times[0,1])} = 0 $.
This also holds in case (b) since the limit satisfies
with $D=D_{u_{02}}$ or $D=D^*_{u_{02}}$
\begin{align*}
\| D \xi^\infty_{02} \|_{L^2(\R\times[0,1])} &\leq 
\liminf_{\nu\to\infty} \| D \xi^\nu_{02} \|_{L^2(\R\times[0,1])} = 0, \\
\qquad
\| \pi_{02}^\perp\xi^\infty_{02}|_{t=0} \|_{L^2(\R)} &\leq 
\liminf_{\nu\to\infty} \| \pi_{02}^\perp\xi^\nu_{02}|_{t=0} \|_{L^2(\R)} = 0 .
\end{align*}
Since $u_{02}$ is assumed regular, $D_{u_{02}}^*\oplus\pi_{02}^\perp$
is injective, and in the second part of case (b) we have in addition
$\xi^\infty_{02}\in \ker(D_{u_{02}}\oplus\pi_{02}^\perp)^\perp$.  So
in all three cases we obtain
$$
\|\xi_{02}^\nu|_{t=0}\|_{L^2(\R)} \leq C \qquad\text{and}\qquad
\|\xi_{02}^\nu|_{t=0}\|_{L^2([-T,T])}\to 0 \qquad \text{for all}\; T>0 .
$$
The same holds for $\hat\xi^\nu|_{t=\delta^\nu}$ since we can apply
Lemma~\ref{lltrans} on the interval $(-T,T)$ for any $T\in(0,\infty]$
to obtain
\begin{align*}
\|\hat\xi^\nu|_{t=\delta^\nu}\|_{L^2}
&\leq  C\bigl(
\Vert \pi_{02}\xi_{02}'^\nu|_{t=\delta^\nu} \Vert_{L^2}
+ \|\pi_{0211}^\perp\hat\xi^\nu|_{t=\delta^\nu}\|_{L^2}
+ \|(\xi_1'^\nu-\xi_1^\nu)|_{t=\delta^\nu}\|_{L^2} \bigl)\\
&\leq
C' \bigl( \| \xi_{02}^\nu|_{t=0} \|_{L^2}
+ \|(\xi_{02}'^\nu-\xi_{02}^\nu)|_{t=0}\|_{L^2}
+ \|\hat\xi^\nu|_{t=0}- \hat\xi^\nu|_{t=\delta^\nu}\|_{L^2} \\
&\qquad\qquad\quad
+ \|\pi_{0211}^\perp \hat\xi^\nu|_{t=\delta^\nu}\|_{L^2} 
+ \|(\xi_1'^\nu-\xi_1^\nu)|_{t=0}\|_{L^2} \bigl) .
\end{align*}
This together with the fact that $\sup_{|s|\geq T}|\pd_s\bu(s)|\to 0$ 
as $T\to\infty$ implies that
$\bigl\| |\pd_s\bu|\cdot|\hat\xi^\nu|_{t=\delta^\nu}| \bigr\|_{L^2(\R)}\to 0$
and hence $\|\pi_{02}^\perp \xi^\nu_{02}|_{t=0} \|_{H^1(\R)} \to 0$ by (\ref{L2H1}).
From this we will move on to prove that 
\begin{equation}\label{xito0}
\Vert \xi^\nu_{02} \Vert_{H^{3/2}(\R\times[0,1])} \to 0 .
\end{equation}
For that purpose we denote by $D$ any of the three operators 
$\nabla_s + J_{02}(\zeta_{02}) \nabla_t$ in case (a)
and $D_{u_{02}}^*$ or $D_{u_{02}}$ in case (b).
Then we use the fact that in all three cases the operator
$D\oplus\pi_{02}^\perp$ is Fredholm on the space of sections 
$\eta$ that satisfy the boundary conditions
$\eta|_{t=1}\in T_{u_{02}}(L_0\times L_2)$, 
see e.g. \cite[Theorem 20.1.2]{ho:an3} for compact domains.
The corresponding estimates add up to
\begin{align} \label{basest}
\Vert \xi_{02}^\nu \Vert_{H^{3/2}(\R\times[0,1])} 
&\leq C \bigl(
\Vert D \xi_{02}^\nu \Vert_{H^1(\R\times[0,1])} 
+ \|\pi_{02}^\perp \xi_{02}^\nu|_{t=0} \|_{H^1(\R)}
+ \| \xi_{02}^\nu \|_{H^0(\R\times[0,1])} \bigr).
\end{align}
In the nonlinear case (a) the constant in this estimate depends
continuously on $J_{02}(\zeta_{02})$ in the $\cC^1$-topology, see
e.g.\ \cite[Appendix~B]{ms:jh}.  In this case the above estimate
already implies the claim (\ref{xito0}) since we assumed
$\|\xi_{02}^\nu\|_{L^2} \to 0$.  In the linear cases we need to use
the injectivity of the operators to remove the last term from the
right hand side of \eqref{basest}.  Since $H^{3/2}(\R)\hookrightarrow
H^0((-T,T))$ is compact only for $T<\infty$, we first have to achieve
a lower order term on a compact domain:

Consider the operator $D_{x^\pm}=\pd_s - A$, where
$A:=-J(x^\pm)\pd_t$ (or $A:=J(x^\pm)\pd_t$ in the case $D=D_{u_{02}}^*$)
is self-adjoint and invertible on its constant domain
$H^1([0,1],T_{x^\pm}M_{02})$ with boundary conditions
$\eta|_{t=0}\in T_{x^\pm}L_{02}$, 
$\eta|_{t=1}\in T_{x^\pm}(L_0\times L_2)$.
Then abstract theory (e.g.\ \cite[Lemma~3.9, Proposition~3.14]{rs:spec}) 
implies the Fredholm property and bijectivity,
$$
\| \eta \|_{H^1(\R\times[0,1])}
\leq C \Vert D_{x^\pm} \eta \Vert_{H^0(\R\times[0,1])} .
$$
In order to apply this estimate to $\xi_{02}^\nu$ we first
find an extension $\zeta\in H^1(\R\times[0,1])$ of
$\zeta|_{t=0}=\pi_{02}^\perp\xi^\nu_{02}|_{t=0}$ such that
$\|\zeta\|_{H^1}\leq C\|\pi_{02}^\perp\xi^\nu_{02}|_{t=0}\|_{H^{1/2}}$.
We moreover fix a cutoff function 
$h\in\cC_0^\infty(\R,[0,1])$ with $h|_{\{|s|\leq T-1\}}\equiv 0$
and $h|_{\{|s|\geq T\}}\equiv 1$,
where we fix $T>1$ sufficiently large such that
$u_{02}|_{\on{supp}(h)}=e_{x^\pm}(\vartheta_{02})$
for some smooth map $\vartheta_{02}:\{\pm s\geq (T-1)\} \to T_{x^\pm}M_{02}$.
Then we can apply the estimate to 
$\eta:=\Phi_{x^\pm}(\vartheta_{02})^{-1}\bigl( h(\xi^\nu_{02}-\zeta)\bigr)$, 
where $\Phi_{x^\pm}(\vartheta_{02})$ denotes parallel transport 
along the path 
$[0,1]\ni\tau\mapsto e_{x^\pm}(\tau\vartheta_{02})$.
We obtain, denoting all uniform constants by $C$,
\begin{align*}
&\Vert h \xi^\nu_{02} \Vert_{H^1(\R\times[0,1])}  \\
&\leq C \Vert \eta \Vert_{H^1(\R\times[0,1])} 
+ \Vert h \zeta \Vert_{H^1(\R\times[0,1])} \\
&\leq C \bigl( 
\Vert \bigl( D_{x^\pm} - D\circ\Phi_{x^\pm}(\vartheta_{02}) \bigr) 
\eta \Vert_{H^0(\R\times[0,1])} 
+\Vert D (h\xi^\nu_{02}) \Vert_{H^0(\R\times[0,1])} 
+ \Vert h \zeta \Vert_{H^1(\R\times[0,1])} \bigr) \\
&\leq C \bigl( 
\Vert \bigl( D_{x^\pm} - D\circ\Phi_{x^\pm}(\vartheta_{02}) 
\bigr)\bigr|_{\{|s|>T-1\}}\Vert \cdot
\| h( \xi^\nu_{02} - \zeta) \Vert_{H^1(\R\times[0,1])} 
+\Vert D\xi^\nu_{02} \Vert_{H^0(\R\times[0,1])} \\
&\qquad
+\Vert \xi^\nu_{02} \Vert_{H^0([-T,T]\times[0,1])} 
+ \|\pi_{02}^\perp\xi^\nu_{02}|_{t=0}\|_{H^{1/2}(\R)} \bigr) .
\end{align*}
Here the difference of the operators goes to zero for $T\to\infty$
since $u_{02}|_{\{|s|\geq T-1\}}\to x^\pm$ with all derivatives.  Thus for sufficiently large $T>0$ we can absorb
the first term into the left hand side and 
$\|h\zeta\|_{H^1}\leq C\|\pi_{02}^\perp\xi^\nu_{02}|_{t=0}\|_{H^{1/2}}$.  
After
all this we can finally replace the last term in (\ref{basest}) by
$\Vert \xi^\nu_{02} \Vert_{H^0([-T,T]\times[0,1])}$.

Now in the first case of (b) we can deduce (\ref{xito0})
from the fact that $D_{u_{02}}\oplus \pi^\perp_{02}$ is surjective 
by assumption and hence $D_{u_{02}}^*\oplus \pi^\perp_{02}$ is injective.
So the compact embedding
$H^{3/2}(\R\times[0,1])\hookrightarrow H^0([-T,T]\times[0,1])$
allows the removal of the lower order term.
Similarly, in the second case of (b) we can employ the injectivity 
of the operator on 
$\ker(D_{u_{02}}\oplus\pi_{02}^\perp)^\perp\ni \xi^\nu_{02}$
to deduce (\ref{xito0}). 

Next, (\ref{xito0}) and the Sobolev trace theorem provide
$\Vert \xi^\nu_{02}|_{t=0} \Vert_{H^1(\R)} +
\Vert \xi^\nu_{02}|_{t=1} \Vert_{H^{1}(\R)}\to 0$,
and again using Lemma~\ref{lltrans} we can deduce that
\begin{align*}
& \Vert \hx^\nu|_{t=\delta^\nu} \Vert_{H^{1}(\R)} \\
&\leq C\bigl(
\Vert \pi_{02} \xi_{02}'^\nu|_{t=\delta^\nu} \Vert_{H^{1}(\R)}
+ \|\pi_{0211}^\perp \hat\xi^\nu|_{t=\delta^\nu} \|_{H^1(\R)}
+ \|(\xi_1'^\nu-\xi_1^\nu)|_{t=\delta^\nu}\|_{H^1(\R)} \bigl) \\
& \leq C\bigl(
 \Vert \xi_{02}^\nu|_{t=0} \Vert_{H^{1}(\R)}
+ \|(\xi_{02}'^\nu-\xi_{02}^\nu)|_{t=0}\|_{H^1(\R)}
+ \|\pi_{0211}^\perp \hat\xi^\nu|_{t=\delta^\nu}\|_{H^1(\R)} \\
&\qquad\qquad\qquad\qquad
+ \|\hat\xi^\nu|_{t=0} - \hat\xi^\nu|_{t=\delta^\nu}\|_{H^1(\R)}
+ \|(\xi_1'^\nu-\xi_1^\nu)|_{t=0}\|_{H^1(\R)} \bigl) 
\;\to 0 .
\end{align*}
Finally, combining this with (\ref{txii}) in case (b) implies
$$\Vert \hx^\nu \Vert_{L^2(\R\times[0,\delta^\nu])}\to 0$$
in contradiction to the assumption.
\end{proof}

Finally, we establish uniform exponential decay for the
solutions of 
Floer's equation (\ref{vvvv}) on the triple strip. 
For that purpose we introduce 
the following notation for integration over finite strips,
\begin{align*}
\int_{[0,1]\sqcup[0,\delta]} \left|\pd_s v(s,t) \right|^2 dt
:= \int_0^1 \left|\pd_s v_{02}(s,t) \right|^2 dt 
+ \int_0^\delta \left|\pd_s \hat v(s,t) \right|^2 dt ,
\end{align*}
and similarly for the $\cC^0$-norm
\begin{align*}
\left\| \pd_s v \right\|_{\cC^0_{1,\delta}([s_0,s_1])}
&:= \| \pd_s v_{02} \|_{L^\infty([s_0,s_1]\times[0,1])}
+ \| \pd_s \hat v \|_{L^\infty([s_0,s_1]\times[0,\delta])} , \\
d_{\cC^0_{1,\delta}([s_0,s_1])}(v , x^\pm )
&:= \sup_{(s,t)\in [s_0,s_1]\times[0,1]} d_{M_{02}}(v_{02}(s,t), x^\pm ), \\
&\quad +
\sup_{(s,t)\in [s_0,s_1]\times[0,\delta]} 
d_{M_{0211}}(\hat v(s,t), (x^\pm,x_1^\pm,x_1^\pm) ) .
\end{align*}

\begin{lemma}\label{expdecay}  
There are constants $\hbar,\Delta>0$ and $C$ such
that the following holds for every $\delta\in(0,1]$.  If
$v\in\widehat\M_\delta(x^-,x^+)$ is a smooth solution of~(\ref{vvvv})
satisfying
\begin{equation}\label{eq:longenergy}
\int_0^\infty \int_{[0,1]\sqcup[0,\delta]} \left|\pd_s v(s,t) \right|^2 dt ds 
\;<\; \hbar ,
\end{equation}
then for every $S\geq 3$
$$
d_{\cC^0_{1,\delta}([S,\infty))}(v , x^+ )^2
+ \left\|\pd_s v \right\|_{\cC^0_{1,\delta}([S,\infty))}^2
\le C e^{-\Delta S} 
\int_0^{2} \int_{[0,1]\sqcup[0,\delta]} 
\left| \pd_s v (s,t) \right|^2 dt ds ,
$$
and the analogous statement holds on $(-\infty,0]$ 
for the convergence to $x^-$.
\end{lemma}

\begin{proof}   
\noindent
{\bf Step 1: }{\it For every $\kappa>0$ there is an $\eps_\kappa>0$
such that the following holds for all $\delta\in(0,1]$. 
If $v\in\widehat\M_\delta(x^-,x^+)$ satisfies
(\ref{eq:longenergy}) with $\hbar=\eps_\kappa$, then }
\begin{equation} \label{step1}
%d_{\cC^0_{1,\delta}([1,\infty))}(v , x^+ ) + 
\left\|\pd_s v \right\|_{\cC^0_{1,\delta}([\frac 12,\infty))}
\le \kappa .
\end{equation}
Assume by contradiction that this is wrong. 
Then there exist $\kappa>0$ and sequences $\delta^\nu\in(0,1]$
and $v^\nu\in\widehat\M_{\delta^\nu}(x^-,x^+)$ such that
\begin{equation}\label{noenergy}
\lim_{\nu\to\infty}
\int_0^\infty\int_{[0,1]\sqcup[0,\delta^\nu]} 
\left|\pd_s v^\nu(s,t) \right|^2 dt ds  = 0 ,
\end{equation}
but the assertion fails. So after a time-shift we can assume that
$$
\left\|\pd_s v^\nu \right\|_{\cC^0_{1,\delta^\nu}([\frac 12,1])}
> \tfrac 12 \kappa .
$$
The equation  $\ol{\pd}_J v^\nu =0$ together with \eqref{noenergy} implies that 
$d v^\nu|_{s\geq 0} \to 0$ in the $L^2$-norm.
If $\delta^\nu$ is bounded away from zero, then the standard compactness\footnote{
If some of the symplectic manifolds are noncompact, see Section~\ref{sec:mon} for a variety of mild `bounded geometry' assumptions which can ensure the initial $\cC^0$-bound on the curves.
} 
for pseudoholomorphic curves with Lagrangian boundary conditions implies 
that $d v^\nu|_{s>0}\to 0$ in $\cC^\infty$ on every compact set
(for a subsequence), in contradiction to the assumption.
In the case $\delta^\nu\to 0$ the standard compactness theory still
implies $d v_{02}^\nu|_{(0,1]\times(0,\infty)}\to 0$
in $\cC^\infty$ on every compact set.
For $\hat v$ and $v_{02}$ near the boundary $t=0$
we obtain a $\cC^1$-bound from Lemma~\ref{bubb}.
So we obtain $\cC^0$-convergence of a subsequence
$v_{02}^\nu\to x_{02}$, $\hat v^\nu\to (x_{02},x_1,x_1)$ 
to constants $x_{02}\in L_0\times L_2$, $x_1\in M_1$
such that $(x_{02},x_1,x_1)\in L_{01}\times L_{12}$.
Now we can use the same compactness arguments as in the proof of 
Lemma~\ref{bubb} (step 2, using a cutoff function only in $s$) 
to deduce that $d v^\nu|_{s\in[\frac 12,1]}\to 0$ in the $\cC^0$-norm.
This again is a contradiction.

\medskip
\noindent
{\bf Step 2: }{\it 
There are constants $\eps_1>0$ and $C_1$ 
such that the following holds for all $\delta\in(0,1]$. 
If $v\in\widehat\M_\delta(x^-,x^+)$ satisfies
(\ref{eq:longenergy}) with $\hbar=\eps_1$, then}
$$
\| \pd_s v (1,\cdot) \|_{\cC^0([0,1]\sqcup[0,\delta])}^2
\leq C_1
\int_{[0,1]\sqcup[0,\delta]} \left|\nabla_t \pd_s v(1,t) \right|^2 dt .
$$
By contradiction we find sequences $\delta^\nu\in(0,1]$ and 
$v^\nu\in\widehat\M_{\delta^\nu}(x^-,x^+)$ that satisfy
(\ref{noenergy}), but there is no uniform constant $C_1$ with
which the estimate holds.
Then as in Step~1 we obtain (for a subsequence) $\cC^1$-convergence 
$v^\nu\to x=(x_{02},\hat x)$ 
on $[\frac 12,2]\times( [0,1]\sqcup[0,\delta^\nu])$
to constants $x_{02}\in L_0\times L_2$, $x_1\in M_1$
with $\hat x=(x_{02},x_1,x_1)\in L_{01}\times L_{12}$.
By assumption $L_{02}$ and $(L_0\times L_2)$ intersect 
transversely in $x_{02}$, and hence we have
for all $\xi_{02}:[0,1]\to T_{x_{02}}M_{02}$ with 
$\xi_{02}(1)\in T_{x_{02}}(L_0\times L_2)$
$$
\| \xi_{02} \|_{\cC^0([0,1])} \leq C \bigl(
\|\nabla_t\xi_{02}\|_{L^2([0,1])} 
+ \bigl|\pi_{02}^\perp\xi_{02}(0)\bigr| \bigr) .
$$
Now consider in addition 
$\hat\xi:[0,\delta]\to T_{\hat x}M_{0211}$ such that
$\hat\xi(\delta)\in T_{\hat x}(L_{01}\times L_{12})$
and $\xi|_{t=0}=(\xi_{02},\hx)|_{t=0}\in 
T_x(\Delta_{M_0\times M_2}\times\Delta_1)$.
We integrate for all $t\in[0,\delta]$
\begin{equation}\label{diff}
\bigl| \hat\xi(t) - \hat\xi(\delta) \bigr|
\leq \int_0^{\delta} |\nabla_t \hx(t) | dt 
\leq \sqrt{\delta} 
\biggl(\int_0^{\delta} |\nabla_t \hx(t) |^2 dt \biggr)^{1/2}.
\end{equation}
Combining this with Lemma~\ref{lltrans} and using the boundary conditions we obtain
$$
\bigl|\pi_{02}^\perp\xi_{02}(0)\bigr|
\leq \bigl|\pi_{0211}^\perp\hx(\delta)\bigr|
+ \bigl| \pi_{0211}^\perp \bigl( \hat\xi(0) - \hat\xi(\delta) \bigr) \bigr|
+ \bigl| \xi_1'(0)-\xi_1(0) \bigr|
\leq C \sqrt{\delta} 
\biggl(\int_0^{\delta} |\nabla_t \hx(t) |^2 dt \biggr)^{1/2} ,
$$
and thus
$$
\| \xi_{02} \|_{\cC^0([0,1])}^2
\leq C^2
\int_{[0,1]\sqcup[0,\delta]} \left|\nabla_t \xi \right|^2 dt .
$$
We moreover obtain from Lemma~\ref{lltrans}
with uniform constants $C,C',C''$
\begin{align*}
\bigl|\hx(\delta)\bigr|
& \leq C\bigl( \bigl|\pi_{02}\xi'_{02}(\delta)\bigr|
+ \bigl|( \xi_1'(\delta) - \xi_1(\delta) )\bigr| \bigr) \\
& \leq C'\bigl( \bigl|\xi_{02}(0)\bigr|
+ \bigl| \hat\xi(0) - \hat\xi(\delta) \bigr| \bigr) 
\;\leq\; C'' \biggl(\int_{[0,1]\sqcup[0,\delta]} 
\bigl|\nabla_t (\xi_{02},\hx) \bigr|^2 dt \biggr)^{1/2} .
\end{align*}
Together with (\ref{diff}) this implies
$$
\| \xi \|_{\cC^0([0,1]\sqcup [0,\delta])}^2 \leq C_1
\int_{[0,1]\sqcup[0,\delta]} \left|\nabla_t \xi \right|^2 dt 
$$
with some uniform constant $C_1$ for all $\delta\in(0,1]$
and all sections $\xi$ over $x$ satisfying the boundary conditions.
Due to the $\cC^1$-convergence $v^\nu\to x$ this estimate continues 
to hold with a uniform constant for sufficiently large $\nu$
for sections $\xi_{02}\in\cC^1([0,1],{v^\nu_{02}|_{s=1}}^*TM_{02})$,
$\hx\in\cC^1([0,\delta^\nu],{\hat v}|_{s=1}^*TM_{0211})$ 
that satisfy the analogous boundary conditions.
(We can write $v^\nu|_{s=1}=e_x(\zeta^\nu)$ with $\|\zeta^\nu\|_{\cC^1}\to 0$
and use $de_x(\zeta^\nu)^{-1}$ to map $(\xi_{02},\hx)$ to a 
section over $x$. This preserves the boundary conditions by construction of $e$.)
In particular, we can apply this new estimate to $\xi=\pd_s v^\nu|_{s=1}$,
which provides a uniform estimate and thus finishes the proof by contradiction.

\medskip
\noindent
{\bf Step 3: }{\it 
There are uniform constants $\eps_2,\Delta>0$ and $C_2$ 
such that the following holds for all for all $\delta\in(0,1]$.
If $v\in\widehat\M_\delta(x^-,x^+)$ satisfies
(\ref{eq:longenergy}) with $\hbar=\eps_2$,
then for all 
$s_0\geq 2$}
$$
\int_{[0,1]\sqcup[0,\delta]} | \pd_s v (s_0,t) |^2 dt
\leq C_2 e^{-\Delta s_0} 
\int_1^2 \int_{[0,1]\sqcup[0,\delta]} | \pd_s v (s,t) |^2 dt ds .
$$
Consider the function $f:[1,\infty)\to[0,\infty)$ defined by
$$
f(s):= \tfrac 12 \int_{[0,1]\sqcup[0,\delta]} |\pd_s v(s,t)|^2 dt .
$$
We can use the equation 
$\ol{\pd}_J v = 
\bigl(\pd_s v_{02}+ J_{02}(v_{02})\pd_t v_{02} \,,\,
\pd_s \hat v+ \hat J(\hat v)\pd_t \hat v \bigr) = 0 $
and the bound $\|\pd_s v\|_\infty\leq\kappa$ from Step~1 to calculate 
for $s\geq 1$
\begin{align*}
f''(s) 
&= \int_{[0,1]\sqcup[0,\delta]} \Bigl( 
|\nabla_s\pd_s v|^2
+ \lan \pd_s v \,,\, \nabla_s^2\pd_s v \ran \Bigr) \\
&= \int_{[0,1]\sqcup[0,\delta]} \Bigl( 
|J\nabla_t\pd_s v + (\nabla_{\pd_s v} J) \pd_t v|^2 
- \lan \pd_s v \,,\, J \nabla_t \nabla_s \pd_s v \ran \Bigr)\\
&\quad
- \int_{[0,1]\sqcup[0,\delta]} \Bigl( 
\lan \pd_s v \,,\, J R(\pd_s v,\pd_t v)\pd_s v
+ 2 (\nabla_{\pd_s v}J) \nabla_s\pd_t v 
+ \nabla_s(\nabla_{\pd_s v}J) \pd_t v \ran \Bigr) \\
&\geq 
\int_{[0,1]\sqcup[0,\delta]} \Bigl( 
2 |J\nabla_t\pd_s v|^2 + \pd_t\bigl( \om(\pd_s v,\nabla_s\pd_s v)\bigr)
- C |\pd_s v|^2 \bigl( |\pd_s v|^2 + |\nabla_t\pd_s v|\bigr)
\Bigr) \\
&\geq 
\bigl(2 - C\kappa\bigr) 
\int_{[0,1]\sqcup[0,\delta]} |J\nabla_t\pd_s v(s,t)|^2 dt
\;- C'\bigl(\kappa + \kappa^2\bigr) 
\|\pd_s v(s,\cdot)\|_{\cC^0([0,1]\cup[0,\delta])}^2 .
\end{align*}
The last step uses $2|\pd_s v|^2 |\nabla_t\pd_s v|\leq 
 \kappa |\pd_s v|^2 + \kappa |\nabla_t\pd_s v|^2$
and the claim 
$$
\biggl|\int_{[0,1]\sqcup[0,\delta]} 
\pd_t\bigl( \om(\pd_s v,\nabla_s\pd_s v)\bigr)\biggr|
\leq C\bigl(|\pd_s v_{02}(1)|^3 + |\pd_s \hat v(\delta)|^3 \bigr) .
$$
To prove the claim
we first use the diagonal boundary conditions to obtain
$$
\biggl|\int_{[0,1]\sqcup[0,\delta]} 
\pd_t\bigl( \om(\pd_s v,\nabla_s\pd_s v)\bigr)\biggr|
= \bigl| \om_{02}(\pd_s v_{02},\nabla_s\pd_s v_{02})|_{t=1}
+ \om_{02}(\pd_s \hat v,\nabla_s\pd_s \hat v)|_{t=\delta} \bigr| .$$
Then we use a smooth family of orthonormal frames
$(\gamma_i)_{i=1,\dots,k}\in\Gamma(T(L_0\times L_2))$ near
$w(s):=v_{02}(s,1)$ (and similarly for $\hat v$),
$$
\pd_s w (s) = \sum \lambda^i(s) \gamma_i(w(s))  ,\quad
\nabla_s\pd_s w (s) = \sum \Bigl( \pd_s\lambda^i(s) \gamma_i(w(s))
+ \lambda^i(s) \nabla_{\pd_s w(s)}\gamma_i \Bigr)
$$
with $\lambda:\R\to\R^k$. By the orthonormality we have
$|\lambda(s)|=|\pd_s w(s)|$, and using the identities 
$\omega(\gamma_i,\gamma_j)=0$ one obtains
$\bigl| \omega(\pd_s w ,\nabla_s\pd_s w) \bigr|\leq  C |\pd_s w|^3 $, 
where the constant $C$ only depends on $\nabla\gamma_i$.
Since $L$ is compact this holds with a uniform constant.

We can now choose $\kappa>0$ sufficiently small and then
fix $\hbar\leq\min\{\eps_1,\eps_\kappa\}$ such that 
Step~1 and Step~2 (applied to time-shifts of $v$)
together with the above calculation yield for all $s\geq 1$
$$
f''(s) 
\geq 
\int_{[0,1]\sqcup[0,\delta]} |J\nabla_t\pd_s v(s,t)|^2 dt
\geq 
((1+\delta)C_1)^{-1}
\int_{[0,1]\sqcup[0,\delta]} |\pd_s v(s,t)|^2 dt 
\geq \Delta^2 f(s) 
$$
with $\Delta>0$. Any such nonnegative convex function satisfies
for all $s\geq 2$ and $T\geq s$
$$
f(s)\leq C e^{-\Delta s} \biggl(
\int_{[1,2]} f(t) dt  + \int_{[2T,2T+1]} f(t) dt \biggr)
$$
with a constant $C$ that only depends on $\Delta$.
A detailed proof can be found in e.g.\ \cite[Lemma~3.7]{sw:floer}
(use the estimate for $\hat f(s-T-1)$, where the function 
$\hat f$ is shifted by $T+1$).
If we let $T\to\infty$ then $\int_{[2T,2T+1]} f(t) dt \to 0$
by the finite energy condition $\int_0^\infty f(s) ds < \hbar $,
and this proves the claim.

\medskip
\noindent
{\bf Step 4: }{\it 
There are constants $\eps_3>0$ and $C_3$ 
such that the following holds for all $\delta\in(0,1]$. 
If $v\in\widehat\M_\delta(x^-,x^+)$ satisfies
(\ref{eq:longenergy}) with $\hbar=\eps_3$, then}
$$
\left\|\pd_s v \right\|_{\cC^0_{1,\delta}([1,2])}
\leq C_3 \left\|\pd_s v \right\|_{L^2_{1,\delta}([\frac 12,\frac 52])} .
$$
By contradiction we find sequences $\delta^\nu\in(0,1]$ and 
$v^\nu\in\widehat\M_{\delta^\nu}(x^-,x^+)$ that satisfy
(\ref{noenergy}), but the assertion fails, i.e.\ we cannot find
a constant $C_3$ for which the estimate is satisfied.
Then as in Step~1 we obtain (for a subsequence) $\cC^1$-convergence 
$v^\nu\to x=(x_{02},\hat x)$ 
on $[\frac 12,\frac 52]\times( [0,1]\sqcup[0,\delta^\nu])$
to constants $x_{02}\in L_0\times L_2$, $x_1\in M_1$
with $\hat x=(x_{02},x_1,x_1)\in L_{01}\times L_{12}$.
So we can find sections $\xi^\nu\in \Gamma_{1,\delta^\nu}$ over $u=x$
such that $v^\nu|_{s\in[\frac 12,\frac 52]}=e_x(\xi^\nu)$.
The equation $\ol{\pd}_J v^\nu$ then becomes
$$
\nabla_s \xi^\nu + J(\xi^\nu) \nabla_t \xi^\nu = 0 
$$
and we have the boundary conditions 
$\nabla_s\xi_{02}^\nu|_{t=1}\in T_{x_{02}}(L_0\times L_2)$
and
$\nabla_s\hx^\nu|_{t=\delta^\nu}\in T_{\hat x}(L_{01}\times L_{12})$.
We fix two cutoff functions $h,\ti{h}\in\cC^\infty(\R,[0,1])$
with $h|_{[1,2]}\equiv 1$, $\ti{h}|_{\on{supp} h}\equiv 1$ and 
$\on{supp}(h),\on{supp}(\ti{h})\subset(\frac 12,\frac 52)$
and consider the sections $h\xi^\nu,\ti{h}\xi^\nu\in\Gamma_{1,\delta^\nu}$.
Note that $\pd_s v^\nu = d e_x(\xi^\nu)\nabla_s\xi^\nu$ with
$d e_x(\xi^\nu)\approx{\rm Id}$. 
So for sufficiently large $\nu$ we have
\begin{align*}
\|\pd_s v^\nu \|_{\cC^0_{1,\delta^\nu}([1,2])}
&\leq 2 \|h \nabla_s \xi^\nu  \|_{\cC^0_{1,\delta^\nu}}
\leq 2C_S \|h \nabla_s \xi^\nu  \|_{H^2_{1,\delta^\nu}} , \\
\|\nabla_s \xi^\nu  \|_{L^2_{1,\delta^\nu}([\frac 12,\frac 52])} 
& \leq 2 \|\pd_s v^\nu \|_{L^2_{1,\delta^\nu}([\frac 12,\frac 52])} ,
\end{align*}
where we used Lemma~\ref{Sobolev}.
Now we apply Lemma~\ref{Ddelta}~(b) to the sections $\xi=h\nabla_s\xi^\nu$
and $\xi=\ti{h}\nabla_s\xi^\nu$
(for which the boundary terms vanish since $\nabla_s\xi^\nu, 
\nabla_s^2\xi^\nu, \nabla_s^3\xi^\nu$ satisfy the boundary conditions)
and $\zeta=\xi^\nu$ (which satisfy $\|\xi^\nu\|_\infty\to 0$ and 
$\|\nabla\xi^\nu\|_\infty\to 0$) to obtain with uniform constants $C,C'$
\begin{align*}
\| h\nabla_s\xi^\nu \|_{H^2_{1,\delta^\nu}}
&\leq C_1 \bigl( 
\| \bigl(\nabla_s + J(\xi^\nu)\nabla_t\bigr) h \nabla_s \xi^\nu \|_{H^1_{1,\delta^\nu}}
+ \| h\nabla_s\xi^\nu \|_{H^0_{1,\delta^\nu}}  \bigr) \\
&= C_1 \bigl( 
\|  h' \nabla_s \xi^\nu \|_{H^1_{1,\delta^\nu}}
+ \| h\nabla_s\xi^\nu \|_{H^0_{1,\delta^\nu}}  \bigr) \\
&\leq C \| \nabla_s\xi^\nu \|_{H^1_{1,\delta^\nu}(\on{supp}h)}  
\;\leq\; C\|\ti{h}\nabla_s\xi^\nu \|_{H^1_{1,\delta^\nu}} \\
&\leq C C_1 \bigl( 
\| \bigl(\nabla_s + J(\xi^\nu)\nabla_t\bigr) \ti{h} \nabla_s \xi^\nu \|_{H^0_{1,\delta^\nu}}
+ \| \ti{h}\nabla_s\xi^\nu \|_{H^0_{1,\delta^\nu}}  \bigr) \\
&\leq C' \| \nabla_s\xi^\nu \|_{H^0_{1,\delta^\nu}([\frac 12,\frac 52]} .
\end{align*}
Now the contradiction follows,
\begin{align*}
\|\pd_s v^\nu \|_{\cC^0_{1,\delta^\nu}([1,2])}
\leq 2 \| h\nabla_s\xi^\nu \|_{H^2_{1,\delta^\nu}} 
\leq 2C' \| \nabla_s\xi^\nu \|_{H^0_{1,\delta^\nu}([\frac 12,\frac 52])} 
\leq 4C' \|\pd_s v^\nu \|_{L^2_{1,\delta^\nu}([\frac 12,\frac 52])} .
\end{align*}

\noindent
{\bf Step 5: }{\it 
We prove the claim, that is for every $s\geq 3$
$$
d_{\cC^0([0,1]\sqcup[0,\delta])}(v(s,\cdot) , x^+ )^2
+ \left\|\pd_s v(s,\cdot) \right\|_{\cC^0([0,1]\sqcup[0,\delta])}^2
\le C e^{-\Delta s} E'(v)
$$
with}
$$
E'(v):=\int_0^2 \int_{[0,1]\sqcup[0,\delta]} | \pd_s v (s,t) |^2 dt ds .
$$
We choose $\hbar=\min\{\eps_2,\eps_3\}$, then Step~3 and Step~4 
(applied to appropriately shifted solutions) combine as follows
for all $s\geq 3$
\begin{align*}
\left\|\pd_s v \right\|_{\cC^0_{1,\delta}([s-\frac 12,s+\frac 12])}^2
&\leq C_3^2 
\int_{s-1}^{s+1} 
\int_{[0,1]\sqcup[0,\delta]} | \pd_s v (s,t) |^2 dt \\
&\leq C_3^2 C_2 \int_{s-1}^{s+1} e^{-\Delta s} E'(v) ds  
\;\leq\; C_3^2 C_2 \Delta^{-1} e^{\Delta} e^{-\Delta s}  E'(v) .
\end{align*}
This proves the second part of the claim.
The estimate on
$d_{\cC^0([0,1]\sqcup[0,\delta])}(v(S,\cdot) , x^+ )$
now simply follows by integration: For all $S\geq 3$ and
$t\in[0,1]$
\begin{eqnarray*}
d_{M_{02}}(v_{02}(S,t) , x^+ )
&\leq&
\int_S^\infty |\pd_s v_{02}(s,t)| ds \\
&\le& C \int_S^\infty e^{-\Delta s/2} \sqrt{E'(v)} ds  \\
&=& 2C\Delta^{-1} e^{-\Delta S/2} \sqrt{E'(v)} ,
\end{eqnarray*}
and similarly for $\hat v$.
\end{proof}

\subsection{Compactness} \label{cpt}

The surjectivity of the map
$\T_\delta: \M^1_0(x^-,x^+) \to \M^1_\delta(x^-,x^+)$,
as introduced in the previous section,
will be a direct consequence of the following compactness result. 
Here we choose $\eps_0\in(0,\eps]$ with $\eps>0$ from in Theorem~\ref{solving}.
Then $v=e_u(\xi)$ with $\xi\in\Gamma_{1,\delta}(\eps_0)\cap K_0$
implies that $[v_u]=\T_\delta([u])$ by the definition of $\T_\delta$
via theorem~\ref{solving}.
We will denote the time-shift by $\tau^\sigma v (s,t):= v(\sigma + s, t)$.

\begin{theorem} \label{compact}
Given $\eps_0>0$ there exists $\delta_0>0$ such that for 
every $\delta\in(0,\delta_0]$ and $v\in\widehat\M^1_\delta(x^-,x^+)$
there exist $u\in\tM^1_0(x^-,x^+)$ and $\sigma\in\R$ such that
$\tau^\sigma v=e_u(\xi)$ with $\xi\in\Gamma_{1,\delta}\cap K_0$
and $\|\xi\|_{\Gamma_{1,\delta}}\leq \eps_0$.
Moreover, the moduli space $\widehat\M^1_\delta(x^-,x^+)$ is regular
for all $\delta\in(0,\delta_0]$ in the sense that the linearized 
operator $D_{v}$ is surjective for every 
$v\in\widehat\M^1_\delta(x^-,x^+)$.
\end{theorem}
\begin{proof}
We assume by contradiction that there is an $\eps_0>0$,
a sequence $\delta^\nu\to 0$, and solutions 
$v^\nu=(v^\nu_{02},\hat v^\nu)\in\widehat\M^1_{\delta^\nu}(x^-,x^+)$ 
for which the assertion of the theorem fails.
Their energy is fixed, 
$E(v^\nu) =  \tfrac 12 \tau + \tfrac 12 c(x_-,x_+)$,
by the analogue of Proposition~\ref{prop monotone} for strips of different widths: For any pair of maps $(v_{02},\hat v)$ that are not necessarily pseudoholomorphic but satisfy the limits and seam conditions of $\widehat\M^1_{\delta}(x^-,x^+)$ we have
\begin{align}
E(v_{02},\hat v) &= \int v_{02}^*\bigl((-\omega_0)\oplus\omega_2\bigr)
+\int \hat v^*\bigl(\omega_0\oplus(-\omega_2)\oplus (-\omega_1)\oplus \omega_1 \bigr)  \nonumber \\
&=  \tfrac 12 \tau {\rm Ind}(D_{(v_{02},\hat v)}) + \tfrac 12 c_{\delta}(x_-,x_+) . \label{delta mon}
\end{align}
Here $c_{\delta}(x_-,x_+)$ is independent of $\delta$ since the equations for different $\delta$ apply to the same map, rescaled to different widths, which has the same energy and index.
Next, we can exclude bubbling by the following argument
based on Lemma~\ref{bubb} below:

If $|d v_{02}^\nu|$ is unbounded near a point $z\in\R\times(0,1]$,
then the standard rescaling method gives rise to a nontrivial
pseudoholomorphic sphere or disk in $(M_0,L_0)$, or in $(M_2,L_2)$,
or in both.\footnote{
In case $M_0$ or $M_2$ are noncompact, this convergence can be ensured by `bounded geometry' assumptions, or energy concentration can be proven directly, as outlined in Section~\ref{sec:mon}.
}
Thus some fixed amount of energy $\hbar>0$ would have to concentrate
near $z$.  The same energy quantization holds for blowup of $d \hat v^\nu$
or $d v_{02}^\nu|_{t=0}$ by Lemma~\ref{bubb}.  So the energy densities $|d
v^\nu|$ can only blow up at finitely many points.  On the complement
the same compactness proof as in the next paragraph
provides a $\cC^0_{\rm loc}$
convergent subsequence $v_{02}^\nu\to u_{02}$, where the limit
corresponds to a solution $u\in\tM_0(y^-,y^+)$ with finitely
many singularities and energy $E(u)<E(v^\nu)$. 
The singularities can be removed by the standard proofs for pseudoholomorphic curves with Lagrangian boundary condition \cite[Theorem 4.1.2]{ms:jh}, so we would obtain a solution $\tilde u\in\tM_0(y^-,y^+)$ of energy $E(\tilde u)<E(v^\nu)$.
Its limits $y_\pm$ may not be the same as those of $v^\nu$, in which case we find a sequence of trajectories $\ul{\tilde u}=(\tilde u_1,\ldots,\tilde u_N)\subset \tM_0(\cdot,\cdot)$ 
connecting $x_-$ to $x_+$, with total energy $E(\ul{\tilde u})=\sum E(\tilde u_j)<E(v^\nu)$.
We claim that monotonicity forces $\ul{\tilde u}$ to have total index 
$\sum{\rm Ind}(D_{\tilde u_j})<{\rm Ind}(D_{v^\nu})=1$, and hence
by regularity of the moduli spaces $\tM_0(\cdot,\cdot)$ consists of 
a single constant trajectory. This however would mean that $v^\nu$ were self-connecting trajectories of  
$x_-= x_+$, i.e.\ we have annuli with ${\rm Ind}(D_{v^\nu})=1$ -- in contradiction to assumption \eqref{d}.
To control the index of $\ul{\tilde u}$ we glue the trajectories to a single map $\tilde w:\R\times[0,1]\to M_0^-\times M_2$ satisfying all limit and boundary conditions of $\tM_0(x^-,x^+)$ except for holomorphicity. Its index and energy coincide with the total energy and index of $\ul{\tilde u}$. With that we obtain
$$
\tau {\rm Ind}(D_{\tilde w}) + c(x_-,x_+) 
= 2 E(\tilde w) < 2 E(v^\nu) =
 \tau + c(x_-,x_+) 
$$
from the monotonicity formula \eqref{delta mon} 
and the index and energy identities in Lemma~\ref{index} applied to
$(\tilde w,\hat w)$, where $\hat w$ is the $t$-independent map given by the lift of $\tilde w|_{t=0}\subset L_{02}$ to $(L_{01}\times_{\Delta_1}L_{12})^T$.
This proves $\sum{\rm Ind}(D_{\tilde u_j})\leq 0$ as claimed and hence
excludes bubbling. 

\medskip

So from now on we assume that $|d v^\nu|\leq C_0$ is uniformly bounded. 
In addition, the boundary condition in the compact set $(L_{01}\times L_{12})^T$ implies a priori bounds on the map $\hat v^\nu|_{t=\delta_\nu}$. Together with the uniform gradient bound on $\hat v^\nu$ this implies a priori bounds on $\hat v^\nu|_{t=0}$, which transfer to $v^\nu_{02}|_{t=0}$ via the boundary condition. Finally using the uniform gradient bounds on $v^\nu_{02}$ we then obtain uniform $\cC^1$-bounds on both $v^\nu_{02}$ and $\hat v^\nu$ on every ball $B_R(0)$ of fixed radius (intersected with the domain of the respective map).
So we can fix $p>2$ and find a subsequence and map $u_{02} \in \cC^0 \cap W^{1,p}_{\loc}(\R \times [0,1],M_0 \times M_2)$ such that $v_{02}^\nu\to u_{02}$ in the $\cC^0$-topology and the weak $W^{1,p}$-topology on every compact subset of $\R\times[0,1]$.
We claim that the limit $u_{02}$ corresponds to a solution $(u_0,u_2)\in\tM_0^1(x^-,x^+)$.  Indeed, the holomorphicity follows from the weak $W^{1,p}$-convergence, and the boundary condition follows from $d_{\cC^0}\bigl(v^\nu_{02}|_{t=0}, L_{02} \bigr)\to 0$. To check the latter, recall that Lemma~\ref{lltrans}~(a) bounds this distance in terms of $d_{\cC^0}\bigl(\hat v^\nu|_{t=0}, (L_{01}\times L_{12})^T \bigr)$, which due to the boundary conditions on $\hat v^\nu|_{t=\delta^\nu}$ is bounded by
$d_{\cC^0}\bigl( \hat v^\nu|_{t=\delta^\nu},\hat v^\nu|_{t=0} \bigr)
\leq C_0\delta^\nu$. 

We also conclude that $\hat v\to\bu=(u_{02}|_{t=0},\bu_1,\bu_1)$ 
in $\cC^0([-T,T]\times[0,\delta^\nu])$ for all $T>0$, where
$\bu_1$ is determined uniquely by $(u_{02}|_{t=0},\bu_1,\bu_1)\in
L_{01}\times L_{12}$.  Indeed, $\hat
v^\nu|_{t=0}=(v^\nu_{02},v^\nu_1,v^\nu_1)|_{t=0}$ satisfies
$d_{\cC^0}( \hat v^\nu|_{t=0} , u_{02}\times\Delta_1 ) \to 0$ as well
as $d_{\cC^0}( \hat v^\nu|_{t=0} , L_{01}\times L_{12} ) \leq
d_{\cC^0}( \hat v^\nu|_{t=0} , \hat v^\nu|_{t=\delta^\nu}) \to 0 $, so
$v_1|_{t=0}$ must converge to $\bu_1$ on compact sets, and the
convergence for $t_0\in[0,\delta^\nu]$ follows from $d_{\cC^0}( \hat
v^\nu|_{t=0} , \hat v^\nu|_{t=t_0}) \leq C_0\delta^\nu\to 0 $.  

In summary we have $v^\nu\to u:=(u_{02},\bu)$ in the $\cC^0$-topology on
every set $\{|s|\leq T\}$ for fixed $T$.  In the following, we will
strengthen this convergence using uniform nonlinear estimates and
exponential decay, to find sections
$\xi^\nu\in\Gamma_{1,\delta^\nu}(\eps_0)$ such that
$v^\nu=e_u(\xi^\nu)$ and $D_{v^\nu}$ is surjective in contradiction to
the assumption.
Let us first note that, by the same monotonicity arguments as above, the limit must be a nonbroken trajectory $u\in\tM_0^1(x^-,x^+)$ of the same index and energy $E(u)=E(v^\nu)$.  
In the next step we strengthen the local convergence.

For fixed $T>0$ and sufficiently large $\nu\ge\nu_0$ we can write
$v^\nu|_{\{|s|\le T\}}=e_u(\xi^\nu)$ with a section
$\xi^\nu\in\Gamma_{1,\delta^\nu}$ (extended smoothly to $\{|s|>T\}$).
The extension of $\xi^\nu$ can be chosen such that 
$\|\xi^\nu\|_\infty \to 0$ and $\sup_\nu\|\nabla\xi^\nu\|_\infty<\infty$
follows from the $\cC^0$-convergence and $\cC^1$-boundedness of 
$v^\nu|_{\{|s|\le T\}}$.
For the latter note that
$\nabla\xi^\nu =  d e_u(\xi^\nu)^{-1} \nabla v^\nu 
- \pd_1e(u,\xi^\nu) \nabla u $, where $\nabla v^\nu$ is uniformly 
bounded, and $d e_u(\xi^\nu)\to{\rm Id}$ as $|\xi^\nu|\to 0$.
This puts us into the position where Lemma~\ref{Ddelta} 
applies with $\zeta=\xi^\nu$.
We fix a cutoff function $h\in\cC^\infty_0([-T,T],[0,1])$
with $h|_{[-T+1,T-1]}\equiv 1$, then
\begin{align*}
\| h\xi^\nu \|_{H^1_{1,\delta}} 
&\leq C_1 \Bigl( 
\| \bigl(\nabla_s + J(\xi^\nu)\nabla_t \bigr) h \xi^\nu \|_{H^0_{1,\delta^\nu}}
+ \| h \xi^\nu \|_{H^0_{1,\delta^\nu}}  \\
&\qquad\qquad
+ \| h \hx^\nu|_{t=\delta^\nu} \|_{H^0(\R)} 
+ \| h \xi_{02}^\nu|_{t=1} \|_{H^0(\R)}  \Bigr) .
\end{align*}
Now we can use (\ref{dbarxi}), $\ol{\partial}_J v^\nu = 0$, 
$\ol{\partial}_{J_{02}} u_{02} = 0$, and $\pd_t\bu=0$ to obtain
\begin{align*}
\bigl\| h \bigl( \nabla_s + \hat J(\hx^\nu) \nabla_t \bigr) \hx^\nu 
\bigr\|_{L^2(\R\times[0,\delta^\nu])}
&= \bigl\| h \cdot 
de_\bu(\hx^\nu)^{-1} \bigl(\pd_1 e(\bu,\hx^\nu) \pd_s \bu
 \bigr) \bigr\|_{L^2([-T,T]\times[0,\delta^\nu])} \\
&\leq C \|\pd_s\bu\|_{L^2([-T,T]\times[0,\delta^\nu])}
\leq C \sqrt{\delta^\nu}\|\pd_s\bu\|_{L^2([-T,T])} ,
\end{align*}
and furthermore, using the fact that $\partial_1 e(u_{02},0)={\rm Id}$ 
commutes with $J(u_{02})$,
\begin{align*}
&\bigl\| h \bigl( \nabla_s + J_{02}(\xi_{02}^\nu) \nabla_t \bigr) 
\xi_{02}^\nu \bigr\|_{L^2(\R\times[0,1])} \\
&= \bigl\| h \cdot de_{u_{02}}(\xi_{02}^\nu)^{-1} \bigl( 
\pd_1 e(u_{02},\xi_{02}^\nu) J(u_{02})\pd_t u_{02}  
- J_{02}(u_{02}) \pd_1 e(u_{02},\xi_{02}^\nu)\pd_t u_{02} \bigr) 
\bigr\|_{L^2(\R\times[0,1])}  \\
& \leq C \|\xi_{02}^\nu\|_{L^2([-T,T]\times[0,1])} .
\end{align*}
Hence we have
\begin{align*}
\| \xi^\nu \|_{H^1_{1,\delta^\nu}(\{|s|\leq T-1\})} 
&\leq C \Bigl( \sqrt{\delta^\nu}
+ \| \xi^\nu \|_{H^0_{1,\delta^\nu}(\{|s|\le T\})} 
+ \| h \hx^\nu|_{t=\delta^\nu} \|_{H^0(\R)} 
+ \| h \xi_{02}^\nu|_{t=1} \|_{H^0(\R)}  \Bigr),
\end{align*}
which converges to zero,
and thus $v_{02}^\nu\to u_{02}$ in the $H^1$-norm on every compact set.
Now we can verify the assumptions of Lemma~\ref{expdecay}
(with the constant $\hbar>0$) and achieve uniform exponential decay:
Pick $T>0$ such that 
$\int_{[-T,T]\times[0,1]} |\pd_s u_{02}|^2 \ge E(u) - \frac 12 \hbar$
and pick $\nu_0$ such that for all $\nu\geq\nu_0$ we have
$\| \pd_s u_{02} \|_{L^2([-T,T]\times[0,1])}^2 
- \| \pd_s v_{02}^\nu \|_{L^2([-T,T]\times[0,1])}^2 \leq \frac 12\hbar$
and thus
$$
\int_{\{|s|>T\}} \biggl( \int_{[0,1]} |\pd_s v^\nu_{02}|^2 
+ \int_{[0,\delta^\nu]} |\pd_s \hat v|^2  \biggr)
\;\le\; E(v^\nu) + \tfrac 12 \hbar - E(u) + \tfrac 12\hbar 
\;=\; \hbar .
$$
Now the exponential decay Lemma~\ref{expdecay} combined with the
local $\cC^0$-convergence implies that
$$
d_{\cC^0}(v_{02}^\nu,u_{02}) + d_{\cC^0}(\hat v^\nu,\bu) \to 0 
$$
uniformly for all $s,t$.
Thus for sufficiently large $\nu$ we can write $v^\nu=e_u(\xi^\nu)$
with $\xi^\nu\in H^2_{1,\delta^\nu}$ and $\|\xi^\nu\|_\infty \to 0 $.
In fact, the uniform exponential decay implies global convergence,
$$
\|\xi^\nu\|_\infty \to 0,\qquad\qquad
\|\xi^\nu\|_{L^p_{1,\delta}} \to 0 \quad \forall p\geq 1 ,
\qquad\qquad \|\nabla\xi^\nu\|_\infty \leq c_0 <\infty.
$$
This puts us into the position where Lemma~\ref{Ddelta} and
~\ref{D02} apply with $\zeta=\xi^\nu$,
\begin{align*}
& \| \xi^\nu \|_{H^2_{1,\delta^\nu}} + \| \nabla \xi^\nu \|_{L^4_{1,\delta^\nu}} \\
&\leq C_1 \Bigl( 
\| \nabla_s \xi^\nu + J(\xi^\nu)\nabla_t\xi^\nu \|_{H^1_{1,\delta^\nu}}
+ \| \nabla_s \xi^\nu + J(\xi^\nu)\nabla_t\xi^\nu \|_{L^4_{1,\delta^\nu}} \\
&\qquad\qquad\qquad
+ \| \xi^\nu \|_{H^0_{1,\delta^\nu}} 
+ \| \hx^\nu|_{t=\delta^\nu} \|_{H^1(\R)} 
+ \| \xi_{02}^\nu|_{t=1} \|_{H^1(\R)}  \Bigr) \\
&\leq C_1(1+C_2) \Bigl( 
\| \nabla_s \xi^\nu + J(\xi^\nu)\nabla_t\xi^\nu \|_{H^1_{1,\delta^\nu}}
+ \| \nabla_s \xi^\nu + J(\xi^\nu)\nabla_t\xi^\nu \|_{L^4_{1,\delta^\nu}} \\
&\qquad\qquad\qquad\qquad\qquad\qquad\qquad
+ \| \xi^\nu \|_{H^0_{1,\delta^\nu}} 
+ \sqrt{\delta^\nu} \Vert \nabla_t\hx^\nu \Vert_{H^1(\R\times[0,\delta^\nu])}  \Bigr).
\end{align*}
The terms in the last line converge to zero or can be absorbed into
the left hand side for $\delta^\nu$ sufficiently small.
We claim that the penultimate line also converges to zero and we thus
obtain the convergence $\| \xi^\nu \|_{\Gamma_{1,\delta}}\to 0$.
To check this we recall from (\ref{dbarxi})
that $\ol{\partial}_J v^\nu = 0$ implies
\begin{equation} \label{quark}
\nabla_s \xi^\nu + J(\xi^\nu) \nabla_t \xi^\nu
= - de_u(\xi^\nu)^{-1} \bigl( \pd_1 e(u,\xi^\nu)\pd_s u 
+ J(u) \pd_1 e(u,\xi^\nu)\pd_t u \bigr) .
\end{equation} 
Recall that 
\begin{equation} \label{pd2}
\pd_1 e(u,0) = {\rm Id}_{T_u M}, \qquad 
\pd_2 e(u,0)=de_u(0)={\rm Id}_{T_u M}.
\end{equation}  
So in zeroth order we have, using the equations $\pd_t\bu=0$ and
$\pd_s u_{02}=-J_{02}(u_{02})\pd_t u_{02}$,
\begin{align*}
\bigl| \nabla_s \hx^\nu + \hat J(\hx^\nu) \nabla_t \hx^\nu \bigr|
&\leq
\bigl| de_{\bu}(\hx^\nu)^{-1} \bigl( \pd_1 e(\bu,\hx^\nu) \pd_s\bu \bigr) \bigr|
\leq
C | \pd_s\bu |, \\
\bigl| \nabla_s \xi_{02}^\nu + J_{02}(\xi_{02}^\nu) \nabla_t \xi_{02}^\nu \bigr| 
&\leq
\bigl| de_{u_{02}}(\xi^\nu_{02})^{-1} \bigl( \pd_1 e(u_{02},\xi_{02}^\nu) J_{02}(u_{02}) \\
&\qquad\qquad\qquad\quad
- J_{02}(u_{02}) \pd_1 e(u_{02},\xi_{02}^\nu) \bigr)   \pd_t u_{02} \bigr| 
\;\leq\;
C | \xi^\nu_{02} | ,
\end{align*}
and thus
\begin{align*}
&\| \nabla_s \xi^\nu + J(\xi^\nu)\nabla_t\xi^\nu \|_{L^2_{1,\delta^\nu}}
+ \| \nabla_s \xi^\nu + J(\xi^\nu)\nabla_t\xi^\nu \|_{L^4_{1,\delta^\nu}} \\
&\leq
C\bigl( \|\xi_{02}^\nu\|_{L^2(\R\times[0,1])} + \|\xi_{02}^\nu\|_{L^4(\R\times[0,1])}
+ (\delta^\nu)^{1/2}\|\pd_s\bu\|_{L^2(\R)} 
+ (\delta^\nu)^{1/4} \|\pd_s\bu\|_{L^4(\R)} \bigr)
\to 0 .
\end{align*}
For the first derivative we calculate
from \eqref{quark}, 
denoting all uniform constants by $C$,
\begin{align*}
\bigl|\nabla\bigl( \nabla_s \hx^\nu + \hat J(\hx^\nu) \nabla_t \hx^\nu \bigr)\bigr|
&\leq
C (1+ |\nabla\hx^\nu|) \bigl| \pd_1 e(\bu,\hx^\nu)\pd_s\bu \bigr| 
+ C \bigl| \nabla \bigl( \pd_1 e(\bu,\hx^\nu)\pd_s\bu \bigr) \bigr|  \\
&\leq  C \bigl( 1 + |\nabla\hx^\nu| \bigr) 
       \bigl( | \pd_s\bu | + |\nabla_s\pd_s \bu| \bigr) ,
\end{align*}
and (in between dropping the subscript from $\xi_{02}^\nu$)
\begin{align*}
\bigl|\nabla\bigl( \nabla_s \xi_{02}^\nu 
+ J_{02}(\xi_{02}^\nu) \nabla_t \xi_{02}^\nu \bigr)\bigr|
&\leq
C (1+|\nabla\xi^\nu|)
 \bigl| \pd_1 e(u,\xi^\nu)J(u)\pd_t u - 
J(u) \pd_1 e(u,\xi^\nu)\pd_t u \bigr| \\
&\quad
+ C \bigl| \nabla \bigl( \pd_1 e(u,\xi^\nu)J(u) 
- J(u) \pd_1 e(u,\xi^\nu) \bigr)\bigr|
\cdot |\pd_t u| \\
&\quad
+ C \bigl| \pd_1 e(u,\xi^\nu)J(u) 
- J(u) \pd_1 e(u,\xi^\nu) \bigr| \cdot |\nabla\pd_t u| \\
&\leq
C  | \xi^\nu_{02} | \bigl( 1 + |\nabla\xi^\nu_{02}| \bigr) .
\end{align*}
Here the estimate for the second summand follows from (\ref{pd2}) and
the identity
$$
\nabla_s(\pd_1 e(u,\xi) X) 
= \pd_1 e(u,\xi)\nabla_s X + (\nabla_{(\pd_s u,\nabla_s\xi)}\pd_1 e )(u,\xi) X
$$
(and similarly for $\nabla_t(\pd_1 e(u,\xi) X)$), where we have
$
\bigl(\nabla_{(\pd_s u,\nabla_s\xi)}\pd_1 e \bigr)(u,0) =0
$
since
$$
\bigl(\nabla_{(Y,0)} \pd_1 e \bigr)(u,0)
= \nabla_Y {\rm Id}_{T_u M} = 0 
$$
and, calculating in local normal coordinates
with an extension $\tilde Y\in\Gamma(TM)$ of $Y\in T_u M$
that is covariantly constant along $\tau\mapsto \exp_u(\tau X)$, 
$$
\bigl(\nabla_{(0,Y)}\pd_1 e \bigr)(u,0) X
= \pd_\sigma|_{\sigma=0} \pd_\tau|_{\tau=0} e(\exp_u(\tau X) , \sigma Y) 
= \pd_\tau|_{\tau=0} \tilde Y(\exp_u(\tau X)) 
= 0.
$$
Now the uniform estimate $\|\nabla\xi^\nu\|_\infty\leq c_0$ 
and the exponential decay of $\bu=\bu(s)$ imply
\begin{align*}
\bigl\| \nabla\bigl( \nabla_s \xi^\nu 
+ J(\xi^\nu)\nabla_t\xi^\nu \bigr) \bigr\|_{L^2_{1,\delta^\nu}}
\leq
C(1+c_0)\bigl( \|\xi_{02}^\nu\|_{L^2(\R\times[0,1])} 
+ (\delta^\nu)^{1/2}\|\pd_s\bu\|_{H^1(\R)} \bigr) \to 0. 
\end{align*}
This proves 
$$
\| \xi^\nu \|_{\Gamma_{1,\delta^\nu}}\to 0 .
$$
It remains to find a time-shift such that
$\tau^{\sigma}v^\nu=e_u(\xi^\nu(\sigma))$
with some $\xi^\nu(\sigma)\in K_0$
but still $\| \xi^\nu(\sigma) \|_{\Gamma_{1,\delta^\nu}}\leq\eps_0$.
In order to find this shift we write 
$\tau^\sigma v^\nu=e_u(\xi^\nu(\sigma))$ with
\begin{equation} \label{xinusigma}
\xi^\nu(\sigma) := \bigl( e_u^{-1} \circ \tau^\sigma \circ e_u \bigr)(\xi^\nu) 
\in \Gamma_{1,\delta^\nu} . \end{equation} 
This will satisfy
$$
\| \xi^\nu(\sigma) \|_{\Gamma_{1,\delta^\nu}} 
\leq C \bigl( \| \xi^\nu \|_{\Gamma_{1,\delta^\nu}} 
+ |\sigma| \| d u \|_{\Gamma_{1,\delta^\nu}} \bigr) ,
$$
so it is well defined whenever $|\sigma|\leq\sigma_0$, where we fixed
$\sigma_0=\frac 12 \eps_0 C^{-1}\| d u \|_{\Gamma_{1,\delta^\nu}}^{-1}$ 
such that $\| \xi^\nu (\sigma) \|_{\Gamma_{1,\delta^\nu}}\leq\eps_0$ 
is ensured for sufficiently large $\nu\geq\nu_0$.
The $L^2$-estimate on $\xi^\nu(\sigma)$ can be seen from the pointwise estimate
\begin{align*}
\bigl| e_u^{-1}\tau^\sigma e_u (\xi) \bigr|
& \leq  \bigl| e_u^{-1}\tau^\sigma e_u (\xi) - e_u^{-1}\tau^\sigma e_u (0) \bigr|
+  \bigl| e_u^{-1}\tau^\sigma u - e_u^{-1} u \bigr| \\
& \leq  C \bigl( d\bigl(\tau^\sigma e_u (\xi) , \tau^\sigma e_u(0) \bigr)
+  d \bigl( \tau^\sigma u , u \bigr)  \bigr) \\
& \leq  C \bigl( \bigl|\tau^\sigma \xi \bigr|
+  \sigma | \partial_s u |  \bigr) .
\end{align*}
Here $C$ is a continuity constant for $e_u^{-1}$.
The higher derivatives of $\xi(\sigma)=e_u^{-1}\tau^\sigma e_u (\xi)$ 
are estimated similarly. Now consider the function
$$
\Theta^\nu(\sigma):=\lan \xi^\nu_{02}(\sigma) , \pd_s u_{02} \ran_{L^2}.
$$
It satisfies
$$
|\Theta^\nu(0)|\leq
\| \pd_s u \|_{L^2_{1,\delta^\nu}} 
\| \xi^\nu \|_{L^2_{1,\delta^\nu}} \to 0 
$$
and (dropping the $02$-subscript) we obtain from \eqref{xinusigma}
\begin{align*}
&\Bigl| \tfrac\pd{\pd\sigma}\Theta^\nu (\sigma) 
- \| \pd_s u \|_{L^2}^2 \Bigr| \\
&= \Bigl| 
\lan \bigl( d e_u(\xi^\nu(\sigma))^{-1}
\tau^\sigma \bigl( \pd_1 e(u,\xi^\nu) \pd_s u 
+ d e_u (\xi^\nu) \pd_s \xi^\nu \bigr) 
- \tau^\sigma\pd_s u \bigr), \pd_s u \ran_{L^2} \\
&\qquad\qquad\qquad\qquad\qquad\qquad\qquad\qquad\qquad\quad
+ \lan \bigl( \tau^\sigma \pd_s u - \pd_s u \bigr) , 
\pd_s u \ran_{L^2} \Bigr| \\
&\leq 
C \bigl(  \|\xi^\nu\|_{H^1} \|\pd_s u\|_{L^2}
+ \|\xi^\nu\|_\infty \|\pd_s u\|_{L^2}^2 
+ |\sigma| \|\nabla_s\pd_s u \|_{L^2} \| \pd_s u \|_{L^2} \bigr) .
\end{align*}
The latter is an arbitrarily small error for large $\nu$ and small $\sigma$.
Hence we will find solutions
$\sigma^\nu \sim - \Theta^\nu(0)/\|\pd_s u_{02}\|_{L^2}^2 
\in[-\sigma_0,\sigma_0]$ of $\Theta^\nu(\sigma^\nu)=0$.
With these we have
$\tau^{\sigma^\nu}v^\nu=e_u(\xi^\nu(\sigma))$, where
$\xi^\nu\in 
K_0 = \bigl\{ \xi\in\Gamma_{1,\delta} \big| 
\lan \xi_{02} , \pd_s u_{02} \ran_{L^2} = 0 \bigr\}$
and $\| \xi^\nu (\sigma) \|_{\Gamma_{1,\delta^\nu}}\leq\eps_0$.
So with this small time-shift on $v^\nu$ we obtain a contradiction to
the assumption that $\T_\delta^\nu$ is not surjective.  

Finally, to prove the transversality we need to check that 
$D_{v^\nu}=D_{e_u(\xi^\nu)}$ is surjective.
(The same then holds for the time shifts $\tau^{\sigma^\nu}v^\nu$.)
This follows from the quadratic estimate in Lemma~\ref{quadratic} :
Let $Q:\Om_{1,\delta^\nu}\to \Gamma_{1,\delta^\nu}$ be the right
inverse of $D^\delta=d\F_u(0)$, then 
\begin{eqnarray*}
\| \Phi_u(\xi^\nu)^{-1}D_{e_u(\xi^\nu)}E_u(\xi^\nu)Q - {\rm Id} \|
&\leq& 
\| \Phi_u(\xi^\nu)^{-1}D_{e_u(\xi^\nu)}E_u(\xi^\nu) - d\F_u(0) \| \cdot \|Q\|\\
&\leq& 2 C_2\|Q\| \|\xi^\nu\|_{\Gamma_{1,\delta^\nu}} ,
\end{eqnarray*}
where $\|Q\|<\infty$ by (\ref{right inverse}) 
and $\|\xi^\nu\|_{\Gamma_{1,\delta^\nu}}\to 0$.
This shows that 
$\Phi_u(\xi^\nu)^{-1}D_{e_u(\xi^\nu)}E_u(\xi^\nu) Q$ 
and hence also the operator
$\Phi_u(\xi^\nu)^{-1}D_{e_u(\xi^\nu)}E_u(\xi^\nu)$
has a right inverse for all sufficiently large $\nu\geq\nu_0$.
Here the parallel transport $\Phi_u(\xi^\nu)$ is an isomorphism 
on the target and $E_u(\xi^\nu)$ identifies $\Gamma_{1,\delta}$
with the domain of $D_{e_u(\xi^\nu)}$. For the latter see the discussion
before Lemma~\ref{quadratic} and recall that $E_u(0)={\rm Id}$.
So we have established that $D_{v^\nu}$ is surjective, and this
finishes the proof.
\end{proof}

\begin{lemma} \label{bubb}
There exists a universal constant $\hbar>0$ such that the following holds
for any sequence of Floer trajectories $v^\nu\in\widehat\M_{\delta^\nu}(x^+,x^-)$ 
with $\delta^\nu\to 0$.
If for some $s\in\R$
$$ 
\liminf_{\nu \to \infty} 
\bigl( \| dv_{02}^\nu \|_{L^\infty(B_\eps(s,0))} 
+\| d\hat v^\nu \|_{L^\infty(B_\eps(s,0))} \bigr)
= \infty \qquad\forall\eps>0 ,
$$
then there exists a sequence $\eps^\nu\to 0$ such that
$$ 
\liminf_{\nu \to \infty} \left(
\int_{B_{\eps^\nu}(s,0)} | dv^\nu_{02} |^2  
+ \int_{B_{\eps^\nu}(s,0)} | d\hat v^\nu |^2 \right)
\geq \hbar .
$$
Here $B_\eps(s,0)$ is the $\eps$-ball in $\R\times[0,1]$
or $\R\times[0,\delta^\nu]$ respectively.
\end{lemma}

In the usual analysis of bubbling effects, one would prove this lemma by rescaling around 
points where the differentials blow up, identifying the limits with pseudoholomorphic spheres or disks, and hence obtaining an energy quantization constant $\hbar$ that is geometrically determined by the minimal nonzero energy of spheres or disks.
In the present case however, depending on the relative speed of blow-up and strip-shrinking $\delta^\nu\to 0$, the rescaling may lead to
sphere bubbles in $M_0$, $M_1$, or $M_2$,
disk bubbles in $(M_0\times M_1,L_{01})$, $(M_1\times M_2,L_{12})$,
or $(M_0\times M_2,L_{01}\circ L_{12})$,
or the novel {\em figure eight bubble} described in the introduction.
Since we do not have a geometric bound on the minimal energy of figure eight bubbles, we use a mean value inequality to obtain $\hbar$ by purely analytic methods.

\begin{proof}[Proof of Lemma \ref{bubb}]
For notational convenience we introduce the noncontinuous function
$|dv|:\R\times[0,1]\to[0,\infty)$ given by
$|dv(s,t)|^2= |d v_{02}(s,t)|^2 + |d\hat v(s,t)|^2$ for $t\in[0,\delta]$
and $|dv(s,t)|=|d v_{02}(s,t)|$ for $t\in(\delta,1]$.

Suppose the lemma is false, that is, for every $k\in\N$
there exists a sequence $v^{k,\nu}\in\widehat\M_{\delta^{k,\nu}}(x^+,x^-)$ 
with $\delta^{k,\nu}\to 0$ such that (after time shift to $s=0$)
$R^\nu_k:=| dv^{k,\nu}(s^\nu_k,t^\nu_k)|\to\infty$
for some $(s^\nu_k,t^\nu_k)\to(0,0)$, but
$$
\liminf_{\nu \to \infty} 
\int_{B_{\eps^\nu}(0)} | dv^{\nu,k} |^2  \leq \frac 1k .
$$
for every sequence $\eps^\nu\to 0$.
In particular, this will hold for a fixed sequence $\eps^\nu_k\to 0$
that satisfies in addition $\eps^\nu_k\geq \delta^\nu_k$,
$(s^\nu_k,t^\nu_k)\in B_{\frac 14\eps^\nu_k}(0)$ 
and $\eps^\nu_k R^\nu_k\to\infty$.
We can then find diagonal sequences $v^k\in\widehat\M_{\delta_k}(x^+,x^-)$ 
with $\delta_k\to 0$, and $\eps_k\to 0$, $(s_k,t_k)\in B_{\frac 14\eps_k}(0)$
such that $\eps_k R_k:=\eps_k|dv^k(s_k,t_k)|\to\infty$ 
and 
\begin{equation}\label{no energy} 
\int_{B_{\eps_k}(0)} | dv^k |^2  \to 0 .
\end{equation}
Next, we use Lemma~\ref{hofer} to refine the choice of the blowup 
points $(s_k,t_k)$.
For that purpose we consider the spaces $X_{02}=\R \times [0,1]$, 
$\hat X = \R \times [0,\delta_k]$, and $X=\R\times [0,1]$, 
with the obvious inclusion $\pi:X_{02}\cup \hat X\to X$.
Using the function $f=|dv^k_{02}|$ on $X_{02}$ and $f=|d\hat v^k|$ on $\hat X$
one can then vary the point $\pi(x)=(s_k,t_k)\in\R\times[0,1]$ by 
$2\rho=\frac 14\eps_k$ to find $(s_k,t_k)\in B_{\frac 12\eps_k}(0)$
and $\eps'_k\leq\frac 18\eps_k$, 
such that $\eps'_k R_k:=\eps'_k|dv^k(s_k,t_k)|\to\infty$ 
and $|dv^k|\leq 4R_k$ on $B_{\eps'_k}(s_k,t_k)$.
Here (\ref{no energy}) continues to hold on 
$B_{\eps_k(0)}\supset B_{\eps'_k(s_k,t_k)}$.

Now in a first step we will prove that figure eight bubbles (arising
from rescaling in the case $\delta_k R_k\to\Delta\in(0,\infty)$) have
a minimal energy (possibly depending on $\Delta>0$.)  More precisely,
we claim that (\ref{no energy}) implies
\begin{equation}\label{tR}
t_kR_k\to 0, \qquad\text{and}\qquad \delta_k R_k\to 0 .
\end{equation}
In a second step we will then see that this gives rise to a disk
bubble in $(M_0\times M_2,L_{01}\circ L_{12})$.

\medskip
\noindent
{\bf Step 1:}{\it We prove (\ref{tR}).}

First consider the case $|dv_{02}^k(s_k,t_k)|\ge \frac 12|dv^k(s_k,t_k)|$
and $t_k\geq \frac 12\delta_k$.
Then for all sufficiently large $k$ we can apply 
the mean value inequality\footnote{
For this and the following applications of mean value inequalities note that they continue to hold with uniform constants for noncompact symplectic manifolds, if one has uniform bounds on the curvature and up to second derivatives of the almost complex structures $J_i$ w.r.t.\ $J_i$-compatible metrics.
} 
\cite[Lemma~4.3.1]{ms:jh} 
to $|dv_{02}^k|$ on the ball 
$B_{r_k}(s_k,t_k)\subset\R\times(0,1) \cap B_{\eps_k}(0)$
with $r_k:=\min\{t_k,\eps'_k\}$,
$$
\tfrac 14(r_k R_k)^2 \leq r_k^2 |dv_{02}^k(s_k,t_k)|^2 
\leq c \int_{B_{r_k}(s_k,t_k)} | dv_{02}^k |^2 \to 0 .
$$
Here we cannot have $r_k=\eps'_k$ since $\eps'_k R_k\to\infty$,
so we have $r_k=t_k$ 
and thus $\frac 12 \delta_k R_k\leq t_k R_k\to 0$ as claimed.

In the case $|d\hat v^k(s_k,t_k)|\ge \frac 12|dv^k(s_k,t_k)|$ 
and $\delta_k\geq t_k\geq\frac 12\delta_k$
we can apply the mean value inequality \cite[Theorem~1.3, Lemma~A.1]{we:en}
to $|d\hat v^k|$ with boundary condition
$\hat v^k|_{t=\delta_k}\in L_{01}\times L_{12}$ on the partial ball 
$B_{r_k(s_k,t_k)}\subset\R\times(0,\delta_k] \cap B_{\eps_k}(0)$
for $r_k:=\min\{\hh\delta_k,\eps'_k\}$,
$$
\tfrac 14 (r_k R_k)^2 \leq r_k^2 |d\hat v^k(s_k,t_k)|^2 
\leq c \int_{B_{r_k}(s_k,t_k)} | d\hat v^k |^2  \to 0 .
$$
As before we cannot have $r_k=\eps'_k$ since $\eps'_k R_k\to\infty$,
so we have $r_k=\hh\delta_k$ 
and thus $t_k R_k\leq \delta_k R_k\to 0$ as claimed.

In the remaining case $t_k\leq\frac 12\delta_k$
we consider the pseudoholomorphic map
$$
w^k:=(v_{02}^k,\hat v^k):\R\times[0,\delta_k]\to
M_0\times M_2 \times M_0\times M_2\times M_1 \times M_1 ,
$$
which satisfies the Lagrangian boundary condition 
$w^k|_{t=0}\in \Delta_{M_0\times M_2}\times \Delta_{M_1}$.
By the above we have $|dw^k(s_k,t_k)|\geq R_k\to\infty$ and
$\int_{B_{\eps_k}(0)} | dw^k |^2 \to 0$.
So for all sufficiently large $k$ we can apply 
the mean value inequality \cite[Theorem~1.3, Lemma~A.1]{we:en} 
on the partial ball 
$B_{r_k(s_k,t_k)}\subset\R\times[0,\delta_k) \cap B_{\eps_k}(0)$
for $r_k:=\min\{\hh\delta_k,\eps'_k\}$,
$$
(r_k R_k)^2 \leq
r_k^2 |dw^k(s_k,t_k)|^2 
\leq c \int_{B_{r_k}(s_k,t_k)} | dw^k |^2 \to 0 .
$$
Again we cannot have $r_k=\eps'_k$ since $\eps'_k R_k\to\infty$,
so we have $r_k=\hh\delta_k$ and thus $2t_k R_k\leq\delta_k R_k\to 0$
as claimed.

\medskip
\noindent
{\bf Step 2:} {\it We prove the lemma.}

We consider the rescaled maps
$w^k=(w^k_{02},\hat w^k)$, where
$w^k_{02}:B_{\eps_k R_k}(0)\cap{\mathbb H}^2\to M_0\times M_2$ 
is defined on half balls of radius $\eps_k R_k\to\infty$
in the half space ${\mathbb H}^2:=\R\times[0,\infty)$
by $w^k_{02}(s,t):=v^k_{02}(s_k+s/R^k,t/R^k)$, and 
$\hat w^k:B_{\eps_k R_k}(0)\cap(\R\times[0,\delta_k R_k])
\to M_0\times M_2\times M_1\times M_1$
is defined by $\hat w^k(s,t):=\hat v^k(s_k+s/R^k,t/R^k)$
on balls of radius $\eps_k R_k$ intersected with the 
strip of width $\delta_k R_k\to 0$.

This rescaling preserves the nontriviality $|d w^k(0,t_k R_k)|\geq 1$,
but on both domains $|d w^k|$ is uniformly bounded.
Hence we can find a subsequence of the $w_{02}^k$ that
converges in the $\cC^0$-topology on the unit half ball
$D_1:=B_1(0)\cap{\mathbb H}^2$.
The (scaling invariant) energy 
$\int_{B_{\eps_k R_k}(0)}|d w_{02}^k|^2$ 
converges to zero by (\ref{no energy}), so the limit 
has to be constant.
In fact, we have $w_{02}^k\to x_{02}\in L_{02}$
since the boundary values $w^k_{02}|_{t=0}$ 
converge to $L_{01} \circ L_{12}=L_{02}$ in $\cC^0([-1,1])$.
To see the latter use Lemma~\ref{lltrans}~(a) to bound the distance to $L_{02}$
by the distance $d( \hat w^k(s,0), (L_{01}\times L_{12})^T)$, which is zero for $t=0$ replaced by $t=\delta_k$. However, the bound on $|\pd_t\hat w^k|$ provides a bound
$d\bigl( \hat w^k(s,0) , \hat w^k(s,\delta_k)\bigr) \leq \delta_k 2R_k \to 0$ and thus proves $x_{02}\in L_{02}$.
This also proves that $\hat w^k\to x_1$ in 
$\cC^0([-1,1]\times[0,\delta_k R_k])$,
where $x_1\in M_1$ is uniquely determined by 
$\bar x:=(x_{02},x_1,x_1)\in L_{01}\times L_{12}$.
The maps $w_{02}^k$ are $\bar J_{02}$-holomorphic, so by elliptic
regularity the convergence $w_{02}^k\to x_{02}$ is in the
$\cC^\infty$-topology on every compact subset of ${\mathbb H}^2\setminus\pd{\mathbb H}^2$.
However, in order to obtain a contradiction to the fact that
$|d w^k(0,t_k R_k)|\geq 1$ with $t_k R_k\to 0$ we need to establish
$\cC^1$-convergence on $D_1$ up to the boundary.

We begin by noting that due to the $\cC^0$-convergence we can express
$w^k=e_x(\xi^k)$ in terms of 
sections $\xi^k=(\xi_{02}^k,\hat\xi^k)\in H^2(D_1,x_{02}^*T(M_0\times M_2))\times 
H^2([0,1]\times[0,\delta_k R_k],\bar x^*T(M_0\times M_2\times M_1\times M_1))$
using the exponential map centered at $x=(x_{02},\bar x)$.
These sections satisfy the diagonal and Lagrangian boundary conditions
$\xi^k|_{t=0}\in T_{x}(\Delta_{M_0\times M_2}\times\Delta_{M_1})$ and
$\hat\xi^k|_{t=\delta_k R_k}\in T_{\bar x}(L_{01}\times L_{12})$,
the $\cC^0$-convergence $\|\xi^k\|_\infty\to 0$, and a uniform bound
$\|\nabla\xi^k\|_\infty\leq c_0$.
Since $\ol{\partial}_J w^k=0$ and $\nabla x=0$ we obtain from (\ref{dbarxi}) 
$$
\nabla_s\xi^k + J(\xi^k)\nabla_t\xi^k = 0 .
$$
Now $d w^k=de_x(\xi^k)\nabla_s\xi^k ds + de_x(\xi^k)J(\xi^k)\nabla_s\xi^k dt$,
so it suffices to prove the $\cC^0$-convergence of $\nabla_s\xi^k$ near $0$.
For that purpose we multiply the sections by cutoff functions $h=(h_{02},\hat h)$
with $h_{02}:\R\times[0,1]\to[0,1]$ supported in $D_1$,
$\hat h:\R\to[0,1]$ supported in $[-1,1]$, and both equal to $1$ near $0$.
Then we obtain sections on the multistrip
$h\xi^k:=(h_{02}\xi_{02}^k,\hat h\hat\xi^k)\in \Gamma_{1,\delta_k R_k}$
that also satisfy the boundary condition $h_{02}\xi_{02}^k|_{t=1}=0$.
These satisfy a uniform bound
\begin{align*}
 \sup_k \Bigl(
\|\nabla_s(h\xi^k) + J(\xi^k)\nabla_t(h\xi^k)\|_{H^1_{1,\delta_k R_k}}
+ \| h\xi^k \|_{H^0_{1,\delta_k R_k}}  \Bigr)
&\leq \sup_k C \|\xi^k\|_{H^1_{1,\delta_k R_k}(\on{supp}(h))} 
 <\infty
\end{align*}
due to the bounds on $\|\xi^k\|_\infty$ and $\|\nabla\xi^k\|_\infty$ 
and the compact support of $h$.
From this Lemma~\ref{Ddelta}~(b) provides a uniform bound
$$
\sup_k  \|h\xi^k\|_{H^2_{1,\delta_k R_k}} \leq C_{\Gamma} < \infty .
$$
Indeed, the boundary terms vanish since the constant boundary conditions 
directly transfer to the derivatives,
$\nabla_s\xi_{02}^k|_{t=1},\nabla_s^2\xi_{02}^k|_{t=1}\in T_{x_{02}}(L_0\times L_2)$
and
$\nabla_s\hx^k|_{t=\delta_k R_k},\nabla_s^2\hx^k|_{t=\delta_k R_k}\in T_{\hat x}(L_{01}\times L_{12})$.

We now fix a pair of cutoff functions $h'$ with support in $h^{-1}(1)$
and still equal to $1$ near~$0$.
Then we apply Lemma~\ref{Ddelta}~(b) to $h' \nabla_s \xi^k$,
again with vanishing boundary terms, to obtain
\begin{align*}
\sup_k \| h' \nabla_s\xi^k \|_{H^2_{1,\delta_kR_k}} 
&\leq \sup_k C_1 \Bigl( 
\bigl\| \bigl( \nabla_s + J(\xi^k)\nabla_t \bigr) h'\nabla_s\xi^k
\bigr\|_{H^1_{1,\delta_k R_k}}
+ \| h' \nabla_s\xi^k \|_{H^0_{1,\delta_k R_k}} \Bigl) \\
&\leq \sup_k C (1 + c_0) \| h\xi^k\|_{H^2_{1,\delta_k R_k}}  < \infty .
\end{align*}
We can pick the cutoff functions such that $h'_{02}|_{D_{1/2}}\equiv 1$
on the half ball $D_{1/2}\subset{\mathbb H}^2$ and 
$\hat h|_{[-\frac 12, \frac 12]}\equiv 1$.
Then the compact Sobolev embedding 
$H^2(D_{1/2})\hookrightarrow\cC^0(D_{1/2})$
provides $\cC^0$-convergence of a subsequence $\nabla_s\xi^k_{02}$.
We already know that the limit is $0$, so we obtain
$\nabla_s\xi^k_{02}\to 0$ and $\pd_s w^k_{02}\to 0$ in $\cC^0(D_{1/2})$.
It remains to establish
$\|\nabla_s\hat\xi^k\|_{\cC^0([-\frac 12,\frac 12]\times[0,\delta_k R_k])}
\to 0$ and thus
$\|\pd_s \hat w^k\|_{\cC^0([-\frac 12,\frac 12]\times[0,\delta_k R_k])}\to 0$ 
in contradiction to $|d w^k(0,t_k R_k)|\geq 1$ with $t_k R_k\to 0$.
To see this we follow the argument in Lemma~\ref{Sobolev}.
Using the standard Sobolev embedding 
$H^1([-\frac 12,\frac 12])\hookrightarrow\cC^0([-\frac 12,\frac 12])$
we obtain for all $t_0\in[0,\delta_k R_k]$
\begin{align}\label{hatxik} \nonumber
\tfrac 1C \| \nabla_s\hat\xi^k|_{t=t_0} - \nabla_s\hat\xi^k|_{t=\delta_k R_k} 
\|_{\cC^0([-\frac 12,\frac 12])}^2
&\leq
\| \nabla_s\hat\xi^k|_{t=t_0} - \nabla_s\hat\xi^k|_{t=\delta_k R_k} 
\|_{H^1([-\frac 12,\frac 12])}^2 \\
&\leq \delta_k R_k \int_0^{\delta_k R_k} 
\|\nabla_t\nabla_s\hat\xi^k\|_{H^1([-\frac 12,\frac 12])}^2 \\
&\leq \delta_k R_k 
\|\nabla_s\hat\xi^k\|_{H^2([-\frac 12,\frac 12]\times[0,\delta_k R_k])}^2
\to 0 . \nonumber
\end{align}
From the above we moreover have
$ \| \nabla_s\xi'^k_{02}|_{t=0} \|_{\cC^0([-\frac 12,\frac 12])}
=\| \nabla_s\xi^k_{02}|_{t=0} \|_{\cC^0([-\frac 12,\frac 12])} \to 0$.
Now using Lemma~\ref{lltrans} and the boundary conditions, in particular
$(\xi_1^k-\xi_1'^k)|_{t=0}=0$, we obtain
\begin{align*}
&\|\nabla_s\hat\xi^k|_{t=\delta_k R_k}\|_{\cC^0([-\frac 12,\frac 12])} \\
&\leq C \bigl(\Vert \pi_{02}(\nabla_s\hx^k)|_{t=\delta_k R_k}\Vert_{\cC^0([-\frac 12,\frac 12])}
+ \Vert \nabla_s(\xi_1^k-\xi_1'^k)|_{t=\delta_k R_k} 
\Vert_{\cC^0([-\frac 12,\frac 12])} \bigr) \\
&\leq C \bigl(\Vert \nabla_s\xi'^k_{02}|_{t=0} \Vert_{\cC^0([-\frac 12,\frac 12])}
+ 3 \Vert \nabla_s\hat\xi^k|_{t=\delta_k R_k} - \nabla_s \hat\xi^k|_{t=0} 
\Vert_{\cC^0([-\frac 12,\frac 12])}
\bigr) 
\to 0 .
\end{align*}
Combining 
$\|\nabla_s\hat\xi^k|_{t=\delta_k R_k}\|_{\cC^0([-\frac 12,\frac 12])}\to 0$
with (\ref{hatxik}) then proves
$\|\nabla_s\hat\xi^k
\|_{\cC^0([-\frac 12,\frac 12]\times[0,\delta_k R_k])}\to 0$
and thus $|d w^k(0,t_k R_k)|\to 0$ in contradiction to the assumption.
\end{proof}

\begin{lemma}  \label{hofer}
Let $(X,d)$ be a metric space, $X_1,\ldots, X_n$ topological spaces,
$\pi: X_1 \cup \ldots \cup X_n \to X$ a continuous map, and $f: X_1
\cup \ldots X_n \to \R$ a non-negative continuous function.  
Fix $x\in X_i$ for some $i = 1,\ldots, n$ and $\rho > 0$.  
Suppose that $\pi^{-1}(B_{2\rho}(\pi(x))) \cap X_i$ is complete for each
$i=1,\ldots, n$.  Then there exists an $x' \in X_1 \cup \ldots X_n$
and a positive number $\rho' \leq \rho$ such that
$$ d(\pi(x'),\pi(x)) < 2 \rho, \ \ \sup_{\pi^{-1} B_{\rho'}(\pi(x'))}
f \leq 2f(x'), \ \ \ \rho' f(x') \ge \rho f(x) .$$
\end{lemma}

\begin{proof}  Otherwise, the same argument as in the proof of Hofer's
lemma \cite[p.93]{ms:jh} shows that there exists a sequence $x_\alpha
\in X_1 \cup \ldots \cup X_n$ such that
$$ x_0 = x, \ \ d(\pi(x_\alpha),\pi(x_{\alpha+1})) \leq \rho/2^\alpha, \ \
f(x_{\alpha+1}) > 2 f(x_\alpha) .$$
After passing to a subsequence, we obtain a Cauchy sequence $x_\alpha$ in
some $X_i$ with $f(x_\alpha) \to \infty$, which
contradicts completeness
and continuity of $f$.
\end{proof}  

\def\cprime{$'$} \def\cprime{$'$} \def\cprime{$'$} \def\cprime{$'$}
  \def\cprime{$'$} \def\cprime{$'$}
  \def\polhk#1{\setbox0=\hbox{#1}{\ooalign{\hidewidth
  \lower1.5ex\hbox{`}\hidewidth\crcr\unhbox0}}} \def\cprime{$'$}
  \def\cprime{$'$}

\end{document}